\documentclass[12pt]{amsart}
\usepackage[active]{srcltx}
\usepackage{amssymb,amsmath,amsthm,graphicx,tikz,xcolor}
\usepackage{array}% http://ctan.org/pkg/array
\usepackage{multicol,multirow,colortbl}
\tikzstyle{every picture}+=[remember picture]
\usepackage[charter]{mathdesign}
\definecolor{lightergray}{rgb}{0.65,0.65,0.65}
\definecolor{lightestgray}{rgb}{0.95,0.95,0.95}
\definecolor{refkey}{gray}{.75}
\usepackage[font=small,labelfont=bf,tableposition=top]{caption}

\DeclareCaptionLabelFormat{andtable}{#1~#2  \&  \tablename~\thetable} 
\DeclareCaptionLabelFormat{table}{\tablename~#2} 
\DeclareCaptionLabelFormat{figure}{\figurename~#2} 
 
\def\e{{\rm e}}
\def\cic{\boldsymbol}
\def\eps{\varepsilon}

\def\d{{\rm d}}
\def\Calpha{C_{\vec \alpha}}

\def\R {\mathbb{R}}

\def\supp{{\mathrm{supp}}}
\def\G{{\mathcal G}}

\def\C {{\mathcal C}}
\def\D {{\mathcal D}}

\def\J {{\mathbb J}}
\def\I {{\mathcal I}}

\def\F {{\mathcal F}}

\def\S{{\mathbf S}}
\def\T{{\mathbf{T}}}

\def\size{{\mathrm{size}}}
\def\Q {{\mathbf{Q}}}
\def\tr {{\mathsf \Lambda}}
\def\M {{\mathsf M}}

\def \l {\langle}
\def \r {\rangle}

\def \and{\qquad\text{and}\qquad}

\def \no#1#2#3 {{\bf #1} (#3), #2.}
\def \eds#1#2#3 {#1, #2, #3.}
\newtheorem{theorem}{Theorem}
\newtheorem{corollary}[theorem]{Corollary}
\newtheorem{lemma}{Lemma}
\newtheorem{proposition}[lemma]{Proposition}
\newtheorem*{proposition*}{Proposition}
\newtheorem*{theorem1s}{Theorem 1'}
\newtheorem*{theorem3s}{Theorem 3'}
\newtheorem*{theorem2s}{Theorem 2'}
\theoremstyle{definition}

\newtheorem{remark}[lemma]{Remark}
\newtheorem{conjecture}{Conjecture}

\numberwithin{lemma}{section}
\numberwithin{equation}{section}

\title[Endpoint bounds for the bilinear Hilbert transform]{Endpoint bounds for the bilinear Hilbert transform}
\author[F.\ Di Plinio]{Francesco  Di Plinio}
\address{  Dipartimento di Matematica,  Universit\`a degli Studi di Roma ``Tor Vergata'', \newline  \indent Via della Ricerca Scientifica,   00133 Roma,  Italy   \newline \indent \centerline{and}   \vskip1mm \indent 
 The Institute for Scientific Computing and Applied Mathematics,
Indiana University
\newline\indent
831 East Third Street, Bloomington, IN  47405, U.S.A. }
\email{diplinio@mat.uniroma2.it  \textrm{(F.\ Di Plinio)} }
\author[C.\ Thiele]{Christoph Thiele}
\address{Hausdorff Institute for Mathematics,  	Universit\"at Bonn 
 \newline\indent	Endenicher Allee 60,	D - 53115 Bonn, Germany
 \newline \indent \centerline{and} \vskip1mm  \indent
 Department of Mathematics, UCLA, Los Angeles, CA 90095-1555
  }
 \email{thiele@math.uni-bonn.de \textrm{(C.\ Thiele)}}
\thanks{  The  first author is an INdAM - Cofund Marie Curie Fellow and is  partially
supported by the National Science Foundation under the grant
   NSF-DMS-1206438, and by the Research Fund of Indiana University. The second author is partially supported by
the grant NSF-DMS-1001535}

\begin{document}  
 \setlength{\extrarowheight}{7pt}
 
 \begin{abstract}We study the behavior  of the bilinear Hilbert transform 
  $\mathrm{BHT}$  at the boundary of the known boundedness region $\mathcal H$.
    A sample of our results   is  the estimate 
\begin{equation*}
 \label{ab1} 
|\l \mathrm{BHT}(f_1,f_2),f_3 \r | \leq \textstyle C  |F_1|^{\frac34}|F_2| ^{\frac34} |F_3|^{-\frac12}   \log\log \Big(\e^\e +\textstyle \frac{|F_3|}{\min\{|F_1|,|F_2|\}} \Big) \ ,
\end{equation*}
valid for all  tuples of sets $F_j\subset \R $ of finite measure and functions $f_j$ such that $|f_j| \leq \cic{1}_{F_j}$, $j=1,2,3$, with the additional restriction that $f_3$ be supported on a major subset $F_3'$ of $F_3$ that depends on $\{F_j:j=1,2,3\}$. The use of subindicator functions in this fashion is standard in the given context, see \cite{MTT}. The double logarithmic term improves over the single logarithmic term obtained in \cite{BG}. Whether the double logarithmic term can be removed entirely, as is the case for the quartile operator discussed in \cite{DD2}, remains open.
%  The above estimate is  optimal, up to the doubly logarithmic term,   for a generic modulation-invariant bilinear multiplier operator, and improves on the single logarithmic correction obtained in   \cite{BG}. 

We employ our endpoint results to describe the blow-up rate of weak-type and strong-type estimates for $\mathrm{BHT}$ as the tuple $\vec \alpha$ approaches the boundary of $\mathcal H$.
 We also discuss   bounds on Lorentz-Orlicz spaces near $L^{\frac23}$, improving on results of \cite{CGMS}.
 The main technical novelty in our article is an enhanced version of the 
multi-frequency Calder\'on-Zygmund decomposition of  \cite{NOT}.
 \end{abstract}

\maketitle

\section{Introduction and Main results}
Recall that the classical Hilbert transform is bounded
in $L^p$ for $1<p<\infty$. At the endpoint $p=1$ one
has several types of estimates such as Hardy space estimates 
or Lorentz-Orlicz space estimates. Most relevant for
our discussion is the classical weak-type bound 
in $L^1$.  The language of generalized restricted type estimates allows
to formulate a corresponding dual estimate at $L^\infty$, and the two 
endpoint estimates suffice to recover $L^p$ bounds by interpolation. 

Somewhat analogously, it was shown  in \cite{KS} that the Coifman-Meyer 
bilinear singular integrals  
$$
T(f_1,f_2) (x) =  \mathrm{p.v.} \int_{\R^2} f_1(x-t_1) f_2(x-  t_2) K(t_1,t_2)  \, \d t_1 \d t_2  \qquad x \in \R\ ,
$$
where  $K$ is a homogeneous Calder\'on-Zygmund  kernel in $\R^2$,
%(i.e.\ $K$ is homogeneousof degree $-2$,  smooth away from the origin, and has mean value 0 on the unit sphere in $\R^2$)
obey the weak endpoint bound $T: L^{1}(\R) \times L^{1}(\R) \to L^{\frac12,\infty}(\R)   $. In the language of generalized restricted type, this is 
one of a triple of symmetric endpoint estimates which one may interpolate  to 
obtain    bounds in the entire allowed region for $L^p$ estimates established 
by Coifman and Meyer in \cite{CM1},  see also \cite{GT} for an extension 
to the non-homogeneous case.

This article is concerned with endpoint bounds for the more singular family of  bilinear operators known as bilinear Hilbert transforms.
%defined below in \eqref{BHTdef},
Such endpoint bounds have been previously investigated in \cite{BG,BG2,CGMS,DD2}. The region of known $L^p$ estimates for a bilinear Hilbert transform 
constitutes an open hexagon depicted in Figure \ref{hexafig}. The six extremal points
of the hexagon are all symmetric in the language of generalized restricted
type estimates and each of them corresponds to a hypothetic estimate 
$L^1\times L^2\to L^{2/3}$ for some dual of the bilinear Hilbert transform.
The shape of this region already suggests that the bilinear Hilbert transform has a more 
colorful endpoint theory than the bilinear Coifman-Meyer operators discussed above, whose open region of boundedness is the entire open triangle depicted in
the figure.

The additional thresholds provided by the short sides of the hexagon,
which correspond to hypothetic estimates of the bilinear Hilbert transform mapping into $L^p$
with $p=\frac23$, are an important structural  feature of  
modulation-invariant singular integrals. 
%At the present time, it is not even  known whether estimates of the type 
%\begin{equation} \label{rwtforeword}
%\|\mathrm{BHT}(\cic{1}_{F_1}, \cic{1}_{F_2})\|_{p,\infty} \leq C  \|\cic{1}_{F_%1}\|_{2p} \|\cic{1}_{F_2}\|_{2p}
%\end{equation}
%may hold for $p= \frac23$, but the analogue estimate for the discretized versio%n of $\mathrm{BHT}$ employed in the proofs of \cite{LT1,LT2} is known to fail w%henever $p<2/3$.
While it is not known whether or not bounds for the bilinear Hilbert transform can be extended past this threshold, such an extension does not hold for bilinear operators very closely related to the bilinear Hilbert transform, such as those obtained from allowing bounded coefficients in
model sums representing the bilinear Hilbert transform, as explained in \cite{Lac}. Using the same effect, it is possible to construct a trilinear modulation-invariant multiplier form 
which satisfies no bounds beyond these thresholds 
({C.\ Muscalu, personal communication}).
This motivates the study of endpoint estimates for the bilinear
Hilbert transform at the boundary of the hexagon. 
Similar endpoint questions for Carleson's  operator, the other stalwart of time-frequency analysis,  have enjoyed some popularity as well in recent years \cite{DP,DP2,DL,LIE2,LIE1}.

 Consider the  family of trilinear forms with parameter  $\vec \beta  \in  \R^3$   defined, for   Schwartz functions  $f_1,f_2,f_3:\R \to \mathbb C$, by the principal value integral
\begin{equation} \label{trdef}
\Lambda_{\vec\beta} (f_1,f_2,f_3)  = \int_\R \mathrm{p.v.} \int_{\R} f_1(x-\beta_ 1 t) f_2(x-\beta_2 t) f_3 (x-\beta_3 t) \, \frac{\d t}{t}\; \d x. \qquad 
\end{equation}
 By scaling and translation invariance we can restrict to   vectors $\vec \beta$ of unit length and perpendicular to $(1,1,1)$. In effect this reduces 
$\Lambda_{\vec\beta}$ to a  one-parameter family. The trilinear forms $\Lambda_{\vec\beta}$ arise as   duals to the family of bilinear   operators known as \emph{bilinear Hilbert transforms}, written in singular integral form as
\begin{equation*}
\mathrm{BHT}_{\vec b}(f_1,f_2)(x)=\mathrm{p.v.} \int_{\R} f_1(x-b_ 1 t) f_2(x-b_2 t)  \, \frac{\d t}{t}, \qquad x \in \R.
%\label{BHTdef}
\end{equation*}
Indeed,  $\Lambda_{\vec \beta}(f_1,f_2,f_3)=\l \mathrm{BHT}_{\vec b}(f_1,f_2), \overline{f_3}\r,$ with $\vec \beta$ and $\vec b$ related by $\beta_1-\beta_3=b_1,\beta_2-\beta_3=b_2 .$ 

In a pair of articles by Lacey and the second author \cite{LT1,LT2},   it is proved that
in the non-degenerate case, meaning no two components of $\vec \beta$ are equal,
\begin{equation}
\label{BHTnoform} \|\mathrm{BHT}_{\vec b}(f_1,f_2)\|_{\frac{p_1p_2}{p_1+p_2}} \leq {C_{\vec\beta}} C_{p_1,p_2} \|f_1\|_{p_1} \|f_2\|_{p_2}
\end{equation}
for all  $1<p_1,p_2\leq \infty$ with $2/3<\frac{p_1p_2}{p_1+p_2}<\infty$. These bounds can be obtained via the   interpolation  procedure described e.\ g.\ in \cite{MTT,ThWp}
as a consequence of 
the family of  \emph{generalized restricted weak-type} (GRWT) estimates 
\begin{equation} \label{rwtintro}
\big|\Lambda_{\vec\beta}(f_1,f_2,f_3)\big| \leq {C_{\vec\beta}} \Calpha  \prod_{j=1}^3 |F_j|^{\alpha_j}, \qquad \forall \vec \alpha \in \mathrm{int}\, \mathcal{H},  \end{equation}
where 
\begin{equation}
\label{tuples}
\mathcal H =\Big\{\vec \alpha = (\alpha_1,\alpha_2,\alpha_3): \alpha_1+ \alpha_2+ \alpha_3=1,\; \max_{j} \alpha_j \leq 1,\; \min_{j} \alpha_j  \geq -\textstyle\frac12\Big\}\end{equation}
is  the shaded hexagon  in Figure \ref{hexafig}.
\setlength{\unitlength}{0.85mm}
 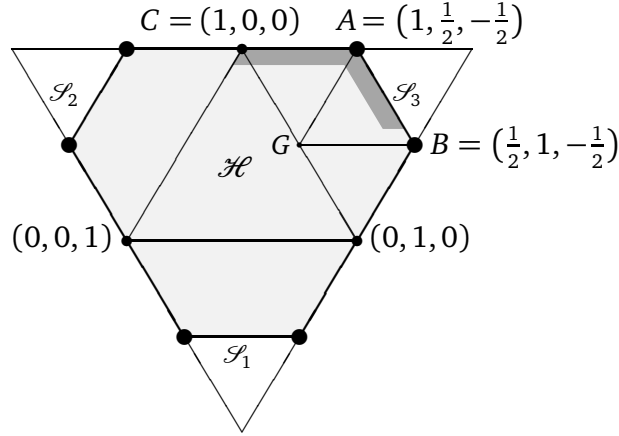
\begin{figure}[!t] 
  \hspace{-4em}
 \begin{picture}(72,78)
%\put(46,40){\circle*{0.8}}
\put(28,60){\tikz \coordinate (point1);}
    \put(64,60){\tikz \coordinate (point2);}
    \put(73,45){\tikz \coordinate (point3);}
    \put(55,15){\tikz \coordinate (point4);}
    \put(37,15){\tikz \coordinate (point5);}
    \put(19,45){\tikz \coordinate (point6);}
    \put(0,0){\tikz[overlay] \fill[lightestgray] (point1) -- (point2)%
             -- (point3) -- (point4) -- (point5) -- (point6) -- cycle;}
%             \put(44.5,57.5){\line(1,0){17.5}}
%\put(62,57.5){\line(3,-5){6}}
%\put(68,47.5){\line(1,0){3}}
\put(46,60){\tikz \coordinate (point7);}
    \put(44.5,57.5){\tikz \coordinate (point8);}
    \put(62,57.5){\tikz \coordinate (point9);}
    \put(68,47.5){\tikz \coordinate (point10);}
    \put(71.5,47.5){\tikz \coordinate (point11);}
    \put(64,60){\tikz \coordinate (point12);}
    \put(0,0){\tikz[overlay] \fill[lightergray] (point7) -- (point8)%
             -- (point9) -- (point10) -- (point11) -- (point12) -- cycle;}
\put(46,60){\circle*{1.6}}
\put(28,30){\circle*{1.6}}
\put(64,30){\circle*{1.6}}
\put(42.1,40){$\mathcal H$}
%\put(46,37){\circle*{0.8}}
\put(19,45){\circle*{0.8}}
\put(28,60){\circle*{0.8}}
\put(64,60){\circle*{0.8}}
\put(73,45){\circle*{0.8}}
\put(55,45){\circle*{0.8}}
%\put(50,42){ $\scriptstyle G=(\frac12,\frac 12,0)$}
\put(50.5,43.5){\small $ G$}
\put(37,15){\circle*{0.8}}
\put(55,15){\circle*{0.8}}
 
%\put(37,45){\circle*{0.8}}
%\put(55,45){\circle*{0.8}}

\put(10,29){$(0,0,1)$}
\put(66,29){$(0,1,0)$}
\put(30,63){$C=(1,0,0)$}
\put(61,63){$A=\big(1,\frac 12,-\frac 12\big)$}
\put(75.5,43.5){$ B=\big(\frac 12, 1,-\frac 12\big)$}
\color{black}
\thinlines
\put(10,60){\line(1,0){18}}
\put(82,60){\line(-1,0){18}}
\put(10,60){\line(3,-5){9}}
\put(46,0){\line(-3,5){9}}
\put(46,0){\line(3,5){9}}
\put(82,60){\line(-3,-5){9}}
\thicklines
\put(28,60){\line(1,0){18}}
\put(28,30){\line(-3,5){9}}
\put(28,30){\line(3,-5){9}}
\put(64,30){\line(3,5){9}}
\put(64,30){\line(-3,-5){9}}
\put(46,60){\line(1,0){18}}
 
\put(64,60){\circle*{2.5}}
 
\put(69.5,52){\footnotesize$\mathcal S_3$}
\put(73,45){\circle*{2.5}}
\put(28,60){\circle*{2.5}}
\put(15.8,52){\footnotesize$\mathcal S_2$}
\put(19,45){\circle*{2.5}}
\put(37,15){\circle*{2.5}}
\put(55,15){\circle*{2.5}}
\put(43,11){\footnotesize$\mathcal S_1$}
\put(37,15){\line(1,0){18}}
\put(19,45){\line(3,5){9}}
 
 \put(73,45){\line(-3,5){9}}
 \thinlines
 
 \put(64,60){\line(-3,-5){9}}
 \put(73,45){\line(-1,0){18}}

\put(28.4,30){\line(1,0){35.2}}
\put(46.2,59.67){\line(3,-5){17.6}}
\put(28.2,30.33){\line(3,5){17.6}}

%\put(28,30){\line(2,3){100}}
\end{picture}
\caption{The hexagon $\mathcal H$ of \eqref{tuples}. The darker shade indicates the approach region of \eqref{geometry}. \label{hexafig}}
\end{figure}
% We remark that the constants $\Calpha$ provided by the proofs blow up as $\vec% \alpha$ approaches $\partial\mathcal H$. 
The diction GRWT stands for \eqref{rwtintro} holding for all tuples of sets $F_j \subset \R, \, j=1,2,3$ of finite measure and for all   functions $|f_j| \leq \cic{1}_{F_j}$ with the additional restriction that,  if $j_*$ is    the maximal index $j$ such that $\alpha_j=\min_k \{\alpha_k \}$,  $|f_{j_*}| \leq \cic{1}_{F_{j_*}'}$, for some subset  $F_{j_*}'\subset F_{j_*}$ which is major in the
sense $|F_{j_*}| \leq 2|F_{j_*}'|$ and which may depend on $F_1,F_2,F_3$.
Allowing the passage to a major subset is crucial if one of the parameters
$\alpha_j$ is less than or equal to $0$. Note that at every point of the open region discussed
at most one parameter is less than or equal to $0$ and this is the one with index $j_*$. Assuming for example that $j_*=3$, estimate \eqref{rwtintro} is 
equivalent to the weak-type  bound $$
\|\mathrm{BHT}_{\vec b}(f_1,f_2)\|_{\frac{1}{1-\alpha_3},\infty} \leq {C_{\vec\beta}} \Calpha \prod_{j=1,2}|F_j|^ {\alpha_j} , \qquad 
 \forall \, |f_j| \leq \cic{1}_{F_j}, \,j=1,2.$$
The symmetric nature of the form $\Lambda_{\vec \beta}$ and the notion of GRWT
shows that specializing $j_*=3$ is no loss of generality.

% Unlike    multilinear forms   of the Coifman-Meyer type    \cite{CM1,GT,KS}, t
%he  portion of $\partial \mathcal H$    $$ \mathcal S= \bigcup_{j=1}^3 \mathcal
% S_j, \qquad 
%{\mathcal S}_j=\Big\{\vec \alpha: \alpha_1+\alpha_2+\alpha_3=1, \; \textstyle\a
%lpha_j = -\frac12,\; \max\{ \alpha_k\} < 1\Big\}, \qquad j=1,2,3,
%$$
%   seems to be a  critical threshold  for \eqref{rwtintro}. Indeed,   the   ana
%logue of \eqref{rwtintro} for the model sums discretizing $\Lambda_{\vec \beta}
%$, which, in the present article, make their appearance in Section \ref{sect3},
%  fails whenever $\min\alpha_j<-\frac12$ \cite{LT1,Lac}. Albeit no such counter
%example has been exhibited for $\Lambda_{\vec\beta}$, 
%it is possible to construct a trilinear modulation-invariant multiplier form wh
%ich   satisfies no GRWT bounds in the region $\min\alpha_j<-\frac12$,  exploiti
%ng the same effect of \cite{Lac} ({C.\ Muscalu, personal communication}).

In the degenerate case, that is if two components of $\vec \beta$ are
equal, questions on bounds for the bilinear Hilbert transform trivialize.
The bilinear Hilbert transform then degenerates to a combination of the
classical Hilbert transform and a pointwise product. There are three
different degenerate cases, depending on which two components of
$\vec \beta$ are equal. In each case, the region of GRWT estimates  
is no longer symmetric under permutation of the three indices, and 
neither contains the above hexagon $\mathcal H$ nor it is contained in the hexagon. 
This leads to an interesting array of questions concerning uniformity of bounds for 
the nondegenerate case in the vicinity of the degenerate case. Such 
questions have been addressed for example in 
\cite{GL,Li,OT,ThUE}. The present paper focuses only on the non-degenerate case
and ignores the above uniformity questions 
by allowing constants $C_{\vec \beta}>0$ that depend on $\vec \beta$ in
an unspecified manner. In general the constants will blow up
as $\Delta_{\vec \beta}$, the distance from $\vec \beta$ to the union of the 
three hyperplanes  $\beta_j=\beta_k$, $k \neq j$, tends to $0$.

%
%Denoting by $\Delta_{\vec\beta}$ 
%the distance from $\vec \beta$ to the union of the three hyperplanes  $\beta_j=%\beta_k$, $k \neq j$, 
%the cases $\Delta_{\vec \beta}=0$ of $\mathrm{BHT}_{\vec b}$ reduce to a combin%ation of  pointwise products of the  identity operator and of the linear Hilber%t transform, and are thus termed \emph{degenerate}. From now on, we   discuss   the nondegenerate case $\Delta_{\vec \beta}>0$, and the constants  $ {C_{\vec\%beta}}>0$ appearing in our estimates will   depend on $\Delta_{\vec \beta}$. 

% The central object of study of this article is the extension of \eqref{BHTnoform}, \eqref{rwtintro} or of suitable  modifications thereof   to    tuples  $\vec \alpha \in \mathcal S$. 

A folklore conjecture  is that the generalized restricted weak-type estimate
\eqref{rwtintro}
for the non-degenerate bilinear Hilbert transform extends to the region $\mathcal S$
defined by
$$ \mathcal S= \bigcup_{j=1}^3 \mathcal S_j, \qquad 
\mathcal{S}_j=\Big\{\vec \alpha: \alpha_1+\alpha_2+\alpha_3=1, \; \textstyle\alpha_j = -\frac12,\; \max\{ \alpha_k\} < 1\Big\}, \qquad j=1,2,3,
$$
that is the union of the short open segments of the boundary of the hexagon.
\begin{conjecture} \label{conj1} The GRWT estimate \eqref{rwtintro} holds for  all tuples $\vec \alpha \in  {\mathcal S} $. 
 \end{conjecture}    

A main theme of the present paper is that additional insight can be obtained by 
lifting the subindicator condition on a careful choice of the functions 
$f_j$, that is to allow for general $L^p$ functions $f_j$ rather than
functions dominated by an indicator function. This is analogous to the 
classical case of the linear Hilbert transform discussed above, where the
crucial boundary estimate is a weak-type estimate which allows
the input function to be a general $L^1$ function, while the test function
that one pairs with to obtain a bilinear form has to be a subindicator function,
supported on a major subset as elaborated in the GRWT definition above.
In the case of the trilinear form $\Lambda_{\vec \beta}$ one has a choice
of three functions on which to lift the subindicator condition, yielding
a relatively more diverse set of possible estimates.
 
At the typical corner $A$ of the hexagon, the second coordinate 
$\alpha_2=1/2$ stands for the Hilbert space $L^2(\R)$. In the vicinity
of that corner it is therefore particularly 
efficient to lift the subindicator condition on the function $f_2$, 
since one has the full Hilbert space technique at hand.
This was already observed in \cite{DD2} for the quartile operator.

A side product of our  investigations is a fairly straightforward adaption of the strategy of \cite{DD2} to the present case of the bilinear Hilbert transform
to obtain the following endpoint estimate at the corner $A$ with a 
logarithmic correction term.
\begin{theorem} \label{propL1L2ag}
Let    $f_2 \in L^{2}(\R)$, and   sets $F_1, F_3 \subset \R$ of finite measure be given. Then, there exists a major subset  $F_3'$ of $F_3$, depending on $f_2, F_1, F_3$, such that      
\begin{equation*} %\label{propsL1L2inintro} 
\big|\Lambda_{\vec \beta}(f_1,f_2,f_3)\big| \leq C_{\vec \beta}   |F_1|  \|f_2\|_2 |F_3|^{-\frac{1}{2}} \log \big( \e + \textstyle \frac{|F_3|}{|F_1|} \big)\qquad \forall   |f_1| \leq \cic{1}_{F_2},\,
 |f_3| \leq \cic{1}_{F_3'}.
\end{equation*}
\end{theorem}
This is our only estimate directly at the corner A of the hexagon.
It is a strengthening of a result of \cite{BG,BG2}, where the same estimate is shown to 
hold under the further assumption that $f_2$ be a subindicator function as well. 
One may view the logarithmic correction term as a fallout
of being on the edge $\overline{AC}$, which corresponds to
the space $L^1$ for the first function.

Of course a symmetric estimate holds at the other six corners of the hexagon, 
and specializing again to three subindicator functions as in \cite{BG,BG2}
one obtains by 
interpolation GRWT estimates everywhere in the open hexagon. However, 
the interpolated estimates one obtains in this way are not as efficient 
in the vicinity of the boundary of the hexagon as what one obtains using 
the next two theorems.

To motivate the next theorem, consider the symmetric estimate to Theorem 
\ref{propL1L2ag} at corner $B$, which puts the function $f_1$ in $L^2$.
We would like to prove sharp estimates on the edge $\overline{AB}$.
There the function $f_1$ is in $L^p$ with $p$ between $1$ and $2$.
It is therefore natural to seek a Calder\'on-Zygmund decomposition of
the function $f_1$, using Hilbert space technique on the good portion 
and some additional localization information on the bad portion. The Calder\'on-Zygmund 
decomposition has to respect a number of frequencies as does the multi-frequency Calder\'on-Zygmund decomposition 
(MFCZ) developed in \cite{NOT}. A main point of the present paper is that in order to be successful on the edge of
 the hexagon we need a very sharp form of this MFCZ. Developing this MFCZ and applying it is the main technical 
advance of the present paper. 

%While the classical approach of \cite{LT2},
%we prove directly estimates on the trilinear form $\Lambda_{\vec \beta}(f_1,f_2%,f_3)$   with the subindicator restriction  lifted for exactly one $f_j \in L^{%1/{\alpha_j}}(\R)$ with $1/2\leq \alpha_j<1$.
%Our first main theorem addresses tuples $\vec \alpha$ on the open line segment %$\overline{AB}$ in Figure \ref{hexafig}.
\begin{theorem}
  \label{prLpL2}  We write $(t)_{*}=(1+t)(\log(\e+ t))^3$.  
   Let $\vec{\alpha}  \in \mathcal S_3$,  $f_1 \in L^{\frac{1}{\alpha_1}}(\R)$, and   sets $F_2, F_3 \subset \R$ of finite measure be given. Then, there exists a major subset  $F_3'$ of $F_3$,  depending on $f_1, F_2, F_3$, such that  the estimate   
\begin{equation} \label{propest} \textstyle
\big|\Lambda_{\vec \beta}(f_1,f_2,f_3)\big|  \leq     \frac{C_{\vec\beta}}{ (1-\alpha_1)(1-\alpha_2)} \|f_1\|_{\frac{1}{\alpha_1}}|F_2|^{\alpha_2} |F_3|^{-\frac12} \Big(\max\big\{ \textstyle\log\big(\frac{|F_3|}{|F_2|}\big), \frac{1}{1-\alpha_1}\big\}\Big)_*^{2\alpha_1-1 }  \end{equation}
holds for all $|f_2| \leq \cic{1}_{F_2}$, $|f_3| \leq \cic{1}_{F_3'} $. 
\end{theorem} 

This theorem is analogous to  \cite[Proposition 2.1]{DD2} for the quartile operator.
Thanks to perfect localization of Walsh wave packets, an even
sharper but trivial form of MFCZ is true in the discrete setting  and hence \cite{DD2} obtains an 
estimate without the starred correction term, which is in fact a stronger form of Conjecture \ref{conj1} for the quartile operator.
The exponent of the starred term tends to $0$ at the corner $B$
and to $1$ at the corner $A$, showing that the correction term caused by MFCZ 
becomes worse as one moves away from the Hilbert space. 
Theorem \ref{prLpL2} is a phenomenon on the open edge $\overline{AB}$; we do not
see how to obtain the theorem by interpolation from any estimates at 
the corners $A$ and $B$, in particular not by interpolation with Theorem \ref{propL1L2ag}.
The constant in Theorem \ref{prLpL2} blows up as we approach either
corner.

We return to Theorem \ref{propL1L2ag} at the corner $A$ as motivation for
the following theorem. We fix the exponent $\alpha_2=1/2$ and the general function
$f_2\in L^2$, which puts us on the bisecting line $\overline{AG}$. 
This time we lift a second subindicator condition, namely on the function $f_1$,
to obtain two unconstrained functions. On the edge $\overline{AG}$ the function 
$f_1$ is in $L^p$ with $1<p<2$, and one can apply again the MFCZ to this function.

% Our (direct) methods are unable to produce a version of Theorem \ref{propL1L2ag} where the function $f_1$ is unrestricted in $L^1(\R)$.  The next theorem does so for a tuple $\vec \alpha$ on the line segment $\overline{GA}$ of Figure \ref{hexafig}, corresponding to $\alpha_2=1/2$, providing the     explicit blowup rate of the bound as $\vec \alpha$ approaches the endpoint $A$.

\begin{theorem}\label{thelinfversion}
For all $0\leq \alpha_1<1$, $f_1 \in L^{\frac{1}{\alpha_1}}(\R), f_2 \in L^2(\R), F_3 \subset \R$, there exists a major subset $F_3'$ of $F_3$,   depending on $f_1, f_2, F_3$, such that
$$
\big|\Lambda_{\vec\beta}(f_1,f_2,f_3)\big| \leq C_{\vec \beta} \textstyle\frac{1}{1-\alpha_1}  \Big(\textstyle\frac{1}{1-\alpha_1}\Big)_*^{2\alpha_1-1} \|f_1\|_{\frac{1}{\alpha_1}}\|f_2\|_2 |F_3|^{\frac12-\alpha_1} \qquad \forall |f_3| \leq \cic{1}_{F_3'}.
$$
\end{theorem}

Similar estimates as in Theorem \ref{thelinfversion} with worse growth of the constant
as one approaches the corner $A$ can be obtained by standard interpolation
methods from Theorem \ref{propL1L2ag} and its symmetric counterparts.
Namely, observe that the estimate of Theorem \ref{propL1L2ag} is equivalent to  the bound 
\begin{equation}\label{ctedgeag}
\big|\Lambda_{\vec\beta}(f_1,f_2,f_3)\big| \leq   \frac{C_{\vec \beta}}{1-\alpha_1}    |F_1|^{\alpha_1} \|f_2\|_2 |F_3|^{\frac12-\alpha_1} 
\end{equation}
for functions $f_1,f_3$ restricted as in the statement of the theorem. Marcinkiewicz type interpolation 
as in\ Lemma \ref{interpolemma} 
deduces the same type of estimate as in Theorem \ref{thelinfversion} from \eqref{ctedgeag}, 
albeit with a blowup rate of $(1-\alpha_1)^{-5/2}$ as one approaches the corner $A$.
On the other hand, one notes that \eqref{ctedgeag} is stronger than what is obtained by specializing the 
estimate of Theorem \ref{thelinfversion} to subindicator functions $f_1$. 
Therefore, neither Theorem \ref{propL1L2ag} nor Theorem  \ref{thelinfversion} implies 
the other in full strength by the obvious deduction methods. 
Again, a sharper
analogue of Theorem \ref{thelinfversion} for the quartile operator has been proved
in \cite[Proposition 2.3]{DD2}, lacking the starred correction term thanks to the perfect discrete MFCZ.

Restricting Theorem \ref{prLpL2} to subindicator functions $f_1$
and interpolating with the symmetric version under interchanging 
the corners $A$ and $B$ yields the following punchline result.
\begin{corollary} \label{rwtTH} For all tuples   $\vec{\alpha}  \in \mathcal S_j$, $j=1,2,3$, we have the GRWT estimate
 \begin{equation} \label{propest2}
\big|\Lambda_{\vec \beta}(f_1,f_2,f_3)\big|  \leq   {C_{\vec\beta}}   \Big(   \prod_{k=1  }^3 |F_k|^{\alpha_ k}\Big) \textstyle	\max\bigg\{  \frac{1}{\min_{k\neq j}\{1-\alpha_k\}},  {\log\log}  \Big( \e^\e+    \frac{|F_j|}{\min_{k\neq j}|F_k|} \Big)\bigg\}.  \end{equation}
 \end{corollary}
The special case of this result at the midpoint of $\overline{AB}$ has been
highlighted in the abstract of this paper. This theorem is a weaker form of 
Conjecture \ref{conj1} by the double logarithmic correction term. 
This estimate cannot be obtained by interpolation of   Theorem \ref{propL1L2ag} 
and its symmetric counterparts, which only
yields the single logarithmic estimate that was observed in
\cite{BG,BG2}. 
This highlights again that Theorem \ref{prLpL2} encodes additional 
information relative to Theorem \ref{propL1L2ag}. Clarifying 
whether the double logarithmic term can be removed in the corollary 
is one of the more intriguing open questions on endpoint bounds for
the bilinear Hilbert transform.
Obviously we do not see how to do this with present technology.

%We supplement our discussion of the main results with several remarks.
We conclude this discussion with a few remarks on our strengthening of the multi-frequency 
Calder\'on Zygmund decomposition, Proposition \ref{corMFCZabs}. The bad portion of the MFCZ 
is the sum of functions $b_I$ localized to intervals $I$ and having mean zero with respect 
to a number $N$ of bad frequencies relevant on the interval $I$. The main issue lies with 
estimating the interaction of this bad function $b_I$ with wave packets which are frequency 
localized in a compact interval near a bad frequency, and which are spatially localized 
away from but not too far away from $I$. To make this interaction sufficiently small for our needs we work on the one hand
with wave packets which have better than mere Schwartz function decay. We use an optimal almost 
exponential decay following a construction by Ingham. To utilize this decay we have to 
prepare the bad function $b_I$ of the MFCZ to have mean zero not only against the dominant bad 
frequency, but also against approximately $\log(N)$ many equidistant frequencies 
in the vicinity of the dominating bad frequency. The price of all this is the appearance of the
extra terms $(\cdot)_*$ occurring in Theorems \ref{prLpL2} and \ref{thelinfversion}. This is in contrast
to the discrete analogues of \cite{DD2}, where of course one has wave packets which are compactly
supported both in frequency and in space, and the interaction terms in question are simply zero.
We stress that the use of almost exponential type wave packets has no precedent in the context
of time frequency analysis. It is unnecessary for deeply interior estimates in the open hexagon, 
but appears relevant for the sharp estimates at and near the boundary of the hexagon that we investigate.

It is our opinion that   Proposition \ref{corMFCZabs}, or variants thereof, could  be employed as well
in the translation to the continuous case of the arguments of \cite{DP,DP2,DL} on Carleson type operators and of \cite{OT} on uniform estimates for the family of Walsh models of the bilinear Hilbert transforms.%\label{remMFCZintro}

\subsection*{Outline of the article} \label{sstechofpf}

Sections \ref{sect2} and \ref{secMFCZ}   are concerned with the multi-frequency
Calder\'on Zygmund decomposition in general. 
Section \ref{sect2} contains technical preliminaries on functions with compact frequency support 
and almost exponential decay rate. Our sharp version of the multi-frequency Calder\'on-Zygmund 
decomposition is introduced in Section \ref{secMFCZ} and its properties are discussed, most notably in
Proposition \ref{corMFCZabs}. 

We then turn to the bilinear Hilbert transform.   
In Sections \ref{sect3} and \ref{sect4} we rephrase  the construction of
the model sums for $\Lambda_{\vec \beta}$ and some classical results of time-frequency analysis.  
In Section \ref{secforest} we  apply Proposition \ref{corMFCZabs} 
to obtain  an estimate   for the model sums of $\Lambda_{\vec \beta} $ restricted to a single forest with appropriate $L^1$ and $L^\infty$ bounds on the counting function.  The main steps of the proof of Theorems  \ref{propL1L2ag},
 \ref{prLpL2},  and \ref{thelinfversion}, as well as the proof of Corollary \ref{rwtTH}, are   given   in Section \ref{sectpfthm}. 

Finally, in Section \ref{sectintcor}, we present several corollaries of our main results, elaborating on alternative ways of
formulating the behaviour of the bilinear Hilbert transform near the boundary. The first group of corollaries is concerned with 
the blow-up rates of the eight possible 
types of estimates for $\mathrm{BHT}_{\vec b}$, corresponding to different choices of sets of unrestricted functions, as the exponents 
approach the boundary of the hexagon $\mathcal H$ in Figure \ref{hexafig}.
These estimates are summarized in Table \ref{tableside}. The second group of corollaries,  in the spirit of the article  \cite{CGMS},   
is devoted to the boundedness properties of  $\mathrm{BHT}_{\vec b}$ on  Lorentz-Orlicz spaces near H\"older tuples $\vec \alpha$ on the open 
segment $\overline{AB}$ and at the corner $A$. These corollaries are proved in Section \ref{sect2new}.

\subsection*{Notational remarks}   
%In this work, we do not attempt to track how our estimates depend on $\Delta_{\vec\beta} $. In short, our proofs rely on a   discretization of the forms $\Lambda_{\v%ec \beta}$ which is  worse conditioned as $\vec\beta$ approaches  the degenerate values, so that our ${C_{\vec\beta}}$'s       blow up as  $\Delta_{\vec\beta} \to 0$%. In what follows,  we work with a fixed, nondegenerate, choice of $\vec \beta$ and with  its permutations $ \pi_j(\vec \beta),\sigma_{j_1 j_2}(\vec \beta)$, which s%hare the same degeneracy constant  ${C_{\vec\beta}}$. Indeed, the arguments are essentially the same for all nondegenerate $ \vec \beta  $ if one is not concerned wi%th uniformity issues. The constant  ${C_{\vec\beta}}$ will be treated as absolute and  hidden by the almost inequality sign. We send the reader to the articles 
%\cite{GL,OT,ThUE} 
%and references therein for a discussion of the problem of uniform  estimates in $\Delta_{\vec\beta} $ for the bilinear Hilbert transform and its discrete analogues. 

The vector $\vec\beta$ is always non-degenerate and all explicit and implicit constants in this paper may depend on $\Delta_{\vec\beta} $,  
the distance from $\vec \beta$ to the degenerate case as discussed above.
Let $I\subset \R$ be an interval; $c(I)$ will denote the  midpoint of $I$ and, for $C>0$, by $CI$ we refer to the interval with center $c(I)$ and length $C|I|$; we also write $x + I$ for the interval $\{x+y: y \in I\}$. We set 
$$ 
\|f\|_{L^{p}(I)}:=  \left(\int_I  |f(x)|^p \textstyle \frac{\d x}{|I|} \right)^{\frac1p},\; 1\leq p< \infty, \qquad \|f\|_{L^{\infty}(I)}:= \mathrm{ess}\sup_{x \in I} |f(x)|.
$$
For $1\leq p<\infty $, the $p$-th Hardy-Littlewood maximal function is defined as
$$
\M_p f(x) = \sup_{I \ni x} \|f\|_{L^p(I)}.
$$
With   $\D$, we indicate a generic grid on $\R$, that is a collection of intervals  such that $I\cap I'\in \{I,I',\emptyset\}$ for each $I, I'\in \D$. We write $\D_0$ for the standard dyadic grid on $\R$, while the  notation $  \mathcal{D}(I)$ refers to the standard dyadic grid on an interval $I\subset \R$.
Finally, the  constants   $C>0 $,     as well as the  constants implied by the almost inequality sign $\lesssim $,    may vary at each occurrence without explicit mention, and are meant to be absolute, once  $\vec \beta$ has been fixed, unless otherwise specified.

\section{Rapidly decaying functions with compact frequency support} \label{sect2}

Throughout the article,  $u$  will be  a positive, increasing and convex function on $[0,\infty)$ satisfying the normalized  Osgood condition
\begin{equation}
\label{int1}
\int_{0}^\infty \frac{1}{  {u}(t)} \,\d t =1, \end{equation}
and such that  \begin{equation}
\label{tricco2}
B_u (\tau):=
\sup_{t\geq 0}\Big((1+|u(t)|) \e^{-\tau t}\Big) <\infty
\end{equation}
for all $\tau>0$.
Condition \eqref{tricco2} holds, for instance, when   $ u(t) \leq C(1+t)^C$ for some   $C>1$.  We  will use   the (evenly extended) inverse function of $u$ 
 $$U:\R \to [0,\infty), \qquad U(x)= \begin{cases} u^{-1}(|x|) & |x| \geq u(0), \\ 0 &|x| < u(0),  \end{cases}$$  
 which is increasing on $[0,\infty)$ and satisfies 
\begin{align}
\label{tricco3} & U(x) \leq |x|  \qquad \forall\,x \in \R, 
\\ &
\label{tricco1}
   U(\vartheta x) \geq \textstyle  \vartheta  U(   x) -u(0) \qquad \forall\,  x \in \R, \vartheta \in [0,1]. 
\end{align}  
The first estimate above is a consequence of \eqref{int1}, while \eqref{tricco1} follows from \eqref{tricco3} and concavity of $U$ on $\{x \geq u(0)\}$.

Significant  examples of functions $u$ as such  are given by the family
\begin{equation}
\label{familyu}
u_\lambda(t) =   \frac{1}{\lambda} (t+\e) \big(\log(t+\e)\big)^{1+\lambda}, \qquad \lambda>0.
\end{equation}

The upcoming Lemma \ref{HORM} is  a reformulation of a result of Ingham  \cite{ING}.  
   In words,    given any   $u$ satisfying the above assumptions, one obtains    a smooth function  $\upsilon $ with compact frequency support and  with exponential decay rate given by a constant times $U$.   
   \begin{lemma} \label{HORM} Let $u$ be as above. Then, there exists a   smooth nonnegative function ${{\upsilon}}: \R \to [0,+\infty)$   with the properties that
\begin{align}
& 
%\label{cpta}
%\mathrm{supp} \, \widehat {{\upsilon}} \subset \big[-\textstyle \frac12,  \frac12\big], \nonumber	\qquad \nonumber  \widehat \upsilon(\xi) >0\quad  \forall \xi \in \big(-\textstyle \frac12,  \frac12\big),
\label{averaging}\cic{1}_{[-\frac{1}{6}, \frac{1}{6}]}(\xi) \leq \widehat{\upsilon}(\xi) \leq \cic{1}_{[-\frac12,\frac12]}(\xi), \qquad \forall \xi \in \R, \\
&   \big|{{\upsilon}}^{(j)}(x)|	 \leq C_{D,u}   \e^{-\frac{1}{100}U( x)}    \qquad  \forall x \in \R, \, j \in\{0,\ldots, D\}. \label{decayA} 
\end{align} 
The  positive constants $C_{D,u}$ in \eqref{decayA}   depend   on  $D$ and on  $B_u(2D)$ from \eqref{tricco2}.
\end{lemma}  
\begin{proof}  Below, the constant $C_D>0$ depends only on $D$ and may vary from line to line.
 Consider the sequence of functions $v_k:\R \to [0,\infty)$, defined by the recurrence
$$
v_1:= {u}(1)\cic{1}_{[0,({u}(1))^{-1}]} , \qquad v_{k}:= v_{k-1} *\big({u}(k)\cic{1}_{[0,({u}(k)^{-1}]} \big), \quad k >1.
$$
It is easy to see that 
$
\int v_k =1,  
$
that $\sup v_k \leq \sup v_1 = {u}(1)$, that $v_k \in \C^{k-2}(\R)$, $k \geq 2$, and finally that $\mathrm{supp} \,v_k \subset \textstyle  \big[0, ({u}(1))^{-1} + \ldots + ({u}(k))^{-1} ]\subset [0,1] $. The ($\C^{m}$, for each $m$) uniform limit $v_0$ of the $v_k$ is therefore a smooth nonnegative function with $\int v_0 =1$, $\mathrm{supp} \,v_0\subset [0,1]$. Moreover,  $v_0$ is strictly positive in $(0,1)$, and satisfies the bounds (see \cite[Theorem 1.3.5]{HOR} for details)
\begin{equation}
\sup |v_0^{(k)}| \leq 2^k  \big(\textstyle\prod_{j=1}^{k+1} {u}(j)\big) \leq (2 u(k+1))^{k}, \qquad k=0,1,\ldots,
\label{eqvo}
\end{equation}
where the last inequality comes from $u$ being increasing.
We set 
%$$
%\upsilon_0(x):= \int_{\R} v(\xi) \e^{2\pi i x \xi}\, \d  \xi; 
%$$
$$
v(\xi):= \int_{-1}^{1} v_0\big(3\xi-t+ \textstyle\frac12\big)\,\d t, \displaystyle \qquad
\upsilon(x):= \int_{\R} v(\xi) \e^{2\pi i x \xi}\, \d  \xi= \int_{-\frac12}^{\frac12} v(\xi) \e^{2\pi i x \xi}\, \d  \xi.
$$  Since $\int v_0=1$, \eqref{averaging} follows by construction. 
For all $x  \in \R$, $k \geq 0$, $ 0 \leq j \leq D$, we have that
\begin{align*}    |  x|^{k}|{{\upsilon}}^{(j)}(x)| & \leq (2\pi)^{j-k}\sup_{|\xi| \leq \frac12} \Big| \big({\textstyle\frac{\d}{\d \xi}}\big)^k \big( \xi^j v(\xi)\big) \Big| \leq    (2\pi)^{D-k}\sum_{n=0}^{j} {\textstyle \binom{k}{n} \frac{j!}{(j-n)!}}  \sup_{|\xi| \leq \frac12}\big(|\xi|^{j-n}|v^{(k-n)}(\xi)|\big)
\\ & \leq C_D  k^D \big(6{u}(k+1)\big)^k.
  \end{align*} 
  In the  last step above, we employed the crude bound $\binom{k}{n} \leq k^n \leq k^D$, and subsequently  \eqref{eqvo} coupled with the obvious fact that $\sup|v^{(k)}| \leq 2 \cdot 3^k \sup |v_0^{(k)}|$.
For each  $|x| \geq 6\e u(1) $, let $k(x)$ be the greatest integer $k\geq 0$ such that $6{u}(k+1)|x|^{-1} \leq \e^{-1} 	 $. Thus   $   k(x)+1 \leq U( x/6\e)\leq  k(x)+2 $, and the above display  for $k=k(x)$  reads
\begin{align*}
|{{\upsilon}}^{(j)}(x)| &\leq C_D \big( k(x) \big)^D {\e}^{-k(x)}  \leq C_D \big(U (x/6\e)\big)^D  \e^{-U(x/6\e)} \leq  C_D \big(1+\textstyle\frac{|x|}{6\e}\big)^D  \e^{-U(x/6\e)}\\ & \leq C_D   \textstyle B_u(2D ) \e^{-\frac12 U(x/6\e) }  \leq C_{D,u} \e^{-\frac{1}{100} U(x) }.
\end{align*}
We have relied on \eqref{tricco3} for the third inequality, on \eqref{tricco2} to pass to the second line  and on \eqref{tricco1}
  for the last step. We have thus obtained   \eqref{decayA}  for $|x|\geq6 \e u(1)$, with $C_{D,u}=C_D   \textstyle B_u(2D)$.
To argue for $|x|\leq 6\e u(1)$, note that  the bound $\sup|{{\upsilon}}^{(j)}| \leq C_Du(1)$  can be  inferred as a particular case of the above discussion. In the range   $|x|\leq 6 \e u(1)$, this entails 
 \eqref{decayA}  with $C_{D,u} := C_D u(1) \e^{100 U(6\e u(1))} \leq C_D u(1) \e^{600\e   u(1)}$, which   depends only on $D$ and $B_u(1)$, and is thus of the required form. This concludes the proof of the lemma. \end{proof}
%\begin{remark} \label{uisSCH}  We already know that ${{\upsilon}}$ is a Schwartz function, since $\widehat \upsilon=v$ is smooth and compactly supported.   Coupling the obtained \eqref{decayA} with \eqref{tricco2},  we have that 
%$$
%|{{\upsilon}}^{(j)}(x)| \leq C_{D,u} \e^{-\frac{1}{100} U(x) } \leq  C_{D,N,u} (1+|x|)^{-N} \qquad \forall\,  x \in \R, \, \forall\, j= 0, \ldots, D,\
%$$
%for all integers $D,N \geq 0$, having set $C_{D,N,u}:=C_{D,u} B_u(100N)$.
%\end{remark}
\begin{remark} 
The existence of an exponentially decaying smooth function   with compactly supported Fourier transform is forbidden by the Paley-Wiener theorem.
In  \cite{ING}, it is pointed out that if   $u$ is such that the integral in \eqref{int1} diverges, 
 there exists no such    function  decaying like \eqref{decayA}.   For instance, there is no smooth function $f$ with $\widehat f$ compactly supported  and decaying like $|  f(x)| \lesssim \exp\big(-c|x|/\log(\e+|x|) \big)$. 
% 
% However,   for each $u_\lambda$ of the family \eqref{familyu}, the associated $U=U_\lambda$ satisfies the lower bound\begin{equation}
%\label{familyu}
%U_\lambda(x) \geq \min\Big\{0,\frac{\lambda|x|}{\big(\log(\lambda|x|)\big)^{1+\lambda}} -\e \Big\}, \qquad x \in \R,
%\end{equation} 
%thus allowing for decays of the type $|  f(x)| \lesssim \exp\big(-\lambda|x|/(\log(\e +\lambda|x|))^{1+\lambda} \big)$ for each $\lambda>0$.
\end{remark}
For the remainder of the section, we write $I_0:=[-\frac12,\frac12]$. In the next two lemmata, we devise a splitting of a  smooth function with spatial decay rate $aU$ and frequency supported on $I_0$ into a  part  having   spatial support contained in the $u(K)$-dilate of $I_0$ and       Fourier transform exponentially small in $K$ away from $I_0$,  plus an exponentially small remainder.
\begin{lemma}\label{lemmasplitmax}
Let   $\varphi$ be  a Schwartz function with $ \supp\,\widehat 	\varphi \subset { I_0} $ and satisfying the bound 
\begin{equation} \label{decayb}
 \big|{{\varphi}}^{(j)}(x)|	 \leq  A  \e^{-a  U( x)}    \qquad  \forall x \in \R, \, j=0,1,
\end{equation}
for some constants $A>0,0<a\leq \frac{1}{100}$.
 For each $K\geq 1$, $N>0$  there exists a decomposition
\begin{equation}\label{eqsplit}
{\varphi} := \phi  +  \e^{-\frac{a}{12} K}\psi 
\end{equation}
  with  the following properties:
\begin{align} &
\mathrm{supp}\, \phi    \subset  u(K)I_0, \label{truncsupp} \\ 
& |\widehat{\phi} (\zeta)| \lesssim     \e^{-\frac{a}{12} K}  (1+|\zeta|)^{-N}\qquad \forall |\zeta|\geq 2, \label{fastfd}\\
& \big|	\phi^{(j)}(x) \big| \lesssim         (1+|x|)^{-N} \qquad \forall x \in \R, \, j\in\{0,1\}, \label{adaptphi}\\
& \big|	\psi^{(j)}(x) \big| \lesssim        (1+|x|)^{-N} \qquad \forall x \in \R, \, j\in\{0,1\}.\label{smallwp}
\end{align}
The implicit constants in \eqref{fastfd}-\eqref{smallwp}  depend only on $A,a,N$ and   $u$.
\end{lemma}
\begin{proof}   Let  ${v}$  be a smooth function such that
\begin{equation}
\cic{1}_{\frac{u(K)} 3 I_0} (x) \leq v(x) \leq  \cic{1}_{{u(K)} I_0} (x) \quad \forall\,  x \in \R, \qquad 
|\widehat{  {{v}}}
 (\zeta)|   \lesssim {{{u(K)}}} \e^{- \frac{1}{100}  U({{{u(K)}}} \zeta)} \quad \forall\, \zeta \in \R;\label{decaytaumu}
\end{equation}
such a function exists by Lemma \ref{HORM}.
We realize the decomposition \eqref{eqsplit} by setting
$$ \phi =  \varphi v ,\qquad \psi(x):=   \e^{\frac{a}{12} K}     \varphi  (\cic{1}_{\R}-{v}).
$$
Then    \eqref{truncsupp} holds by construction. Furthermore, we obtain \eqref{smallwp} from the bound
\begin{equation} \label{tothenext}
 \Big| \big((\cic{1}_{\R}-{v}) {\varphi}\big)^{(j)}(x)\Big|  \lesssim \textstyle \e^ {-\frac{a}{2} U({{{u(K)}}}/6)}
 \e^{-\frac{a}{2} U(x)}\lesssim  \e^{-\frac{a}{12} K}
 (1+|x|)^{-N},
\end{equation}
for $x \in \R$ and $  j=0,1$, which follows by restricting to $|x| \geq u(K)/ 6$ via  support considerations,  then relying on \eqref{decayA}, and finally using  \eqref{tricco1} and \eqref{tricco2}. Then, \eqref{adaptphi} is  derived  by comparison with \eqref{smallwp}.
 We are   left with proving \eqref{fastfd}, that is, estimating $\widehat{{{\varphi}} {v}}(\zeta)$ for $|\zeta|\geq 2$. To do so, we use 
 $ |\widehat 	\varphi |\lesssim \cic{1}_{ I_0} $ and later \eqref{decaytaumu}, so that  \begin{equation} 
\big|(\widehat {{\varphi}} 
* \widehat{{v} }) (\zeta)\big|   \lesssim \int_{I_0} |{v}(\zeta-\eta)| \, \d \eta \lesssim  \sup_{\eta \in {I_0}} |{v}(\zeta-\eta)|   \lesssim {{{u(K)}}} \e^{-\frac{1}{100}U({{{u(K)}}}(|\zeta|-1))}, \nonumber 
\label{killing}
\end{equation}
By repeatedly making use of \eqref{tricco2}, it is easy to see that, when   $|\zeta|\geq 2$, the last right hand side is bounded by $\exp(-aK/12)(1+|\zeta|)^{-N}$ times a constant depending on $u$ and $N$ only. This concludes the proof of the lemma.
\end{proof}
   The next decomposition, which is similar to the one of  Lemma \ref{lemmasplitmax}, but preserves mean zero with respect to a fixed frequency outside $I_0$, was partly inspired by \cite[Lemma 3.1]{MPTT}.
\begin{lemma}\label{lemmasplit} 
Let   $\varphi$, $ K,N$, be  as in Lemma \ref{lemmasplitmax}, $R>1$, $\xi_0 \in R I_0 \backslash I_0 $. There exists a decomposition
\begin{equation}\label{eqsplit2}
{\varphi} := \phi  +  \e^{-\frac{a}{12} K}\psi 
\end{equation}
depending on $\xi_0$, such that \eqref{truncsupp}-\eqref{adaptphi} hold for $\phi$, \eqref{smallwp} holds for $\psi$, and in addition
\begin{equation}
  \int_{\R} \phi   (x) \,\e^{-2\pi i \xi_0  x}\, \d x = \int_{\R} \psi  (x) \,\e^{-2\pi i \xi_0 x}\, \d x =0,  \label{meanzeroinlemma}  
\end{equation}
The implicit constants in \eqref{fastfd}-\eqref{smallwp}  depend only on $A,a,R,N$ and   $u$.
\end{lemma}
\begin{proof}     Set  $w (\cdot):=\varphi(\cdot) \e^{-2\pi i \xi_0 \cdot}$. Then $0 \not\in\mathrm{supp} \, \widehat {w}= I_0-\xi_0$, that is $\int  {{w}}=0$.
Let  ${v}$  be as in \eqref{average}
In view of the support condition on $\cic{1}_\R-v$ and later relying on  \eqref{decayb},  we preliminarily observe that 
\begin{align}\label{average} \quad &  \Big|\int_{\R}w v \, \d x\Big|   =   \Big|\int_{\R}w  (\cic{1}_\R -{v} ) \,\d x\Big|\lesssim   \int_{ |x| \geq \frac{u(K)}{6} }|w(x)|\, \d x \\ & \leq  A  \int_{ |x| \geq \frac{ u(K) }{6} }\e^{- aU(x)} \d x 
\lesssim  \e^{-\frac{a}{2} U( u(K) /6)}   \int_\R \e^{- \frac{a}{2} U(x)} \d x 
 \lesssim  \e^{-\frac{a}{12} K}.\nonumber  
 \end{align}
For the next to last inequality above, we used that $U$ is increasing. Then, in  the last step, we    employed \eqref{tricco1} for the first factor and \eqref{tricco2} to estimate the integral. 
 Since  $w(\cdot)=\varphi (\cdot) \e^{-2\pi i \xi_0 \cdot}$, \eqref{eqsplit2} is fulfilled if we set\begin{align*}
&
 \phi(x):=\left( w(x)v(x) - \frac{\int wv}{ \int v} v(x)  \right) \e^{2\pi i \xi_0 x},\\ &
 \psi(x):=   \e^{\frac{a}{12} K}    \left( \frac{\int wv}{  \int v} {v}(x) + (\cic{1}_{\R}-{v}) (x){{w}}(x) \right) \e^{2\pi i \xi_0 x}.\end{align*}
With these definitions,
the mean zero condition \eqref{meanzeroinlemma} for $\phi$  holds by construction. Then, \eqref{meanzeroinlemma} for $\psi$ follows by difference,  again in view of $\int w=0$. 
By construction as well, $\mathrm{supp}\, \phi   \subset {{{u(K)}}} I_0$, and we have earned \eqref{truncsupp}.
 
 Next, we prove \eqref{adaptphi} and \eqref{smallwp}. Recalling that the  implicit constants are allowed to depend on  $R$, and that $|\xi_0| \leq R$, we can ignore the modulation factor of $\psi$, and   \eqref{smallwp} is a consequence of the bounds
\begin{align*}
 &\left|  \frac{\int_{\R}({{w}} {v})}{ \int_\R {v}} {v}^{(j)}(x) \right| \lesssim \textstyle  \frac{\e^{-\frac{a}{12} K}}{{{{u(K)}}}} (1+{{{u(K)}}})^{N}(1+|x|)^{-N} \lesssim  \e^{-\frac{a}{12} K}  (1+|x|)^{-N},
 \\ 
 &\Big| \big((\cic{1}_{\R}-{v}) {{w}}\big)^{(j)}(x)\Big| \lesssim \textstyle \e^ {-\frac{a}{2} U({{{u(K)}}}/6)}
 \e^{-\frac{a}{2} U(x)}\lesssim  \e^{-\frac{a}{12} K}
 (1+|x|)^{-N},
\end{align*}
for $x \in \R$ and $  j=0,1$. For the first line of the last display, we have used inequality \eqref{average} and  that $v^{(j)}$ is supported on $|x| \leq  u(K)/ 2$,  and subsequently \eqref{tricco2}. The second line follows via the same argument we used for \eqref{tothenext}.

 We now turn to \eqref{fastfd}.   The term involving $\widehat{{v}}$ is easily bounded, taking \eqref{average}, \eqref{decaytaumu} and \eqref{tricco2} into account, by
$$
 \e^{-\frac{a}{12} K} |\widehat{ {v}}(\zeta -\xi_0)| \lesssim   \e^{-\frac{a}{12} K}
 \e^{-\frac{1}{100}U({{{u(K)}}} (\zeta-\xi_0)) }    \lesssim  \e^{-\frac{a}{12} K} (1+{{{u(K)}}}|\zeta-\xi_0|)^{-N}
  \lesssim  \e^{-\frac{a}{12} K}  (1+|\zeta|)^{-N}.$$
The  above estimate  actually holds for all $\zeta \in \R$, and  the last almost inequality sign hides the constant  $(1+|\xi_0|)^N\leq (1+R)^N$, which we ignore.
Finally, the term $(\widehat{{w}} * \widehat{v})(\zeta -\xi_0)$ is handled in exactly the same fashion of \eqref{killing}. The proof is complete.\end{proof}

 \section{A multi-frequency Calder\'on-Zygmund decomposition} \label{secMFCZ}

  Throughout this section,  our definitions   depend on a fixed choice of the function $u$, and of its extended inverse $U$, as in   Section \ref{sect2}, and of   parameters $R\geq 1,0<\eps \leq 2^{-8}$.     The almost inequality signs appearing in the sequel hide implicit constants which are allowed to possibly depend on $u$ and  $R$  only.

\subsection{Top data and adapted functions with fast spatial decay}

We call \emph{top datum} a pair $(I,\xi)$, where $I\subset \R$ is a spatial interval  and $\xi \in \R$ is a frequency. 
We say that a smooth function $\varphi $ is  $u$-adapted\footnote{Here, and in the remainder of the article, we adopt an $L^2$ normalization for our adapted functions.}, with adaptation rate $a>0$, to the top datum $(I,\xi)$  if 
\begin{equation}
  \Big|  \big(\e^{-2\pi i  \xi\cdot }	\varphi (\cdot) \big)^{(j)} (x) \Big| \lesssim    |I|^{-\frac12-j}  \exp  \Big(- a U\big(\textstyle  \frac{|x-c(I)|}{|I|}\big)\Big) \qquad \forall x \in \R, \, j\in\{0,1\}. \label{adaptbandlim}
\end{equation}
Note that, by virtue of property \eqref{tricco2} of $u$, \eqref{adaptbandlim} is stronger than the usual notion of adaptation of e.\ g.\ \eqref{polyadapt} below.
  In what follows, the adaptation rate $a$  will   be   an absolute constant, which may be different at each occurrence.
When we speak about collections of $u$-adapted  functions, we  assume, without explicit mention,   uniformity of the adaptation rate $a$  and of the implicit constants in the almost inequality sign of \eqref{adaptbandlim}.      
Using the  property of $U$ \eqref{tricco1},  we see that  for any pair of top data $(I,\xi),(J,\zeta)$   such that $|J|\leq |I|\leq 2|J|$, $|c(I)-c(J)| \leq |J|$ and $|\zeta-\xi|\leq R|I|^{-1}$, any $\varphi$ 
 $u$-adapted to $(I,\xi)$  is also $u$-adapted to $(J,\zeta)$, with  adaptation rate  $a'\geq a/2$. 
  Hence,  there is  no   loss in generality with assuming that the spatial intervals $I$ of our top data belong to the standard dyadic grid $\D_0$.

\begin{remark} \label{theexample}
For each $0<\eps\leq 1$, examples of  $u$-adapted functions to $(I,\xi)$,   with adaptation rate $a=\eps/100$, are   given by  
$$
\upsilon_{(I,\xi),\eps}:= \mathrm{Mod}_{\xi}\mathrm{Dil}_{\eps^{-1}|I|}^2 \mathrm{Tr}_{c(I)}
\upsilon,$$
where $\upsilon$  is the   output of Lemma \ref{HORM} corresponding to $u$.
In view of \eqref{averaging}, $\widehat{\upsilon_{(I,\xi),\eps}}$ is supported on the interval of length $\eps|I|^{-1}$ centered at $\xi$.      \end{remark}

 \subsection{A multi-frequency Calder\'on-Zygmund decomposition with respect to top data}
  Let $(I,\xi)$ be a top datum. A set of  functions    $  \cic{\varphi}_{(I,\xi)}=\{\varphi_J : J \in \D(I)\}$, indexed by the dyadic subintervals of  $I$,  is a collection of \emph{$u$-adapted   wave packets} if  \begin{itemize}
\item[$\cdot$]   each $\varphi_J $ is smooth and  $u$-adapted  to $(J,\xi)$,
\item[$\cdot$]  the support of each $\widehat{\varphi_J}$ is  contained on an  interval of length $\eps|J|^{-1}$ centered at $\xi_J \in \R$, and $  |\xi_J-\xi| \leq R |J|^{-1}$.
\end{itemize}
 If furthermore, for each $\varphi_J\in \cic{\varphi}_{(I,\xi)} $ 
 we also have that  $  |\xi_J-\xi| \geq  |J|^{-1}$, so that  consequently $\xi \not \in \supp\, \widehat{\varphi_J}$,   and in particular $\widehat{\varphi_J}(\xi)=0$, 
we say that $\cic{\varphi}_{(I,\xi)}$ is a collection of $u$-adapted   wave packets \emph{with mean zero} with respect to the top datum $(I,\xi)$.

In Proposition \ref{corMFCZabs},   we devise  a  multi-frequency  Calder\'on-Zygmund decomposition $f=g+b$ of  $f \in L^p(\R), 1\leq p<2$, adapted to a set of top data $(I,\xi) \in \mathcal T$. The $L^2$ norm of the \emph{good part} $g$   will depend  on the  $L^1$ norm  of the counting function  associated to the spatial intervals of $ \mathcal T$ and, via $u$, on a parameter $k$ related to the $L^\infty$ norms of the counting function. The bad part $b$   is such that    
the Carleson measure norms of  the coefficients $|\l b, \varphi_J \r|^2$, associated to collections $  \cic{\varphi}_{(I,\xi)}$ of $u$-adapted wave packets with mean zero   are  \emph{simultaneously} exponentially small in $k$. 
 The constant $C$ appearing in the statement can be taken equal to $10^3 a^{-1}$, where $a$ is the uniform adaptation rate of the $\cic{\varphi}_{(I,\xi)}$'s.   \begin{proposition} \label{corMFCZabs}
  Let  $k\geq 1$ and   ${\mathcal T}=\{(I, \xi)\}$ be a collection of top data satisfying
 \begin{equation} \label{flinftbd}
 \Big\| \sum_{ (I, \xi) \in {\mathcal T}} \cic{1}_{3u(Ck) I}\Big\|_{\infty} \leq 2^k.
 \end{equation}  Let   $f \in L^p(\R)$, $1\leq p <2$, be given. 
For each  $\lambda>0$, denote by  
$  E_\lambda =\{x \in \R: \M_p f(x) >  \lambda\}.  $  Then, there exists a decomposition $f=g+b$ such that
 \begin{equation}
 \|g\|_2 \lesssim \lambda^{2-p} \|f\|_p^{p-1}    \Big( u(Ck)^2 \log u(Ck) \Big\|\sum_{(I,\xi) \in \mathcal T} \cic{1}_{I} \Big\|_1 \Big)^{\frac1p-\frac12}
 \label{gl2again}
 \end{equation}
 and such that, for each $(I,\xi) \in \mathcal T$, 
 \begin{equation} 
   \sup_{J \not\subset {E_\lambda}} \frac{|\l b,\varphi_J\r|}{|J|^{\frac 12}} \lesssim    2^{-4k} \lambda,     
   \label{megaesti1}
   \end{equation}
 whenever   $\cic{\varphi}_{(I,\xi)}$ is a collection of $u$-adapted wave packets with top datum $(I,\xi)$, and 
  \begin{equation} 
    \Bigg(  \frac{1}{|J_0| }\sum_{ J\not\subset {E_\lambda}, J \subset J_0        } |\l b, \varphi_J\r|^2 \Bigg)^{\frac12} 
\lesssim   2^{-4k}\lambda
   \label{megaesti}
   \end{equation}
  whenever   $\cic{\varphi}_{(I,\xi)}$ is a collection of $u$-adapted wave packets with mean zero with respect to  $(I,\xi)$, and  $J_0 \in \D(I),  J_0 \not\subset {E_\lambda}$.

 \end{proposition}

The proof is given in Subsection \ref{ssproofcormfcz}. In the next subsection, we  develop some technical preliminaries involving $u$-adapted functions. 
\begin{remark}[Dyadic structure of $ \supp\, \widehat{\varphi_J}$] 
It will be   useful to give some sort of dyadic structure to the frequency supports of $\varphi_J \in \cic{\varphi}_{(I,\xi)} $ as well, working with the standard dyadic grid $\D_0$ and its translates $\D_j:=\{\omega+  j   |\omega|/3:\omega\in \D_0\}$, $j=1,2$. We do so by selecting  for each $\varphi_J$ the   unique interval ${\omega} \in \D_0 \cup \D_1 \cup \D_2$ of length $|J|^{-1}$ with minimal $c(\omega)$ such that $(\xi_J-2\eps|J|^{-1}, \xi_J+2\eps|J|^{-1}) \subset \omega $, which   we denote   by $\omega_J$.   Noting that $\xi \in R\omega_J  $ by definition, we realize that the collection $\{\omega_J:   |J|=2^\ell\}$ has at most $3R$ elements. By   pigeonholing   $\cic{\varphi}_{(I,\xi)},$ at the cost of an additional $\lesssim R$ factor in our estimates,   we can  assume   that all $\omega_J$'s come from the same dyadic grid $\D_0$, and that $\omega_J=\omega_{J'}$ whenever $|J|=|J'|$. \label{remdyadfreq}
\end{remark}

\subsection{A splitting of the $u$-adapted wave packets} \label{sssplit}
Let throughout $\varphi_J \in \cic{\varphi}_{(I,\xi)}$ be a collection of $u$-adapted wave packets. The functions
$$
\varphi_J^{(0)}:=  \Big( \mathrm{Tr}_{-c(J)}   \mathrm{Dil}_{|J|^{-1}}^2  \mathrm{Mod}_{-\xi_J} \varphi\Big), \qquad  \varphi_J^{(0)}(x)=\varphi \big( c(J) + |J| x  \big)  \e^{-2\pi i\xi_J (   c(J) + |J| x)},
$$ 
whose frequency support lies in $I_0=[-\frac12,\frac12]$, satisfy the decay assumption \eqref{decayb} of Lemma \ref{lemmasplitmax} with a uniform choice of constants $A,a$. Applying Lemma \ref{lemmasplitmax} to each $\varphi_J^{(0)}$, for a fixed parameter $K\geq 1$,  we split $\varphi_J^{(0)}=\phi + \exp(-aK/12) \psi$. The resulting decomposition\begin{equation}
\label{eqsplitconcrete}
\varphi_J = \phi_J + \e^{-\frac{a}{12}K} \psi_J, \qquad \phi_J:=  \mathrm{Mod}_{\xi_J}\mathrm{Dil}_{|J|}^2 \mathrm{Tr}_{c(J)} \phi, \quad \psi_J:=  \mathrm{Mod}_{\xi_J}\mathrm{Dil}_{|J|}^2 \mathrm{Tr}_{c(J)}  \psi\end{equation}
inherits, from   \eqref{truncsupp}, \eqref{fastfd} and   \eqref{adaptphi}-\eqref{smallwp}  respectively, the properties $(J,\zeta)$,   \begin{align}
 &\mathrm{supp}\, \phi_J   \subset  u(K)J, \label{truncsupptreat} \\ 
& |\widehat{\phi_J} (\zeta)| \lesssim     \e^{-\frac{a}{12}K}  |J|^{\frac12} \big( |J||\zeta-\xi_J|\big)^{-N}\qquad \forall |\zeta-\xi_J| \geq 2|J|^{-1}, \label{fastfdtreat} \\
&
  \Big|  \big(\e^{-2\pi i  \zeta\cdot }	\gamma_J (\cdot) \big)^{(j)} (x) \Big| \lesssim    |J|^{-\frac12-j}   \Big(1+\textstyle  \frac{|x-c(J)|}{|J|} \Big)^{-N} \quad \forall x \in \R, \, j\in\{0,1\} \label{polyadapt}
 \end{align}
the last property holding for either $\gamma_J \in \{\phi_J,\psi_J\}$
and for all $|\zeta-\xi_J| \lesssim R|J|^{-1}$.

If, in addition, $ \cic{\varphi}_{(I,\xi)}$ are $u$-adapted wave packets \emph{with mean zero} with respect to $(I,\xi)$, we can apply Lemma \ref{lemmasplit} to $\varphi_J^{(0)}$ instead, with  choice of frequency $\xi_0:= |J|^{-1}(\xi-\xi_J) \in RI_0 \backslash I_0$, and obtain a decomposition \eqref{eqsplitconcrete} such that, together with the above properties, there holds
\begin{equation} \widehat{\phi_J} (\xi)= \widehat{\psi_J} (\xi)=0. \label{meanzerointreat} 
\end{equation}

In the remainder of this subsection, we present several results involving the  functions $\phi_J,\psi_J$ arising from the above decompositions of $\varphi_J \in \cic{\varphi}_{(I,\xi)}$. We  start with the  following well-known observation, whose  proof relies on integration by parts; see   \cite[Lemma 2.3]{ThWp}.
\begin{lemma} \label{lemmaibp} Let $\gamma_J$ be a smooth function satisfying \eqref{polyadapt} for $\zeta=\xi \in \R  $. Let $H \subset \R$ be an interval,   $h  \in L^{1}(\R) $ with $ \mathrm{supp}\, h \subset H$ and such that 
\begin{equation} \label{meanzerosibp}  
\int_{\R} h(x) \e^{ -  2\pi i   \zeta    x} \, \d x = 0  \end{equation}
for some $\zeta \in \R$ with $|\zeta -\xi |\lesssim R|J|^{-1}$. 
Then, there holds
$$
|\l h, |J|^{\frac12} \gamma_J \r |\lesssim  \min\Big\{\textstyle\frac{|H|}{|J|},1 \Big\} \Big( 1+ \frac{|c(H)-c(I)|}{|I|}\Big)^{-N} \|h\|_{1}.
$$
\end{lemma}
Using Lemma \ref{lemmaibp} above,  as well as the classical Calder\'on-Zygmund decomposition of $h$, one obtains the  Carleson measure type estimate of the next lemma, which, aside from notation,  is the same as  e.\ g.\ \cite[Proposition 2.4.1]{ThWp}.
  \begin{lemma} \label{thelemmapsi} Let $\{\gamma_J: J \in \D(I)\}$ be such that  each  $\gamma_J$ satisfies  \eqref{polyadapt}  with $\zeta=\xi$, and in addition  $\widehat{\gamma_J}(\xi)=0$. For all  $h \in L^1_\mathrm{loc}(\R)$, $\lambda >0$,  $J_0 \in \mathcal{D}(I)$, $J_0 \not\subset \{\M_1 h >\lambda\}$
\begin{equation} \label{sqfh}
\Bigg(\frac{1}{|J_0|}  \sum_{  J \not\subset \{\M_1 h >\lambda\}, J \subset J_0 } |\l h, \gamma_J\r|^2  \Bigg)^{\frac12}\lesssim \lambda.
\end{equation}
\end{lemma}
Of course, one can take $\gamma_J=\varphi_J,\phi_J, \psi_J$  in  the above lemmata. Below, we rely on    
$\phi_J$ being  a smooth function on the torus $  3u(K) J$, with exponentially small Fourier coefficients  outside the frequency  band
\begin{equation}
\label{XiJ}  
 \Xi(J):=5\omega_J \cap \displaystyle\frac{\mathbb Z}{3u(K)|  J|},
\end{equation} 
to show that  if  an integrable function $h$ is supported on $3u(K) J$ and has mean zero with respect to each $\zeta \in \Xi(J)$, then  its integral against $|J|^{\frac12}\phi_J$ is exponentially   small.

 \begin{lemma} \label{lemmadecay}Let $\phi_J$ be a smooth function satisfying \eqref{truncsupptreat} and \eqref{fastfdtreat}. Let $h  \in L^{1}(\R) $ with $ \mathrm{supp}\, h \subset 3 u(K) J$ and such that 
\begin{equation} \label{meanzeros}  
\int_{\R} h(x) \e^{ -  2\pi i   \zeta    x} \, \d x = 0 \qquad \forall\,  \zeta 	\in  \Xi(J).
\end{equation}
Then, there holds
\begin{equation} \label{lemmadecayeq}
|\l h, |J|^{\frac12} \phi_J \r |\lesssim  \e^{-\frac{a}{12} K}  \|h\|_{1}.
\end{equation}
\end{lemma}
\begin{proof} We write $\mu=3 u(K) |J|$ for brevity.
In view of the support condition \eqref{truncsupptreat} and of the decay \eqref{fastfdtreat}, $|J|^{\frac12}\phi_J$  has  Fourier coefficients on the torus $3 u(K) J$ given by
$$
c_\zeta = \frac{1}{{  \mu}} \int_{ 3 u(K)  J } |J|^{\frac12}\phi_J(x) \e^{-2\pi i  \zeta   x}= \frac{|J|^{\frac12}}{ \mu} \widehat{\phi_J}(\zeta), \qquad \zeta \in   \mu^{-1}{\mathbb Z}. 
$$
From \eqref{fastfdtreat}, we learn that
$$
|c_\zeta| \lesssim  \e^{-\frac{a}{12} K} \frac{|J|}{\mu}     \big(|J||\zeta-\xi_J|\big)^{-N}, \quad \forall\, \zeta  \not \in \Xi(J).
$$
  Using assumption \eqref{meanzeros} for the second equality, and later the last display, 
\begin{align*}
\l h, |J|^{\frac12}\phi_J \r & = \int_\R h(x) \left( \sum_{\zeta \in  \mu^{-1} \mathbb{Z}  } \overline{c_\zeta} \e^{-2\pi i \zeta   x}\right)\, \d x = \sum_{\substack{\zeta \in  \mu^{-1}  \mathbb{Z}\\ |J||\zeta-\xi_J|>2}  }\overline{c_\zeta}\l h, \e^{2\pi i \zeta \cdot} \r \leq \|h\|_{1} 
\sum_{\substack{ \zeta \in  \mu^{-1}  \mathbb{Z}\\ |J||\zeta-\xi_J|>2} } |c_\zeta |
  \\ &  \lesssim     \e^{-\frac{a}{12} K}  \|h\|_{1}  \sum_{\substack{ \zeta \in  \mu^{-1}  \mathbb{Z}\\ |J||\zeta-\xi_J|>2} }  {\mu}^{-1} |J| \big(|J||\zeta-\xi_J|\big)^{-N} \lesssim	    \e^{-\frac{a}{12} K}  \|h\|_{1}   
\end{align*}
 as claimed. The final inequality is obtained by interpreting the last  summation over $\zeta$ as a Riemann sum. \end{proof}  
With the above lemmata in hand, we  devise an exponentially small estimate for the discrete square function involving the $\phi_J$'s associated to the top datum $(I,\xi)$, when $h \in L^1(\R)$ is supported on $H$ and has zero average against a  set of   $\lesssim Ku(K)$ frequencies near $\xi$, defined in \eqref{equationxiI}.
\begin{lemma} \label{lemmasphi}  
For a top datum $(I,\xi)$ and an interval $H \subset \R$    with $ I \not\subset 3H, \,H \cap 3 u(K) I \neq \emptyset$, define 
\begin{equation} 
\label{equationxiI} \Xi_H(I,\xi):= 
\{ \xi \}
\cup \bigcup\big\{ \Xi(J): J \in \D(I), J \not\subset 3H,\,H \subset 3u(K) J,  |J| \leq 2^{10K} |H|   \big\}.
\end{equation}
 Let  $h  \in L^{1}(\R) $ with $ \mathrm{supp}\, h \subset H$ and such that 
\begin{equation} \label{lemmasphiass2}  
\int_{\R} h(x) \e^{ -  2\pi i   \zeta    x} \, \d x = 0 \qquad \forall \zeta \in\Xi_H(I,\xi).  \end{equation} 
Let $\{\phi_J: J \in \D(I)\}$ be a collection of functions each satisfying \eqref{truncsupptreat} and \eqref{fastfdtreat}. Then
\begin{equation} \label{sqfphi}
\Bigg\|\Bigg(  \sum_{ {  J \not\subset 3H}} |\l h, \phi_J\r|^2 \frac{\cic{1}_{J}}{|J|}\Bigg)^{\frac12}\Bigg\|_1 \lesssim \e^{-\frac{a}{24} K}   \|h\|_1.
\end{equation}
\end{lemma}
\begin{remark} \label{choiceoffreq}
Before entering the proof, observe  that, if $H \cap 3 u(K) I= \emptyset$ or   $I \subset 3H$, then  the square function in \eqref{sqfphi} is identically zero.
Furthermore,   as a consequence of the definition and of the discussion in Remark \ref{remdyadfreq},  $\Xi(J)=\Xi(J')$ whenever $J,J'\in \D(I), |J|=|J'|=2^{-\ell} |I|$. Below, we refer to this common discrete interval, depending only on $\ell $ and $(I,\xi)$, by $\Xi  (I,\xi,\ell)$, and  we record the following observations.
\begin{itemize} \item[$\cdot$] If $H \cap 3 u(K) I \neq \emptyset, H\cap 9I=\emptyset$, then $\#\Xi_H(I,\xi) \lesssim u(K) \log u(K)$. Indeed, $3u(K) J \cap H \neq \emptyset$ only if $|I| \lesssim u(K) |J|$, thus $\Xi_H(I,\xi)$ is contained in the union over $0\leq \ell \lesssim \log u(K)$ of the  $\Xi(I,\xi,\ell)$,  each of which has $\lesssim u(K)$ elements.
\item[$\cdot$] Otherwise, within the assumptions of the Lemma,  it must be that $H \subset 9 I$. In this case, reasoning as above yields that $\Xi_H(I,\xi) \lesssim Ku(K)$.
\end{itemize}
\end{remark}
\begin{proof}[Proof of Lemma \ref{lemmasphi}]
By virtue of the support condition \eqref{truncsupptreat}, the left hand side of \eqref{sqfphi} is bounded by 
\begin{equation*} 
  \sum_{ \substack{J \not\subset 3H \\ u(K) J \cap H \neq \emptyset} } |\l h, |J|^{\frac12} \phi_J \r | = \sum_{ \substack{J \not\subset 3H \\  H\subset 3u(K) J } }  |\l h, |J|^{\frac12} \phi_J \r | = \sum_{\ell \gtrsim -\log u(K)}  \sum_{  J \in \cic{J}_\ell }|\l h, |J|^{\frac12} \phi_J \r | 
\end{equation*}
where $	\cic{J}_\ell := \{ J \not\subset 3H,    H\subset 3u(K) J, |J|\sim 2^{\ell}|H|\}$. Note that $\#\cic{J}_\ell \lesssim u(K) \log u(K) $. Now, for the $\lesssim K$ values $-\log u(K) \lesssim  \ell \leq 10K$, in view of  assumption \eqref{lemmasphiass2}, we apply Lemma \ref{lemmadecay} and estimate $|\l h, |J|^{\frac12} \phi_J \r | \lesssim \e^{-\frac{a}{12} K} \|h\|_1$. On the other hand, when $\ell \geq 10 K$, we use Lemma \ref{lemmaibp} and estimate $|\l h, |J|^{\frac12} \phi_J \r | \lesssim  |H| \|h\|_1/|J| \lesssim 2^{-\ell}\|h\|_1$. Summarizing, the last display is bounded by
$$
 u(K) \log u(K) \Big(K  \e^{-\frac{a}{12} K} + \sum_{\ell \geq 10 K} 2^{-\ell}\Big)\|h  \|_1 \lesssim \e^{-\frac {a}{24} K}\|h\|_1,$$
 where the last inequality follows from \eqref{tricco2}. The proof is complete.
\end{proof}

 \subsection{Proof of Proposition \ref{corMFCZabs}} \label{ssproofcormfcz} By linearity and dilation invariance of assumptions and conclusions, we can assume $\|f\|_p=1=\lambda$.   In the proof, we  write $K:= C k$, with $C=10^3 a^{-1}$.   
Let $ Q \in  \cic{Q} $ be the collection of   maximal dyadic intervals  such that
 $9Q\subset {E_\lambda}$. Then
 \begin{equation} \label{propoftheQs}
 \sup_{Q \in \cic{Q}} \|f\|_{L^p(Q)} \lesssim 1    \qquad \sup_{x \not \in \bigcup_{Q \in \cic Q} Q} f(x) \lesssim  1, \qquad \sum_{Q \in \cic Q} |Q| \leq |{E_\lambda}|  \lesssim  1
 \end{equation}
 and
the intervals $\{3Q: Q \in \cic{Q}\}$ have finite overlap. By virtue  of the second property in \eqref{propoftheQs}, there is no loss in generality with  assuming that $f$ is supported on $\cup_{Q \in \cic{ Q}} Q$. Also,  we  erase from $\mathcal T$ those $(I,\xi)$ with $I \subset 9Q$ for some $Q \in \cic{Q}$, since     \eqref{megaesti} is zero for such an $(I,\xi)$. \begin{proof}[Construction of $g$ and $b$] \let\qed\relax
For each $Q \in \cic{Q}$, referring to the notation in \eqref{equationxiI}    with the choice of $H=3Q$, define
\begin{equation} \label{thexiq}
\Xi_Q = \bigcup_{ 3Q \cap 3 u(K) I \neq \emptyset} \Xi_{3Q}(I,\xi).
\end{equation}
and denote \begin{align} \nonumber &
{\mathcal T}_{\mathrm{far}}(Q):= \big\{(I,\xi) \in {\mathcal T}:    3Q \cap 9I =\emptyset,  3Q \subset  3u(K) I \big\},\\ \nonumber
& {\mathcal T}_{\mathrm{near}}(Q):= \big\{(I,\xi) \in {\mathcal T}:    3Q \subset 9I \big\}.
\end{align}
In view of the discussion in Remark \ref{choiceoffreq}, and using assumption \eqref{flinftbd}, we have that \begin{align}
\label{MQ}  \# \Xi_Q &\lesssim u(K) \log( u(K)) \big( \#{\mathcal T}_{\mathrm{far}}(Q)\big) + (u(K))^2 \big(\#{\mathcal T}_{\mathrm{near}}(Q)\big) & \\ &\lesssim u(K)^2 \inf_{ 3Q}  \Big( \sum_{(I,\xi)} \cic{1}_{3u(K) I}\Big)\lesssim u(K)^2 2^{k}.   \nonumber
\end{align}
As in  \cite[Theorem 1.1]{NOT},  for each $Q \in \cic{Q}$
we define $g_Q$ to be the Riesz projection of $f_Q:=f\cic{1}_Q$ on the finite-dimensional subspace of $L^2(3Q)$ spanned by 
$
\big\{\exp(2\pi i \zeta x) : \zeta \in \Xi_Q\big\}
$
and set $b_Q:=f_Q-g_Q$, so that each $g_Q,b_Q$ is supported inside $3Q$ and 
\begin{equation}
\label{NOT2-fourieravs}  \int_{\R} b_Q(x) \e^{-2\pi i \zeta x}\, \d x = 0 \qquad \forall \zeta \in \Xi_Q.
\end{equation}
The elegant argument by Borwein and Erdelyi \cite{BE}, see \cite{NOT} for a proof, gives the estimate \begin{equation}
  \label{NOT2-fouriera}
\|g_Q \|_{L^2(3Q)}+ \|b_Q\|_{L^p(3Q)}    \lesssim    (\# \Xi_Q)^{\frac1p-\frac12} \|f_Q\|_{L^p(Q)}\lesssim   (\# \Xi_Q)^{\frac1p-\frac12}.
\end{equation} 
We finally set $g=\sum_{Q \in  \cic{Q}} g_Q,b=\sum_{Q \in  \cic{Q}} b_Q.$
\end{proof}
\begin{proof}[Consequences of the construction] \let\qed\relax
A preliminary observation is that \begin{equation} \label{lambda2k}\{\M_1 b \gtrsim  2^{k}\} \subset {E_\lambda} .
\end{equation}
 Indeed,   for each    interval $Z \not \subset {E_\lambda}$, 
$$
\|b\|_{L^1(Z)} \leq \|b\|_{L^p(Z)} \lesssim \Big(  \sum_{ \substack{Q \in \cic{Q} \\ 3Q \subset 3Z}} \textstyle\frac{|Q|}{|Z|} \|b_Q\|_{L^p(3Q)}^p\Big)^{\frac1p} \lesssim  \big(2^k u(K)^2 \big)^{\frac1p-\frac12} \lesssim   2^{k},
$$
whence the claim. We took into account that $b$ is supported on the union of $\{3Q \in \cic{Q}\}$, having finite overlap, and later \eqref{NOT2-fouriera} coupled with \eqref{MQ}.
 
The next estimate explains  the choice of the mean zero  frequencies $\Xi_Q$ for $b_Q$ which was done in \eqref{thexiq}. Indeed, we claim that whenever $\{\phi_J: J \in \D(I)\}$ is a collection of functions each satisfying \eqref{truncsupptreat} and \eqref{fastfdtreat},
\begin{equation} \label{BQ}
\Bigg\|\Bigg(  \sum_{ {  J \not\subset 9Q}} |\l b_Q, \phi_J\r|^2 \frac{\cic{1}_{J}}{|J|}\Bigg)^{\frac12}\Bigg\|_1 \lesssim \e^{-\frac{a}{24} K}   \|b_Q\|_1 \lesssim   2^{-10k } |Q|, 
\end{equation}
for all $Q \in \cic{Q}$. The last step simply follows from \eqref{NOT2-fouriera} and \eqref{MQ}. To obtain the first inequality, we have used that, if the left hand side of \eqref{BQ} is nontrivial, $b_Q$ has mean zero with respect to all the frequencies $\Xi_{3Q}(I,\xi) \subset \Xi_Q$ appearing in \eqref{lemmasphiass2}, and consequently appealed to Lemma \ref{lemmasphi}.
\end{proof}
\begin{proof}[Proof of \eqref{gl2again}] The function $$
F:=   u(K)\log (u(K))\sum_{(I,\xi)\in \mathcal T} \cic{1}_{3u(K)I} + u(K)^2\sum_{(I,\xi)\in \mathcal T} \cic{1}_{9I} 
$$
satisfies
$$
\|F\|_1 \lesssim u(K)^2\log (u(K))\Big\| \sum_{(I,\xi)\in \mathcal T} \cic{1}_{I} \Big\|_1, \qquad \Xi_Q \leq \inf_{x \in 3Q } F(x),
$$ 
the second inequality coming from the first line of \eqref{MQ}.
Making use of the finite overlap of the $3Q$'s and relying on \eqref{propoftheQs}, 
\begin{align*}&\quad
\|g\|_2^2  \lesssim   \sum_{Q \in \Q} \|g_Q\|_2^2 \lesssim      \sum_{Q \in \Q} |Q| (\#\Xi_Q \big)^{\frac2p-1 } \leq   \Big(\sum_{Q \in \Q} |Q| \Big)^{ 2- \frac2p} \|F\|_1^{ \frac2p-1 }\\  & \lesssim   \Big(u(K)^2\log (u(K))\Big\| \sum_{(I,\xi)\in \mathcal T} \cic{1}_{I} \Big\|_1\Big)^{   \frac2p-1},
\end{align*}
as claimed.  \end{proof}

\begin{proof}[Proof of \eqref{megaesti1}]
  Fix   $(I,\xi) \in \mathcal T$, a collection of  $u$-adapted wave packets  $\cic{\varphi}_{(I,\xi)}$ and  $\varphi_J \in \cic{\varphi}_{(I,\xi)}$ with $J \not \subset {E_\lambda}$.  To bound $|J|^{-\frac12} |\l b, \varphi_J\r|$, we rely on the decomposition  \eqref{eqsplitconcrete} of $\varphi_J$, and on properties \eqref{truncsupptreat}-\eqref{polyadapt}  of $\phi_J$, $\psi_J$.
It is easy to see that
\begin{equation}
\label{megaesti1pf1}
\e^{-\frac{a}{12} K}\frac{|\l b,\psi_J\r|}{|J|^{\frac 12}} \lesssim  \e^{-\frac{a}{12} K} \inf_{x \in J} \M_1 b(x) \lesssim   2^{-4k},
\end{equation}
where the last inequality follows from \eqref{lambda2k} and $J \not \subset {E_\lambda}. $
We then use \eqref{BQ} to estimate  
\begin{equation}
\label{megaesti1pf2}
\frac{|\l b_Q,\phi_J\r|}{|J|^{\frac 12}} \leq  \frac{|Q|}{|J|} \Bigg\|\Bigg(  \sum_{ {  J \not\subset 9Q}} |\l b_Q, \phi_J\r|^2 \frac{\cic{1}_{J}}{|J|}\Bigg)^{\frac12}\Bigg\|_1 \lesssim  \frac{|Q|}{|J|}  2^{-10k}.
\end{equation}
With this in hand, by virtue of the support condition \eqref{truncsupptreat}, 
$$
\frac{|\l b ,\phi_J\r|}{|J|^{\frac 12}} \leq \sum_{3Q \subset 3 u(K) J}  \frac{|\l b_Q,\phi_J\r|}{|J|^{\frac 12}} \lesssim    2^{-10k} \sum_{3Q \subset 3 u(K) J} \frac{|Q|}{|J|} \lesssim   2^{-4k},
$$
and \eqref{megaesti1} follows by combining the last display with \eqref{megaesti1pf1}, in view of \eqref{eqsplitconcrete}. \end{proof}
\begin{proof}[Proof of \eqref{megaesti}]
Let $\cic{\varphi}_{(I,\xi)}$ be a collection of $u$-adapted wave packets \emph{with mean zero} with respect to $(I,\xi) \in \mathcal T$.  Here, we rely on the   decomposition  \eqref{eqsplitconcrete} of $\varphi_J$, with $\phi_J$, $\psi_J$ satisfying, in addition to  \eqref{truncsupptreat}-\eqref{polyadapt}, property \eqref{meanzerointreat}, so that in particular $\widehat{\psi_J}(\xi)=0$ for all $J$. 
In view of this  decomposition, \eqref{megaesti} will follow from estimating, for all $J_0 \in \D(I), J_0 \not \subset {E_\lambda}$,
\begin{equation} \label{prfmgest}
\Bigg(\frac{1}{|J_0|}  \sum_{  J \not\subset {E_\lambda}, J \subset J_0 } |\l b, \phi_J\r|^2  \Bigg)^{\frac12} \lesssim   2^{-4k}, \qquad\e^{-\frac{a}{12}K} \Bigg(\frac{1}{|J_0|}  \sum_{  J \not\subset {E_\lambda}, J \subset J_0 } |\l b, \psi_J\r|^2  \Bigg)^{\frac12} \lesssim    2^{-4k},
\end{equation}

We first prove the second estimate, which is easier. We have
\begin{align*}
 \e^{-\frac{a}{12}K} \bigg(  \sum_{  J \not\subset {E_\lambda}, J \subset J_0 } |\l b, \psi_J\r|^2  \bigg)^{\frac12} \leq \e^{-\frac{a}{12}K}  \bigg(  \sum_{  J \not\subset \{\M_1 b \gtrsim \lambda 2^{k}\}, J \subset J_0 } |\l b, \psi_J\r|^2  \Bigg)^{\frac12}  \lesssim \e^{-\frac{a}{12}K}   2^{k} |J_0|^\frac12    2^{-4k} |J_0|^\frac12 ,
\end{align*}
 employing \eqref{lambda2k} for the first step and then applying Lemma \ref{thelemmapsi} for the second inequality, in view of \eqref{polyadapt} and of $\widehat{\psi_J}(\xi)=0$ for all $J$.

Once we establish the inequality
\begin{equation} \label{thefinaline}
\Bigg\| \Bigg(  \sum_{  J \not\subset {E_\lambda}, J \subset J_0 } |\l b, \phi_J\r|^2 \frac{\cic{1}_{J}}{|J|}\Bigg)^{\frac12} \Bigg\|_{L^1(J_0)}  \lesssim   2^{-4k}   \qquad \forall J_0 \in \D(I), J_0 \not \subset {E_\lambda},
\end{equation}
the first estimate in \eqref{prfmgest}, which is what is left to show in order  to be done with \eqref{megaesti}, can   be  reached via   the same John-Nirenberg type argument used in the proof of \cite[Proposition 2.4.1]{ThWp}. 
We prove \eqref{thefinaline} via the chain of inequalities
\begin{align*}
& \quad \Bigg\| \Bigg(  \sum_{  J \not\subset {E_\lambda}, J \subset J_0 } |\l b, \phi_J\r|^2 \frac{\cic{1}_{J}}{|J|}\Bigg)^{\frac12} \Bigg\|_{1}    = 
\Bigg\| \Bigg(  \sum_{  J \not\subset {E_\lambda}, J \subset J_0 } \Big|\Big\l  \sum_{Q: 3Q \subset 3u(K) J_0} b_Q, \phi_J\Big\r\Big|^2 \frac{\cic{1}_{J}}{|J|}\Bigg)^{\frac12} \Bigg\|_{1}
 \\ & \leq \sum_{Q: 3Q \subset 3u(K) J_0} \Bigg\| \Bigg(  \sum_{  \substack{ J \not\subset 9Q\\ J \subset J_0} } |\l b_Q, \phi_J\r|^2 \frac{\cic{1}_{J}}{|J|}\Bigg)^{\frac12} \Bigg\|_{1}
\lesssim   2^{-10k} \sum_{Q: 3Q \subset 3u(K) J_0} |Q| \lesssim   2^{-10k} u(K)|J_0| \lesssim 
  2^{-4k} |J_0|, \end{align*}
where \eqref{BQ} has been used, for each $b_Q$, for the third step. This concludes the proof of \eqref{megaesti}, and in turn of Proposition \ref{corMFCZabs}.
\end{proof}

  \section{The model sums for $\Lambda_{\vec{\beta}}$} \label{sect3} This  section is dedicated to the discretization of  the trilinear forms $\Lambda_{\vec \beta}$ into the model sums of  \eqref{modelsums} below.    As in Section \ref{secMFCZ}, to which we refer, our definitions depend on a choice of   function $u$ complying with the assumptions of   Section \ref{sect2}. We keep this dependence implicit in the notation. A concrete choice of $u$ will be made in Section \ref{sectpfthm}.

\subsection{Tiles and $u$-wave packets} \label{sstiles} We call \emph{tile} $t=I_t \times \omega_t$  the cartesian product  of two intervals $I_t,   \omega_t \subset \R $ with $|I_t||\omega_t|=1$.  
We  say that  a Schwartz  function $\varphi_t$ is a  $u$-wave packet   adapted to the tile $t$  if $\varphi_t$ is $u$-adapted to $(I_t,c(\omega_t))$ in the sense of  \eqref{adaptbandlim}   and  in addition   $\mathrm{supp} \, \widehat \varphi_t \subset \omega_t$. 

Referring to  Remark \ref{theexample},    a  $u$-wave packet adapted to the tile $t$ is given by  
\begin{equation} \label{thetileada}
\upsilon_{t}:= \upsilon_{(I_t,c(\omega_t)),\eps}=  \mathrm{Mod}_{\xi}\mathrm{Dil}_{\eps^{-1}|I|}^2 \mathrm{Tr}_{c(I)}
\upsilon.\end{equation}
The $u$-wave packets $\upsilon_t$ will be employed in the constructions of our model sums, with a suitable choice of $\eps$.  Again from Remark \ref{theexample}, we infer that the implicit constants for $\upsilon_t$ in \eqref{adaptbandlim} will depend only on our choice of $u$, and that the adaptation rate $a$ will be a  positive   absolute constant.

\subsection{Tritiles and  model sums} A \emph{tritile} $s$ is a triplet of tiles $s_j=I_{s_j} \times \omega_{s_j}$, $j=1,2,3$  with $I_{s_1}=I_{s_2}=I_{s_3}=:I_s$. A collection of tritiles $s \in \S$ is well-discretized (with constant $R>10$) if the following  properties are satisfied:
\begin{itemize}
\item[$\cdot$] the collections $\I_\S=\{I_s:s \in \S\}$, and  $\Omega_{j,\S}:=\{10 \omega_{s_j}: s \in \S \}$ for $j \in \{1,2,3\}$ form  grids;
\item[$\cdot$] if  $s,s' \in \S$ with $|I_s| < |I_{s'}|$, then $|I_s| \leq R^{-10}|I_{s'}|$ (separation of scales);
\item[$\cdot$] if $s\neq s' \in \S$ are such that  $I_s=I_{s'}$ then  $\omega_{s_j}\cap\omega_{s'_j}=\emptyset$ for each  $j \in\{1,2,3\}$;
\item[$\cdot$] if $s,s' \in \S$  with $|I_{s'}| \leq |I_{s}| $ and $2\omega_{s_j} \cap 2 \omega_{s'_j} \neq \emptyset$ for some $j \in\{1,2,3\}$, there holds \begin{equation} \label{welldisc}
   10 \omega_{s_k} \cap 10 \omega_{s'_k}=\emptyset\; \forall k \neq j, \qquad \;  R \omega_{s_k} \subset R{\omega_{s'_k}}, \; \forall k \in\{1,2,3\}.
\end{equation}
\end{itemize}
For a tritile $s$ and $j=1,2,3$,  we denote by $\upsilon_{s_j}$ the  $s_j$-adapted function   obtained from \eqref{thetileada} with $t=s_j$. 
  Consider the model sums  
\begin{equation} \label{modelsums}
\tr_\S(f_1,f_2,f_3):=
  \sum_{s \in \S} \eps_{s}|I_s|^{-\frac12} \prod_{j=1}^3\l f_j, \upsilon_{s_j} \r, \qquad  |\eps_s| \leq   1,
\end{equation}
where   $\S$ is  an arbitrary finite collection of well-discretized tritiles. 
In the next subsection, we show how   estimates for the trilinear form $\Lambda_{\vec\beta}$ are obtained from the corresponding, uniform over all finite collections of    well-discretized tritiles  $\S$,   bounds for the model sums \eqref{modelsums}. Note that we are allowing for bounded coefficients $\eps_s$, so that the finiteness assumption on  $\S$ can be removed by a standard limiting argument.\label{sswd}

\subsection{Discretization of the trilinear forms $\Lambda_{\vec \beta}$} 
To each triple $(\tau, x,\xi) \in (0,\infty)\times \R^2$, we associate the tile $$t(\tau,x,\xi):=\textstyle \big( x-\frac\tau 2,x+\frac\tau 2\big) \times \big( \xi-\frac{1}{2\tau},\xi+\frac{1}{2\tau}\big).$$
For  a unit vector $\vec \beta  \in \R^3$     with $\vec\beta \cdot (1,1,1)=0$, let $\vec \gamma  \in \R^3$ be the unique unit vector  such that $\vec \gamma,\vec \beta, (1,1,1) $ form a positively oriented orthogonal  basis of $\R^3$.  
Choosing $\eps=2^{-16}$,  writing $\lambda=1+\eps$ for brevity and recalling that $\widehat{\upsilon} (0)$ is positive, 
we quote from \cite[Section 6]{DT} the equality
\begin{align} \label{modelsumseq1}
& \quad \;  \Lambda_{\beta}(f_1,f_2,f_3)   \\ & =
c_1 \int_{\R^3} \lambda^{-\frac\sigma2 }\Big( \prod_{j=1}^3 \big\l f_j, \upsilon_{t(\lambda^\sigma, x,\gamma_j \xi + \beta_j \lambda^{-\sigma})}   \big\r\Big) \, \d \sigma \d x \d \xi +c_2 \int_{\R} \Big(\prod_{j=1}^3 f_j(y)\Big)\,\d y   \nonumber
\end{align}
holding for any three Schwartz functions $f_j$, where $c_1,c_2$ are nonzero constants depending on $\upsilon,\eps$ only. The second summand on the right hand side of \eqref{modelsumseq1} is bounded by H\"older's inequality, so that  estimates on $\Lambda_{\beta}$ can be deduced from   corresponding bounds  on the triple integral. For  triples  $ \cic{m}=(m_\sigma,m_x,m_\xi) \in \mathbb Z^3, \cic{\vartheta}:=(\vartheta_\sigma,\vartheta_x,\vartheta_\xi) \in [0,1)^3$, we define  the tritile $s^{(\cic{m},\cic{\vartheta})}$ by
$$
 s^{(\cic{m},\cic{\vartheta})}_j= \Big(\lambda^{m_\sigma+\vartheta_\sigma}(m_x+\vartheta_x)  \pm \frac{\lambda^{m_\sigma+\vartheta_\sigma}}{2}\Big) \times \Big(\frac{\gamma_j (m_\xi+\vartheta_\xi) +\beta_j} {\lambda^{m_\sigma+\vartheta_\sigma}}  \pm \frac{\lambda^{-(m_\sigma+\vartheta_\sigma)}}{2}\Big), \;\,j=1,2,3.
$$ 
By suitably splitting the integration regions, the integral over $\R^3$ in \eqref{modelsumseq1} is equal to 
\begin{equation}  \label{modelsumseq2}
 \int_{[0,1)^3}  \Bigg(  \sum_{\cic{m } \in \mathbb Z^3} |I_{s^{(\cic{m},\cic{\vartheta})}}|^{-\frac12} \prod_{j=1}^3 \l f_j, \upsilon_{s_j^{(\cic{m},\cic{\vartheta})}} \r \Bigg) \, \d \cic{\vartheta}
\end{equation}
Arguing exactly as in \cite[pp.\ 50-51]{ThWp},  for each fixed $\cic{\vartheta}$,  $\{s^{(\cic{m},\cic{\vartheta})}: \cic{m} \in \mathbb Z^3\} $ can be decomposed into a finite union of well-discretized collections of tritiles, provided that the constant $R>10$ appearing in the definition is chosen large enough, depending on  the distance $\Delta_{\vec \beta}$ from the degenerate case. Therefore, $L^p$ bounds for the triple integral in \eqref{modelsumseq2}, and in turn, for $\Lambda_{\vec\beta}$ follow by averaging the corresponding inequalities for the model sums \eqref{modelsums}.
\begin{remark} \label{modelsumrestr}
By the same token, estimates of the type
$$
\Lambda_{\vec \beta} (f_1,f_2,f_3) \leq Q(\|f_1\|_{\frac{1}{\alpha_1}}, \|f_2\|_{\frac{1}{\alpha_2}}, |F_3|^{\alpha_3})
$$
where $Q$ is a certain positive function of its arguments, holding for possibly restricted $f_1,f_2$, and for all sets $F_3 \subset \R$,  with $f_3 $ restricted to a suitable major subset $F_3' $ of $F_3$, are obtained by averaging the bound  of the same type for \eqref{modelsums}, \emph{provided that $F_3'$ can be chosen independently of the model sum}. This will be the case in the proofs of our theorems.
\end{remark}  

\begin{remark}[Scale invariance of the model sums] \label{sinvariance} Let $\vec \alpha $ be a H\"older triplet. For each $\mu>0$, we have the equality
$$
\tr_{\S}(f_1,f_2,f_3) = \sum_{s \in \S}\eps_s |I_{s^\mu}|^{-\frac12}    \prod_{j=1}^3\l \mathrm{Dil}_{\mu}^{\frac{1}{\alpha_j}}f_j, \upsilon_{(s^\mu)_j}\r= \tr_{\S^\mu}(\mathrm{Dil}_{\mu}^{\frac{1}{\alpha_1}}f_1,\mathrm{Dil}_{\mu}^{\frac{1}{\alpha_2}}f_2,\mathrm{Dil}_{\mu}^{\frac{1}{\alpha_3}}f_3)
$$
where $\S^\mu:=\{s^\mu:s\in \S \}$ and each  tritile $s^\mu$ is given by $I_{s^\mu}= T_\mu I_s$,   $\omega_{(s^\mu)_j}= (T_\mu)^{-1} \omega_{s_j}$. Here,  $T_\mu$ is the linear transformation $x \to \mu x$; note that, in general, $\mu I_s$ and $T_\mu I_s$ are not the same. When $\mu$ is a power of $R$,  the collection $\S^\mu$ is again a well-discretized collection of tritiles according to the definition  of Subsection \ref{sswd}, so that the family of trilinear forms \eqref{modelsums} is invariant under dyadic H\"older-type scaling.   \end{remark}

\section{Trees, size, and single tree estimates} \label{sect4}
We summarize the main definitions and technical tools needed to manufacture bounds for the model sums \eqref{modelsums} in the framework of \cite{LT1}.  Our treatment deviates from the classical one in that we work with model sums involving uniformly  $u$-adapted wave packets $\{\varphi_{s_j}: j=1,2,3\}$, indexed over a generic well-discretized collection of tritiles $s \in \cic{S}$.     
%Throughout, $\S$ will denote a  generic finite well-discretized collection of tritiles.   
%The collections  $\{\varphi_{s_j}: s \in \S, j=1,2,3\}$ such that each $\varphi_{s_j}$ is a wave packet adapted to the tile $s_j$, in the sense of Section \ref{sstiles}.      
  
\subsection{Trees and size} \label{sstrees}
In our context, a \emph{tree of tiles}
 $\cic{t}$ with \emph{top datum} $(I_{\cic{t}}, \xi_{\cic{t}})$ 
 is a finite collection of  tiles   such that $I_t \subset I_{\cic{t}} $, and   $\xi_{\cic{t}} \in R \omega_t$   for each $t \in \cic{t}$. For our scopes, the technical requirement that $\{I_t: t \in \cic{t}\}$ is a grid will be always satisfied.
   If $ \xi_{\cic{t}} \in R\omega_{t} \backslash 2\omega_{t}$ for all $t \in {\cic{t}}$, the tree $\cic{t}$ is called \emph{lacunary}. 
We associate to each lacunary tree $\cic{t}$ and each $f \in L^1 (\R)+L^\infty(\R)$   the  quantity $$
 \size  (f; \cic{t}):= \sup_{\{\varphi_t: t \in \cic{t}\}}\displaystyle\Big(  \frac{1}{|I_{\cic{t}}|} \sum_{t \in  {\cic{t}}}|\l f, \varphi_{t} \r|^2 \Big)^{\frac12},$$
 the supremum being taken over all collections $\{\varphi_t: t \in \cic{t}\}$ of uniformly adapted $u$-wave packets. 
Arguing along the lines of Remark \ref{remdyadfreq},   each such  collection can be written as a union of at most $\lesssim R $ collections of $u$-adapted wave packets with top datum $(I,\xi)$, as defined in Section \ref{secMFCZ}. We can thus reformulate the conclusion of  Lemma \ref{thelemmapsi} into the estimate
\begin{equation} \label{infmax}   \size (f; \cic{t})\lesssim  	\sup_{t \in {\cic{t}}}\inf_{x \in I_t} \M_1 f(x). 
\end{equation} 

We give  related definitions for tritiles. A \emph{tree}  of tritiles   $\T$ \emph{of type} $j \in\{1,2,3\}$  \emph{with top datum} $(I_\T,\xi_\T)$ (in short, $j$-tree) is a  finite  well-discretized collection of tritiles such that  $I_s \subset I_\T$,  $\xi_\T \in 2\omega_{s_j}$ for each  $s \in \T$.
A   consequence of \eqref{welldisc} is that 
if   $\T$ is a $j$-tree, for $k\neq j$ the intervals $\{10 \omega_{s_k}: s \in \T\}$ are pairwise disjoint while   $\{R \omega_{s_k}: s \in \T\}$ are nested. It follows that there exists a frequency $\xi_{\T,k}$ such that $\xi_{\T,k} \in R \omega_{s_k} \backslash 2\omega_{s_k}$ for all $s \in \T$; in other words, for $k \neq j$ the collection $\T^{(k)}:=\{s_k: s \in \T\}$ is a lacunary tree of tiles with top datum $(I_\T, \xi_{\T,k}).$ 

For each finite, well-discretized collection of tritiles $\S$, each $f \in L^1 (\R)+L^\infty(\R)$, and each $j=1,2,3$  we define \begin{equation} \label{sizedef}
\size_j(f;\S) :=    \sup_{\substack{\T \subset \S \textrm{ is a tree} \\ \T^{(j)} \textrm{ is a lacunary tree}  }  } \size (f; \T^{(j)}),
\end{equation}
inheriting from  \eqref{infmax}        the bound
\begin{equation}\label{sizeinfmax}
\size_j(f; \S ) \lesssim \sup_{s \in \S} \inf_{x \in I_s} \M_1 f(x).
\end{equation} 
The quantities $\size_j$ enter    the following  \emph{single tree estimate} for model sums of the type \eqref{modelsums}.
\begin{lemma} \label{singletreeest}  Let $\T  $   be a tree of tritiles.   Then
$$  \sum_{s \in \T} |I_s|^{-\frac12} \prod_{k=1}^3|\l f_k, \varphi_{s_k}\r| \lesssim |I_\T| \prod_{k=1}^ 3 \size_k(f; \T ). $$
\end{lemma}
\begin{proof} If $\T$ is a $j$-tree, the $\ell^\infty\times \ell^2\times \ell^2$ H\"older inequality yields
$$
\sum_{s \in \T} |I_s|^{-\frac12} \prod_{k=1}^3|\l f_k, \varphi_{s_k}\r| \leq \Big(\sup_{s \in \T} \frac{| \l f_j, \varphi_{s_j}\r|}{|I_s|^{\frac12}}\Big)\prod_{k \neq j} \Big(\sum_{s \in \T} | \l f_k, \varphi_{s_k}\r|^2 \Big)^{\frac12}.$$
Viewing $\{s_j\}$ as a lacunary tree, the first factor is   $\lesssim\size_j(f; \T )$.  The second and third factor are each  $\lesssim  \size_k(f; \T )|I_\T|^{\frac12}$, since $\T^{(k)}$ is a lacunary tree for $k\neq j$. This concludes the proof of the lemma. 
\end{proof}

\subsection{Size lemma and a forest estimate}
The following lemma, known as \emph{size lemma}, is used to iteratively decompose a collection $\S$ into subcollections which are unions of trees of definite size (also known as \emph{forests}). There is   abundance of analogous results in the literature: for the proof, we   refer to \cite[Lemma 5.3]{ThWp} and  \cite[Lemma 7.7]{MTT}.
\begin{lemma}
\label{sizelemma}  Let  $\S$ be a finite, well-discretized collection of tritiles and  $f \in L^2 (\R)$ such that
$
\size_j(f ; {\S} 	)\leq \sigma.
$
Then $\S=\S_\mathrm{lo} \cup \S_\mathrm{hi}, $ where
\begin{align} &
\size_j(f;{\S_\mathrm{lo}} 	)\leq \textstyle\frac\sigma2,  \\ 
&  \S_\mathrm{hi} \textrm{\emph{ is a disjoint union of trees }} \T \in \F   \textrm{\emph{ with }}
  \displaystyle\Big(  
\sum_{\T \in \F}|I_\T| \Big)^{\frac 12} \lesssim \frac{  \|f\|_2}{\sigma}.
\end{align}
If in addition $f \in L^\infty(\R)$, for each interval $I \subset \R$ we have
\begin{equation}
\label{BMOest}\frac{1}{|I|} \sum_{\T \in \F: I_\T \subset I} |I_\T| \lesssim   \frac{  \|f\|_\infty^2}{\sigma^2}  .
\end{equation} 
\end{lemma}
It is convenient to combine Lemmata \ref{singletreeest} and \ref{sizelemma} into an estimate for model sums restricted to a union of trees satisfying a certain relation between size and counting functions. This result is analogous to \cite[Lemma 4.4]{DD2}, but we include the proof for convenience.  \begin{lemma} \label{orglemma} Let $f_3 \in L^2(\R)$ be given and    $\S$ be a finite, well-discretized collection of tritiles. Assume that  $\S$ can be written as a disjoint union of     trees $\T \in \F$    satisfying, for some $A>0$,  
\begin{equation} \label{sumtopspop}
\Big(\sum_{\T \in \F} |I_\T|\Big)^{\frac12} \leq\frac{A}{\size_3(f_3;\S)}.
\end{equation}
Then, for all $f_1 \in L^2(\R), f_2 \in L^{1} (\R) + L^{\infty} (\R)  $,
\begin{equation}
\label{forestest}
 \sum_{s \in \S} |I_s|^{-\frac12} \prod_{j=1}^3|\l f_j, \varphi_{s_j}\r| \lesssim A   \|f_{ 1}\|_2
 \size_{2} (f_{2}; {\S} ).
\end{equation}
\end{lemma}
\begin{proof}Denote $\sigma_j=\size_{j} (f_{j}; {\S} ) $ for $j=1,2,3$. By linearity in $f_1$, $f_2$ we can assume $\|f_{ 1}\|_2=1,\sigma_2=1$.  Let $n_0= \lceil\log \sigma_1 -\log \sigma_3 + \log A \rceil$. There are two cases: if $n_0 \leq 0$, in other words $A\sigma_1 \leq \sigma_3$, the left hand side of \eqref{forestest} is bounded, using Lemma \ref{singletreeest}, by
 $$
 \sum_{\T \in \F} \sum_{s \in \T} |I_s|^{-\frac12} \prod_{j=1}^3|\l f_j, \varphi_{s_j}\r| \lesssim  \prod_{j=1}^3 \size_{j} (f_{j}; {\S} )   \sum_{\T \in \F} |I_\T| \lesssim A^2\sigma_1   \sigma_3^{-1}  \leq A,
 $$ 
 which is what we had to prove. Otherwise, we   decompose $\S$ into collections $\S_n, n=0, \ldots, n_0$ each being a disjoint union of trees $\T \in \F_{n}$ such that
 $$
 \sum_{\T \in \F_n} |I_\T| \lesssim 2^{2n} (\sigma_1)^{-2} \qquad  \size_{1} (f_{1}; {\S_n} ) \leq 2^{-n} \sigma_1, \,
 \size_{2} (f_{2}; {\S_n} ) \leq 1, \,   \size_{3} (f_{3}; {\S_n} ) \leq \sigma_3.
 $$
 by iteratively applying the size lemma with $f=f_1$ for $n=0, \ldots, n_0-1$, and by organizing the leftover collection $\S_{n_0}$ into a disjoint union of trees $\F_{n_0}:=\{\T\cap \S_{n_0}: \T \in \F\}$. For this last collection, the first bound in the last display is   inherited from \eqref{sumtopspop}.
 Applying again the single tree estimate for each $\T \in \F_n$, the left hand side of \eqref{forestest} is bounded by
 $$
 \sum_{n=0}^{n_0}\sum_{\T \in \F_n} \sum_{s \in \T} |I_s|^{-\frac12} \prod_{j=1}^3|\l f_j, \varphi_{s_j}\r| \lesssim \sum_{n=0}^{n_0} \Big( \prod_{j=1}^3 \size_{j} (f_{j}; {\S_n} )   \sum_{\T \in \F_n} |I_\T| \Big)\lesssim \sum_{n=0}^{n_0} 2^{n} \textstyle\frac{\sigma_3}{ \sigma_1} \lesssim A,
 $$
which finishes the proof. 
\end{proof}

\section{Forest estimates} \label{secforest}
  
Let $\S$ be a finite, well-discretized collection of tritiles and   $\{\varphi_{s_j}:  s \in \S, j=1,2,3 \}$ be a collection of uniformly  $u$-adapted wave packets. 
We will provide several estimates on $$\tr_\S (f_1,f_2,f_3)= \sum_{s \in \S} \eps_s|I_s|^{-\frac12} \prod_{j=1}^3 \l f,\varphi_{s_j} \r$$  when $f_3$ is bounded by $1$ and supported on   $F_3 \subset \R $ of finite measure, or on a suitable major subset of $F_3$. Note that the model sum \eqref{modelsums} is a particular instance of the above display. In what follows, for any $\S'\subset \S$, we are indicating with  $\tr_{\S'}$ that we are summing over $s \in \S'$.

 \subsection{Estimates inside exceptional sets} The first estimate deals with the case of $\S$ being localized inside the superlevel sets of the maximal functions of ${h_1},{h_2}$.  
\begin{proposition} \label{excsetprop} Let $\vec \alpha =(\alpha_1,\alpha_2,\alpha_3)$ be a H\"older tuple with $0\leq \alpha_1,\alpha_2 \leq 1, \alpha_3 \geq -\frac12$.
For  functions $h_j \in L^{\frac{1}{\alpha_j}}(\R)$, $j=1,2$,  and a set $F_3\subset \R$ of finite measure, define
\begin{align} \label{esets} &
E^{\vec \alpha}_{{h_1},{h_2},F_3} := \bigcup_{j=1,2} \Big\{\M_{\frac{1}{\alpha_j}} {h_j}  > \textstyle C \frac{\| {h_j}\|_{1/\alpha_j}}{|F_3|^{\alpha_j}}   \Big\}, \; {\bar E}^{\vec \alpha}_{{h_1},{h_2},F_3}= \cup\big\{3Q: Q \textrm{\emph{ max.\ dyad.\ int.\ }} \subset E^{\vec \alpha}_{{h_1},{h_2},F_3}  \big\} \\ &  F_3(\vec\alpha,{h_1},{h_2}):= F_3 \backslash {\bar E}^{\vec \alpha}_{{h_1},{h_2},F_3}, \label {excset}
\end{align}
with     $C$   chosen large enough so that $|F_3|\leq 4|F_3(\vec\alpha,{h_1},{h_2})| $. Assume  that 
$I_s\subset {  E}^{\vec \alpha}_{{h_1},{h_2},F_3} $ for all $s \in  \S$. Then, for all functions $|f_1|\leq |h_1|, |f_2|\leq |h_2|$
$$
\tr_\S({f_1},{f_2},f_3) \lesssim   \| {h_{ 1}}\|_{ \frac{1}{\alpha_1}} \|{{h_2}}\|_{ \frac{1}{\alpha_2}}
 |F_3|^{\alpha_3} \qquad \forall\, |f_3| \leq \cic{1}_{F_3(\vec\alpha,{h_1},{h_2})}. 
$$
\end{proposition}
\begin{proof} We will rely on the following estimate,
which is proved in the same way as, for instance,  \cite[Lemma 3.1]{BG}.  For an interval $J$,  $A>1$,     $|f_3|\leq \cic{1}_{\R \backslash AJ}$  there holds
$$
\tr_{\S(J)}({f_1},{f_2},f_3)\lesssim A^{-100} |J| \Big(\inf_{x \in J} \M_{\frac{1}{\alpha_1}}{f_1} (x)\Big) \Big(\inf_{x \in J} \M_{\frac{1}{\alpha_2}}{f_2} (x) \Big).
$$
where $\S(J)=\{s \in \S: I_s=J\}$.
Now, for each    interval  $J$, let $k(J)$ be the minimal  integer such that $2^{k+1} J \not\subset E_{{h_1},{h_2},F_3}$. Then one can take $A=2^k$ in the above estimate, and, since $|f_j|\leq |h_j|$,  \begin{align*}
&\quad \sum_{J:k(J)=k } \tr_{\S(J)}({f_1},{f_2},f_3) \lesssim 2^{-100k}\sum_{J:k(J)=k } |J| \Big(\inf_{x \in J} \M_{\frac{1}{\alpha_1}}{h_1} (x)\Big) \Big(\inf_{x \in J} \M_{\frac{1}{\alpha_2}}{h_2} (x) \Big) \\ & \lesssim 2^{-98k} \|h_{ 1}\|_{ \frac{1}{\alpha_1}} \|{h_2}\|_{ \frac{1}{\alpha_2}} |F_3|^{-(\alpha_1+\alpha_2)} \sum_{J:k(J)=k } |J| \lesssim 2^{-97k} \|h_{ 1}\|_{ \frac{1}{\alpha_1}} \|{h_2}\|_{ \frac{1}{\alpha_2}} |F_3|^{\alpha_3}
\end{align*}
since the intervals $\{J:k(J)=k\}$ have at most $2^{k+1}$ overlap and are contained in ${  E}^{\vec \alpha}_{{h_1},{h_2},F_3}$. The proof of the lemma is then finished by summing up over $k$.\end{proof}
%Analogous versions of Lemmata \ref{singletreeest} to \ref{orglemma} hold for collections of wave packets $\{\varphi_{s_j}:j=1,2,3\}$ which are simply adapted to the tile $s_j$, namely with  \eqref{polyadapt} in place of the stronger \eqref{adaptbandlim} in the definitions of Subsection \ref{sstiles}. The final result of this section, combining Lemma \ref{orglemma} with the multi-frequency decomposition of Proposition \ref{corMFCZabs}, relies instead on the full strength of \eqref{adaptbandlim}.
\subsection{The $f_3$-decomposition of $\S$ and forest estimates}  \label{f3dec} Throughout this subsection, fix a  set  $F_3 \subset \R$  and $|f_3| \leq \cic{1}_{F_3}$. In view on the dyadic scaling invariance of the family of model sums, see Remark \ref{sinvariance}, we lose no generality by working with $|F_3|\sim 1$ in what follows.
 Note that  any finite collection  $\S$ admits the decomposition 
\begin{align}
 & \S = \bigcup_{k =0,1,\ldots} \S_k,  
\\ & \size_3(f_3;\S_k) \lesssim 2^{-k},  \label{sizebdSK}
\\ & \S_k= \bigcup_{\T \in \F_k} \T, \; \textrm{ each } \T \textrm{ is a tree}, \qquad 
 \sum_{\T \in \F} |I_\T|  \lesssim 2^{2k}  |F_3| \sim 2^{2k} , \label{topsSK}
 \\ &  \sup_{I\subset \R}\frac{1}{|I|} \sum_{\T \in \F_k: I_\T \subset I} |I_\T| \lesssim   2^{2k}, \label{bmoSK}
\end{align} 
which we call the $f_3$-decomposition of $\S$ into forests (unions of trees) $\S_k$. The decomposition is obtained by iteratively applying the size  lemma with $f=f_3$, starting from  $\sigma= \size_3(f_3;\S) \lesssim 1 $. Since  $\S$ is finite, the iteration terminates in   finitely many steps.

The next two propositions provide estimates for the model sums restricted to $\S_k$, when either $f_2 \in L^2(\R)$ or $|{f}_2|\leq\cic{1}_{F_2}$ for  some $F_2\subset\R$ of finite measure. 
 To unify notation, we write ${h_2}=f_2$ if $f_2 \in L^2(\R)$ is unrestricted and ${h_2}=\cic{1}_{F_2}$ if $f_2$ is restricted  to $F_2$. \begin{proposition} \label{propL2SK}
Let $f_1 \in L^{2}(\R)$ be given  and $f_2,F_3,f_3$ as above.  Assume that $\S_k$ satisfies \eqref{sizebdSK}-\eqref{topsSK} and that in addition
$$
\size_{2} (f_2; \S_k) \lesssim 2^{-n_0}  \|{h_2}\|_2 , 
$$
for some $n_0 \geq 0$. Then
$$
|\tr_{\S_k}(f_1,f_2,f_3) |\lesssim 2^{-n_0} \min\Big\{1,(k-n_0)2^{-(k-n_0)}\Big\} \|f_1\|_2  \|{h_2}\|_2.
$$
 \end{proposition}
 \begin{proof} By linearity of assumptions and conclusions in $f_1$ we can assume $\|f_1\|_2=1$.
We split the proof into two cases, the first being  when $k \leq n_0$. In this case, we straightforwardly apply Lemma \ref{orglemma}, whose assumptions are satisfied in view of \eqref{sizebdSK}-\eqref{topsSK}, to the triple $(f_1,f_2,f_3)$ with $A=|F_3|^{\frac12}\sim 1$. This yields
\begin{equation} \label{L2SK-1}
|\tr_{\S_k}(f_1,f_2,f_3) | \lesssim  |F_3|^{\frac12}\size_{2} (f_2; \S_k)  \lesssim 2^{-n_0} \|{h_2}\|_2\end{equation}
which is what we had to prove.
Let us deal with the case $k >n_0$. We begin by   decomposing $\S$ into collections $\S_{k,n}, n=  n_0,\ldots, k$ each being a disjoint union of trees $\T \in \F_{k,n}$ such that
 \begin{equation}\label{L2SK-2}
 \sum_{\T \in \F_{k,n}} |I_\T| \lesssim 2^{2n}  \qquad  \size_{2} (f_{2}; {\S_{k,n}} ) \leq 2^{-n}  \|{h_2}\|_2    ,  \,
   \size_{3} (f_{3}; {\S_{k,n}} ) \leq 2^{-k},
 \end{equation}
 by iteratively applying the size Lemma \ref{sizelemma} to $f=f_2$ and by organizing the leftover collection $\S_{k,k}$ into a disjoint union of trees $\F_{k,k}:=\{\T\cap \S_{k,k}: \T \in \F_k\}$.
We are now allowed to apply Lemma \ref{orglemma} to the collections $\S_{k,n}$ with the roles of $f_2,f_3$ interchanged and with $A=\|{h_2}\|_2$, so that
$$
|\tr_{\S_{k}}(f_1,f_2,f_3) | \leq \sum_{n=n_0}^k
|\tr_{\S_{k,n}}(f_1,f_2,f_3) | \lesssim (k-n_0) |\|{h_2}\|_2 \size_{3} (f_3; \S_{k,n})  \lesssim (k-n_0)2^{-k}\|{h_2}\|_2.
$$
Putting together \eqref{L2SK-1} with the last display, the proposition is proved.
 \end{proof}  
 This proposition is a version of the previous one,  differing in  that  the function  $f_1$ is not locally $L^2$. The proof makes use of the multi-frequency decomposition of Proposition \ref{corMFCZabs}.
\begin{proposition} \label{proptouse}  
Let $f_1\in L^{\frac{1}{\alpha_1}}(\R)$,  $1/2<\alpha_1\leq 1$, be given and $f_2,F_3,f_3\subset \R$ as above.  Assume that $\S_k$ satisfies \eqref{sizebdSK}-\eqref{bmoSK} and that in addition
   \begin{align} &I_s \not \subset \big\{\M_{\frac{1}{\alpha_1}} f_1 (x) \gtrsim      \|f_1\|_{\frac{1}{\alpha_1}}\big\}\qquad \forall  s \in \S_k, 
\label{smallmf}
\\ &\size_{2} (f_2; \S_k) \lesssim 2^{-n_0}  \|{h_2}\|_2 \label{size2LpSk}
\end{align}
for some $n_0\geq 0$. Then, 
$$
|\tr_{\S_{k}}(f_1,f_2,f_3)| \lesssim   \|f_1\|_{\frac{1}{\alpha_1}}\|{h_2}\|_2   \big(  u(Ck)^2 \log u(Ck)  \big)^{\alpha_1-\frac12}   \begin{cases}2^{-n_0}2^{ (2\alpha_1-1) k} & k \leq n_0, \\  \textstyle\frac{1}{2\alpha_1-1} 2^{- 2(1-\alpha_1) k}  & k>n_0.\end{cases}$$
\end{proposition}
 \begin{proof}
 By linearity, we can assume $\|f_1\|_{1/\alpha_1}=1$. 
 We claim  that $\S_k=\S_k' \cup \S_k'' $, respectively written as a disjoint union of trees from the collections $\F_k',\F_k''$ satisfying
 \begin{align}\label{effesec}
 & \sum_{\T \in \F_k''} |I_\T| \lesssim 2^{-10k}  ,
 \\ \label{effepri}
 & \Big\| \sum_{\T \in \F_k'} \cic{1}_{3u(Ck) I_\T} \Big\|_\infty \lesssim 2^{5k}. 
 \end{align}
The proof of the claim is standard but technical, and we postpone it to the end of the section. Accordingly,   we split $\tr_{\S_{k}}= \tr_{\S_{k}'}+ \tr_{\S_{k}''}$ and estimate each summand separately. 

The summand involving  $\S_k''$ is an error term. Relying on the tree estimate of Lemma \ref{singletreeest}, we   estimate
\begin{equation}
\label{bmolemmapf1}
\prod_{j=1}^3 \size_{j} (f_{j}; {\S_k''} )   \sum_{\T \in \F_k''} |I_\T| \lesssim   2^{-n_0} 2^{-10k} \|{h_2}\|_2  
 \end{equation}
We have relied on \eqref{smallmf} and   inequality \eqref{sizeinfmax} to obtain that $\size_1(f_1;\S) \lesssim 1$, on assumption \eqref{size2LpSk}, and later employed \eqref{effesec}.

We turn to the $\tr_{\S_{k}'}$ summand, and first deal with the case $k>n_0$, which is the  harder one.
The first step consists again of   decomposing $\S_k'$ into collections $\S_{k,n}', n=  n_0,\ldots, k$ each being a disjoint union of trees $\T \in \F_{k,n}'$ such that
 \begin{equation} \label{LpSkn}
 \sum_{\T \in \F_{k,n}'} |I_\T| \lesssim 2^{2n}  ,  \qquad \size_{2} (f_{2}; {\S_{k,n}'} ) \leq 2^{-n}  \|{h_2}\|_2  ,  \,
   \size_{3} (f_{3}; {\S_{k,n}'} ) \leq 2^{-k},
\end{equation}
and in addition
\begin{equation} \label{LpSknbdinfty}
 \Big\| \sum_{\T \in \F_{k,n}'} \cic{1}_{3u(Ck) I_\T} \Big\|_\infty \lesssim 2^{5k}.
\end{equation}
The three properties of \eqref{LpSkn} are obtained by the same argument used in Proposition \ref{propL2SK} for \eqref{L2SK-2}, while \eqref{LpSknbdinfty} is carried over from \eqref{effepri} for  $\F_{k}'$ by means of a reshuffling argument. Details are given at the end of the section.
    The next step is the definition of a  set of top data $\mathcal T$ which is suitable for Proposition \ref{corMFCZabs}.  Noting that, for each 1-tree [resp.\ $j$-tree, $j \neq 1$] $\T \in \F'_{k,n}$,   $\{\varphi_{s_1}: s \in \T\}$ is  a collection of $u$-adapted wave packets [resp.\  $u$-adapted wave packets with mean zero] with respect to the top datum $(I,\xi_\T)$ [resp. $(I,\xi_{\T,1})$], according to the terminology of Section \ref{secMFCZ}, we are led to define
$$
\mathcal T:=\{(I_\T, \xi_\T): \T \in \F_{k,n}' \textrm{ is a 1-tree} \} \cup \{(I_\T, \xi_{\T,1}): \T \in \F_{k,n}' \textrm{ is a } j \textrm{-tree, } j \neq 1 \}. 
$$
With this definition,   by virtue of \eqref{LpSknbdinfty}, we may appeal to
  Proposition \ref{corMFCZabs}  with $p=1/\alpha_1$, $f=f_1$, $\lambda \sim 1=\|f_1\|_{1/\alpha_1}$  and $k$ replaced by $5k$. Writing $ A_k= u(Ck)^2 \log u(Ck)$ for brevity,  we obtain the decomposition
$f_1=g_n+b_n$, with
\begin{align} \label{gl2fine}   &\|g_n\|_2\lesssim   \Big(  A_k \sum_{\T \in \F_{k,n}'} |I_\T|\Big)^{\alpha_1-\frac12} \lesssim  2^{(2\alpha_1-1)n}\big(  A_k \big)^{\alpha_1-\frac12}   , 
\\ \label{expoverlac}&
\sup_{\T \in \F_{k,n}' \textrm{1-tree}} \sup_{s \in \T} \frac{|\l b_n, \varphi_{s_1}\r|}{|I_s|^{\frac12}} \lesssim    2^{-16k},\quad 
  \sup_{\T \in  \F_{k,n}' \textrm{j-tree}}  \size (b_n; \T^{(1)}) \lesssim     2^{-16k} 
\end{align}  
We used conclusion \eqref{gl2again} of Proposition \ref{corMFCZabs} for the first bound of \eqref{gl2fine}, and \eqref{LpSkn} for the second step, while
the inequalities of \eqref{expoverlac} follow respectively from conclusions \eqref{megaesti1} and \eqref{megaesti}.
Repeating the proof of the tree Lemma \ref{singletreeest} and using \eqref{expoverlac} yields the estimate
\begin{align}
\label{boundforb} & \quad| \tr_{\S_{k,n}'}(b_n,f_2,f_3) |\leq  
\sum_{ \T \in \F_{k,n}'}\sum_{s \in \T} |I_s|^{-\frac12} |\l b, \varphi_{s_1}\r|\prod_{j=2}^3|\l f_j, \varphi_{s_j}\r| \\ &\lesssim     2^{-16k }  \prod_{j=2,3} \size_j(f_j; \S_{k,n}) \Big( \sum_{\T \in  \F_{k,n}'} |I_T| \Big)    \lesssim 2^{-n_0} 2^{-14k} \| {h_2}\|_2      .  \nonumber\end{align}
Now, in view of \eqref{LpSkn}, we can  apply Lemma \ref{orglemma} to   $\S_{k,n}'$, with tuple $(g_n,f_3,f_2)$ and $A=\|{h_2}\|_2$. Note that the roles of $f_3$ and $f_2$ are interchanged. This leads to
\begin{align}
\label{boundforg} &  \quad
| \tr_{\S_{k,n}'}(g_n,f_2,f_3)| \lesssim \|g_n\|_2 \| {h_2}\|_2  \size_{3}(f_3;\S_{k,n})\lesssim   2^{(2\alpha_1-1)n} \big(  A_k \big)^{\alpha_1-\frac12 } \| {h_2}\| 2^{-k}    .  \end{align}
Note that the last right hand side of \eqref{boundforb} is always smaller than  the second member of \eqref{boundforg}. Therefore, we estimate
 \begin{align*}| \tr_{\S_{k}'}(f_1,f_2,f_3) | & \leq \sum_{n=n_0}^k | \tr_{\S_{k}'}(f_1,f_2,f_3) | \leq \sum_{n=n_0}^k\big( | \tr_{\S_{k}'}(b_n,f_2,f_3) | +   | \tr_{\S_{k}'}(g_n,f_2,f_3) |\big) \\ & \lesssim \| {h_2}\|_2  \big(  A_k \big)^{\alpha_1-\frac12}    2^{-k} \sum_{n=n_0}^k 2^{(2\alpha_1-1)n} \lesssim      \| {h_2}\|_2    \big(  A_k \big)^{\alpha_1-\frac12}  \textstyle \frac{ 1}{{2\alpha_1-1}} 2^{-2k(1-\alpha_1)}.
 \end{align*}
Collecting \eqref{bmolemmapf1} and the last display, we   have proved the required  estimate when $k>n_0$.

In the case  $k<n_0$, there is no need for the additional decomposition of $\S_k'$. We appeal directly to Proposition \ref{corMFCZabs} along the same lines as above, this time using the trees of $ \F_k'$ as our top data, and obtain a  decomposition   
$f_1=g+b$, with
\begin{align*}   &\|g\|_2\lesssim    \Big(  A_k \sum_{\T \in \F_{k,}'} |I_\T|\Big)^{\alpha_1-\frac12} \lesssim  2^{(2\alpha_1-1)k}\big(  A_k \big)^{\alpha_1-\frac12}    , 
\\  &
\sup_{\T \in \F_{k}' \textrm{1-tree}} \sup_{s \in \T} \frac{|\l b , \varphi_{s_1}\r|}{|I_s|^{\frac12}} \lesssim   2^{-16k},\quad 
  \sup_{\T \in  \F_{k}' \textrm{j-tree}}  \size (b ; \T^{(1)}) \lesssim     2^{-16k}. 
\end{align*} 
We then apply Lemma \ref{orglemma} to $\S_{k}'$ with  tuple $(g,f_2,f_3)$ and  $A=|F_3|^{\frac12}\sim 1$, yielding 
$$
| \tr_{\S_{k }'}(g,f_2,f_3)| \lesssim \|g\|_2  \size_{2}(f_2;\S_{k}') \lesssim  2^{-n_0} 2^{(2\alpha_1-1)k}  \big(  A_k \big)^{\alpha_1-\frac12}   \| {h_2}\|_2 .  $$
Arguing exactly as in the previous case, one sees that the $\tr_{\S_k'}(b,f_2,f_3)$ summand and the error term \eqref{bmolemmapf1} are again smaller than the right hand side of the above estimate. This completes  the proof of the proposition. 
 \end{proof}
 \begin{proof}[Proof of the decomposition \eqref{effesec}-\eqref{effepri}] We  begin with some notation and preliminaries. We write
 $
 \mu:= 3u(Ck)$ and $   J_\T:= \mu I_\T
 $
 for $\T \in \F_k$. Note that $\mu \lesssim 2^k$. 
Furthermore, for any $\G \subset \F_k$, and for each interval $J \subset \R$, we denote
 $$
 N_{\G} (x)= \sum_{\T \in \G} \cic{1}_{J_\T} (x), \qquad N_{\G',J}(x) = \sum_{\substack{\T \in \G: J \subsetneq  J_\T}}  \cic{1}_{J_\T} (x).
 $$
 
We claim that \eqref{effesec}-\eqref{effepri} will follow    once we show that any    $\G \subset \F_k$ such that the dilated intervals $\{J= J_\T : \T \in \G \}$  belong to a fixed   grid $\D$ 
admits the decomposition
 \begin{equation} \label{claimpf1}
 \G=\G' \cup  \G'', \qquad \|N_{\G'}\|_{\infty} \lesssim 2^{4k}, \qquad \|N_{\G''}\|_{1} \lesssim 2^{-100 k} . \end{equation} 
The claim simply follows by   decomposing $\F_k$   into $\lesssim \mu \log \mu \lesssim 2^k $ such   $\G$'s, which is possible since the intervals $\{I= I_\T : \T \in \F_k\}$ belong to a finite union of dyadic grids.

 We begin the proof of \eqref{claimpf1}, fixing one such $\G$.
 Let $J \in \cic{J}^0$ be the collection of maximal intervals of $\{J= J_\T : \T \in \G \} \subset \D$. We inherit from  \eqref{topsSK} the inequality
\begin{equation} \label{claimpf0}
\sum_{J \in \cic{J}_0} |J| \leq \mu \sum_{\T \in \F_k} |I_\T| \lesssim \mu 2^{2k}  \lesssim 2^{3k}  .
\end{equation} 
  Moreover, a consequence of \eqref{bmoSK} is that
 $$
  \sum_{\T \in \G: J_\T \subset J} |J_\T| \lesssim   \mu 2^{2k}|J| \lesssim 2^{3k}|J|, \qquad \forall J \subset \R. $$  Observing that for each $J \in \D$ $N_{\G,J}(x)=N_{\G,J}$ is constant on $J$, the last display implies that    
 \begin{equation}
 \label{claimpf2}
 \big|\{x \in J: N_{\G}(x)- N_{\G,J}  > C \lambda 2^{3k}\}\big| \leq  2^{-\lambda} |J|, \qquad \forall \lambda>0,
 \end{equation}
if $J \in \D$ and  the constant $C$ is chosen large enough; this is John-Nirenberg's inequality.

 We now construct $\G',\G''$. The set $\J^1=\{N_\G >C k2^{3k} \}$ is  the union of its  maximal intervals $J \in  \D$. We call $\cic{J}^1$ the collection of such intervals. Setting $\G':=\{\T \in \G: J_\T \not \subset \J^1\}$, $\G'':= \G \backslash \G'$, it is easy to see that
\begin{equation}
\|N_{\G'}\|_\infty \lesssim k 2^{3k} \lesssim 2^{4k}, \qquad \supp\ N_{\G''} \subset \J^1.  \label{claimpf3}
\end{equation}
Furthermore, using \eqref{claimpf2} in the second step and \eqref{claimpf0} for the final inequality, we have the  estimate 
\begin{equation} \label{claimpf6}
|\J^1| = \sum_{J \in \cic{J}^0} |\{x \in J: N_\G (x) >Ck2^{3k}\}| \leq  \sum_{J \in \cic{J}^0} 2^{-400k} |J| \lesssim 2^{-200 k}  .
\end{equation} 
We will show that $\G''$ satisfies \eqref{claimpf1} by means of an iterative procedure. Assume that, at the $j$-th step, we have written $\G''= \G''_{\textrm{now}} \cup \G_{\textrm{stock}},$ where 
$
\|N_{ \G''_{\textrm{now}} }\|_1 \lesssim 2^{-100k} ,$
and $N_{\G_{\textrm{stock}}}$ is supported on the set $\J^j$, which is a union of disjoint intervals $J \in \cic{J}^j \subset \D$  such that $|\J^{j}| \lesssim  2^{-200k j }  $. In \eqref{claimpf6}, we have  the base case $j=1$, with $\G''_{\textrm{now}}=\emptyset$, $\G_{\textrm{stock}}=\G''$. The $(j+1)$-th inductive step is as follows. We  
define $\J^{j+1}:=\{N_{\G_{\textrm{stock}}} > C k 2^{3k}\}$, which is a union of  maximal   intervals $J' \in \cic{J}^{j+1}\subset\D$. Setting $\G_*:=\{\T \in \G_{\textrm{stock}}: J_\T \not \subset \J^{j+1}\}$, we observe that  
$$
\|N_{\G_*}\|_\infty \lesssim k 2^{3k} \lesssim 2^{4k}\label{claimpf5}
$$
so that
$$
\|N_{\G_*}\|_1  \leq \|N_{\G_*}\|_\infty |\supp \, N_{\G_*} | \lesssim 2^{4k}   |\J^{j}| \lesssim 2^{-100k j } ,
$$
by the inductive assumption on $\J^j$. Also, relying on  \eqref{claimpf2} to pass to the second line,
\begin{align*} 
|\J^{j+1}| &\leq \sum_{J \in \J^j} \big|\{ x \in J: N_{\G_{\textrm{stock}}}(x)   > C k 2^{3k}\}\big|  \leq   \sum_{J \in \J^j} \big|\{ x \in J: N_{\G}(x) - N_{\G,J}   > C k 2^{3k}\}\big| \\ & \lesssim 2^{-400k} \sum_{J \in \J^j} |J| \lesssim 2^{-200k(j+1)} .
\end{align*}
The inductive step is completed by updating $\G''_{\textrm{now}}:=\G''_{\textrm{now}} \cup \G_*, \G_{\textrm{stock}}:=\G_{\textrm{stock}}\backslash \G_*$. We iterate  until $\G_{\textrm{stock}}$ is empty, which   happens after finitely many steps, since $\G$ is a finite collection. At this point, $\G''= \G''_{\textrm{now}} $ satisfies \eqref{claimpf1}. This completes the proof of the claim.
 \end{proof}
 \begin{proof}[Details of the construction \eqref{LpSkn}-\eqref{LpSknbdinfty}]
The same argument employed in Proposition \ref{propL2SK} for \eqref{L2SK-2} yields the decomposition of $\S'_k$ into subcollections $\S'_{k,n}$, $n=n_0,\ldots, k-1$,  each partitioned into a union of trees $\T \in \G_{k,n}$ satisfying   \eqref{LpSkn} with $\G_{k,n}$ in place of $\F'_{k,n}$.
The remaining collection $\S'_{k,k}:=\S'_k \backslash (\S'_{k,n_0} \cup \ldots \cup \S'_{k,k-1})$, which has $\size_{2}(f_2;\S'_{k,k})\lesssim 2^{-k}\|{h_2}\|_2 |F_3|^{-\frac12}$,  is partitioned into trees by  $\F'_{k,k}=\{\T':=\T \cap \S'_{k,k}: \T \in \F'_k\}$,  and the remaining claims of  \eqref{LpSkn}-\eqref{LpSknbdinfty} are immediately inherited from \eqref{sizebdSK}, \eqref{topsSK}, and \eqref{effepri}.

 We now show how to construct a   new partition $\F'_{k,n}$ of $\S'_{k,n}$ inheriting \eqref{LpSknbdinfty} from $\F'_k$ as well as retaining \eqref{LpSkn}. 
By   partitioning $\S'_{k,n}$,  $\G_{k,n}$ into three subcollections, we can reduce to the case where   all trees of $\G_{k,n}$ are $1$-trees. 
Let $ \mathbf{tops} $ be the collection of maximal tritiles in $\S_{k,n}'$ with respect to the following order relation:  $
s \ll_1 s'
$ 
when $I_s \subset I_{s'}$ and $2\omega_{s'_1} \subset 2\omega_{s_1}$.
Note that the boxes $\{I_s \times 2\omega_{s_1}: s \in \mathbf{tops}\}$ are pairwise disjoint. For each $s \in \mathbf{tops}$, form the tree $\T(s)=\{s\}$ with top data $(I_\T, \xi_T)=(I_s,c(\omega_{s_1}))$. Now,   each $s' \in \S'_{k,n}$  is added to $\T(\bar s)$ where $\bar s$ is the tritile     with minimal $c (\omega_{s_1})$  among those $s \in \mathbf{tops}$  with $s'\ll_1 s$. We call $\F'_{k,n}:=\{\T=\T(s): s \in \mathbf{tops}\}$ the resulting partition of $\S_{k,n}'$.

To prove  \eqref{LpSkn}  for $\F_{k,n}'$,  recall that each tritile $s \in \mathbf{tops}$ belonged to a unique  tree   $\tilde \T(s) \in \G_{k,n}$. Observing that $\{\omega_{s_1}: s  \in \mathbf{tops}, \tilde \T(s)=\tilde \T \}$ have nonempty intersection, the intervals  $\{I_s: s  \in \mathbf{tops}, \tilde \T(s)=\tilde \T \}$, all contained in $I_{\tilde \T}$, must be pairwise disjoint. Hence,  
$$
\sum_{\T \in \F'_{k,n}} |I_\T| =\sum_{s  \in \mathbf{tops} } |I_s| = \sum_{\tilde \T \in\G_{k,n} } \sum_{s  \in \mathbf{tops} : \tilde \T=\tilde \T (s)} |I_s| \leq  \sum_{\tilde \T \in \G_{k,n}} |I_{\tilde\T}| \lesssim  2^{2n} , 
$$
and we have verified \eqref{LpSkn} for $\F_{k,n}'$. The argument for \eqref{LpSknbdinfty} is similar, with trees from the forest $\F_{k}'$ playing the role of the $\tilde \T$'s above.  This concludes our decomposition.
 \end{proof}

\section{Proofs  of the   main results} \label{sectpfthm}
\subsection{Proofs of Theorems \ref{propL1L2ag} to     \ref{thelinfversion}}
We will obtain our restricted type  estimates on $\Lambda_{\vec \beta}$ via the reduction to the  model sums \eqref{modelsums}, in particular, relying on Remark \ref{modelsumrestr}.
At this time, we make our choice of generating function $u$, and, consequently, of our mother function $\upsilon$ in the definition \eqref{thetileada} of $\upsilon_{s_j}$, taking $u:=u_1$ from the family \eqref{familyu}. Any other choice of the parameter $\lambda>0$  in \eqref{familyu} is  legal throughout our arguments. We invite the willing reader to check that alternative choices of $\lambda$ (or of $u$ altogether) do not bring  essential improvements to the estimate of Theorems \ref{prLpL2} and \ref{thelinfversion}, and bring no improvements at all to Corollary \ref{rwtTH}.

   Therefore, Theorems   \ref{propL1L2ag} to \ref{thelinfversion} will respectively follow  from the corresponding versions for the model sums $\tr_\S$ below.   We stress that the implicit constants appearing in the statements are  uniform  over all finite well-discretized collections of tritiles $\S$,  and the major set  $F_3'$ is explicitly chosen independently of $\S$. 
   \begin{theorem1s}    
Let $\vec\alpha=(\frac12,1,-\frac12)$. For\footnote{Note that, to unify notations in the proofs below, we have switched herein the role of the first and second argument with respect to the statement of Theorem \ref{propL1L2ag}.} $f_1 \in L^2(\R)$,  $|f_2|\leq \cic{1}_{F_2} $,   and $F_3 \subset \R $ of finite measure,  let $F_3'$ be the major subset of $F_3$ defined via \eqref{excset} by
$
F_3':= F_3(\vec \alpha, f_1,\cic{1}_{F_2})
$. Then, for all $   |f_3|\leq \cic{1}_{F_3'}$,
$$|\tr_\S(f_1,f_2,f_3)|\lesssim     \|f_1\|_{2}|F_2| |F_3|^{-\frac12}  \log \Big( \e+\textstyle \frac{|F_3|}{|F_2|} \Big). $$ 
\end{theorem1s}
\begin{theorem2s}    
Let $\vec\alpha=(\alpha_1,\alpha_2,-\frac12) \in \mathcal S_3 $. For $f_1 \in L^{\frac{1}{\alpha_1}}(\R)$, $|f_2|\leq \cic{1}_{F_2} $,   and $F_3 \subset \R $ of finite measure, 
let $F_3'$ be the major subset of $F_3$ defined via \eqref{excset} by
$
F_3':= F_3(\vec \alpha, f_1,\cic{1}_{F_2})
$. Then for all $  |f_3|\leq \cic{1}_{F_3'}$, we have the estimate  
$$|\tr_\S(f_1,f_2,f_3)|\lesssim \textstyle \frac{1}{(1-\alpha_1)(1-\alpha_2)}   \|f_1\|_{\frac{1}{\alpha_1}}|F_2|^{\alpha_2} |F_3|^{-\frac12}  \Big(\max\big\{ (1-\alpha_1)^{-1},  \log \big(  \textstyle \frac{|F_3|}{|F_2|} \big)\big\}\Big)^{2(1-\alpha_2) }_*.$$
\end{theorem2s}

 \begin{theorem3s}    
Let $0\leq \alpha_1 < 1$, $\vec\alpha=(\alpha_1,\frac12,\frac12-\alpha_1)$. For $f_1 \in L^{\frac{1}{\alpha_1}}(\R) $, $f_2 \in L^2(\R)$,  and $F_3 \subset \R $ of finite measure,  let  
$F_3'$ be the  major subset of $F_3$ defined via  \eqref{excset} by  $
F_3':= F_3(\vec \alpha, f_1,f_2).
$ Then, for all $  |f_3|\leq \cic{1}_{F_3'}$,
$$ |\tr_\S(f_1,f_2,f_3 )| \lesssim \textstyle\frac{1}{1-\alpha_1}\Big( \textstyle\frac{1}{1-\alpha_1}\Big)_*^{2\alpha_1-1} \|f_1\|_{\frac{1}{\alpha_1}} \|f_2\|_{2} |F_3|^{\frac12-\alpha_1}. $$ 
\end{theorem3s}
\begin{remark}
By the dyadic H\"older scaling invariance of the family $\tr_\S$ pointed out in Remark \ref{sinvariance}, we may assume that $|F_3|\sim 1$ in our proofs. 
Also, linearity in $f_1$ of  assumptions and conclusions for Theorem 1', 2', and in both $f_1,f_2$ for Theorem 3' allows us to work, in these cases,  with $f_1,f_2$ of unit norm in the respective spaces.
We will work in the range $\alpha_1>3/4$ (say) in our proof of Theorem 3', since the bounds in the complementary region are well-known (and uniform in $\alpha_1$) from \eqref{BHTnoform}. Noting   that the estimate of Theorem 3' is stronger than the one of Theorem 2' when $|F_2| \geq |F_3|$,   we may conveniently restrict to $|F_2| \leq |F_3|\sim 1$ when proving Theorem 2'.
Finally, to unify notation, we write ${h_2}=f_2$ if $f_2$ is unrestricted and ${h_2}=\cic{1}_{F_2}$ if $f_2$ is restricted to $F_2$.
\end{remark}
The first two steps  of the proof are shared among the  three theorems.   Recalling from   \eqref{esets} the definition
$$
E^{\vec \alpha}_{f_1,{h_2},F_3}=\Big\{\M_{\frac{1}{\alpha_1}} f_1 \gtrsim \frac{1}{|F_3|^{\alpha_1}}\Big\} \cup \Big\{\M_{\frac{1}{\alpha_2}} {h_2} \gtrsim \frac{\|{h_2}\|_{1/\alpha_2}}{|F_3|^{\alpha_2}}\Big\},$$
 we decompose 
\begin{equation} \label{sbigl1l2} \S=\S^{\mathrm{bad}} \cup \S^1, \qquad 
\S^{\mathrm{bad}}=\big\{s \in \S: I_s \subset   E^\alpha_{f_1,{h_2},F_3} \big\}, \quad \S^{1}=\S\backslash\S^{\mathrm{bad}}.\end{equation}
Clearly $|\tr_\S| \leq |\tr_{\S^{\mathrm{bad}}}| + |\tr_{\S^{1}}|$. We handle the $\tr_{\S^{\mathrm{bad}}}$ term by a straightforward application of Proposition \ref{excsetprop}, which gives
\begin{equation} \label{theexcsetfine}
|\tr_{\S^{\mathrm{bad}}}(f_1,f_2,f_3 )| \lesssim \|f_1\|_{\frac{1}{\alpha_1}} \|{h_2}\|_{\frac{1}{\alpha_2}} |F_3|^{\alpha_3}, \qquad \forall |f_3| \leq \cic{1}_{F_3'}. \end{equation}
Note that \eqref{theexcsetfine} complies with the required estimate for $\tr_\S$ in all three cases.

We now fix $|f_3| \leq \cic{1}_{F_3}$, and perform the $f_3$-decomposition of $\S^1$ of Subsection \ref{f3dec} into collections $\S_k$ complying with  \eqref{sizebdSK} to \eqref{bmoSK}, and in addition inheriting from $\S^1$ the property
\begin{equation}
\label{notinexcfine}
I_s \not \subset   E^{\vec\alpha}_{f_1,{h_2},F_3} \qquad \forall s \in \S_k.
\end{equation} 
The remaining part of the proof, consisting in the estimation of the right hand side of $$
|\tr_{\S^{1}}(f_1,f_2,f_3 )| \leq \sum_{k \geq 0} |\tr_{\S_k}(f_1,f_2,f_3 )|
$$
 is specific to each theorem.

\begin{proof}[Conclusion of the proof of Theorem 1'] Recall that $f_2$ is restricted, thus ${h_2}=\cic{1}_{F_2}$, and that we are assuming $\|f_1\|_2=1$, $|F_3|\sim 1$.
A consequence of \eqref{infmax} and \eqref{notinexcfine}  is that 
\begin{align*}
\size_2(f_2; \S_k) & \lesssim \sup_{s \in \S^1} \inf_{x \in I_s} \M_1 f_2 (x)\leq  \sup_{s \in \S^1} \inf_{x \in I_s} \M_1 {{h_2}} (x)  \lesssim \min\big\{1,     |F_2|  \big\}  \lesssim  2^{-n_0}  |F_2|  ^{\frac12},
\end{align*}
where we  have set $n_0 = \textstyle\frac12\big|      \log |F_2|  \big|$. 
The first bound after the second almost inequality sign is actually due  to $|{h_2}|\leq 1$.
At this point, we apply Proposition \ref{propL2SK} to each $\S_k$, and bound
\begin{align*} &\quad  
 \sum_{k \geq 0} |\tr_{\S_k}(f_1,f_2,f_3 )|  \lesssim 2^{-n_0}  \|f_1\|_2 \|{h_2}\|_2 \sum_{k \geq 0}\min\{1, (k-n_0)2^{-(k-n_0)}\}  \lesssim   \min\Big\{ 1, \log \big(\textstyle\frac{1}{|F_2|}\big)\Big\}   |F_2|   
\end{align*}
which, combined with \eqref{theexcsetfine}, finishes the proof of Theorem 1'.
\end{proof}

  \begin{proof}[Conclusion of the proof of Theorem 3'] For this theorem, $f_2$ is unrestricted, so ${h_2}=f_2$, and we  are assuming $\|f_1\|_{\frac{1}{\alpha_1}}=\|f_2\|_2=1$.
  Again, from \eqref{notinexcfine} and  Lemma \ref{infmax}, we learn that 
\begin{align*}
\size_2(f_2; \S_k) \leq \size_2(f_2; \S^1) \lesssim \sup_{s \in \S^1} \inf_{x \in I_s} \M_2 f_2 (x) \lesssim   |F_3|^{-\frac12} \sim 1
\end{align*}
Also in view of \eqref{notinexcfine}, the assumption \eqref{smallmf} of Proposition \ref{proptouse} is satisfied. Applying the proposition to each $\S_k$, with $n_0=0$,  observing that $2\alpha_1-1$ is bounded away from zero in our range $\alpha_1 >3/4$, and recalling $u(t) \lesssim t(\log(\e+t))^2$ and the notation $t_*=t(\log(\e+t))^3$, we find
\begin{align*}
& \quad \sum_{k \geq 0} |\tr_{\S_k}(f_1,f_2,f_3 )| \lesssim  \|f_1\|_{\frac{1}{\alpha_1}}\|f_2\|_2    \sum_{k \geq 0}  \big(  u(Ck)^2 \log u(Ck)  \big)^{\alpha_1-\frac12}  2^{- 2(1-\alpha_1) k}\\ & \lesssim       \sum_{k \geq 0}  \big(  k \big)_*^{2\alpha_1-1}  2^{- 2(1-\alpha_1) k}    \lesssim \textstyle\frac{1}{1-\alpha_1}\Big( \textstyle\frac{1}{1-\alpha_1}\Big)_*^{2\alpha_1-1}   .
 \end{align*}
  The proof of Theorem 3' is finished by combining the last display with \eqref{theexcsetfine}.
 \end{proof}
 
  \begin{proof}[Conclusion of the proof of Theorem 2'] Here $f_2$ is restricted, thus ${h_2}=\cic{1}_{F_2}$, and we are assuming $\|f_1\|_{\frac{1}{\alpha_1}}=1$.
  Also,  we only need to treat the case $|F_2| \leq |F_3| \sim 1$. As in the previous proofs, we take advantage of \eqref{notinexcfine} and of Lemma \ref{infmax} to obtain the inequality
\begin{align*}
\size_2(f_2; \S_k)& \lesssim \sup_{s \in \S^1} \inf_{x \in I_s} \M_{1} f_2 (x) \leq  \sup_{s \in \S^1} \inf_{x \in I_s} \M_{1} {h_2} (x)   = \Big(\sup_{s \in \S^1} \inf_{x \in I_s} \M_{\frac{1}{\alpha_2}} \cic{1}_{F_2} (x)\Big)^{\frac{1}{\alpha_2}} \lesssim     2^{-n_0}  |F_2|^{\frac12} ,
\end{align*}
where we  have set $n_0 = \textstyle - \frac12 \log |F_2|\geq 0  $.
%\tr_{\S_{k}}(f_1,f_2,f_3) \lesssim 2^{- 2(1-\alpha_1) k} \big(  u(Ck)^2 \log u(Ck)  \big)^{\alpha_1-\frac12}   \|f_1\|_{\frac{1}{\alpha_1}}\|{h_2}\|_2 |F_3|^{\frac12-\alpha_1}  \begin{cases} 2^{k-n_0} & k \leq n_0 \\  \textstyle\frac{1}{2\alpha_1-1}  & k>n_0.\end{cases}$$
We make use of \eqref{notinexcfine} to verify the remaining  assumption \eqref{smallmf} of Proposition \ref{proptouse}, and apply the proposition to each $\S_k$, 
estimating 
\begin{align*} & \quad 
  \frac{\sum_{k \geq 0} |\tr_{\S_k}(f_1,f_2,f_3 )|}{|F_2|^{\alpha_2}}\lesssim   \sum_{k =0}^{n_0}  \big(  k \big)_*^{2\alpha_1-1} 2^{ (2\alpha_1-1) (k-n_0)} + {\textstyle\frac{1}{2\alpha_1-1}}  \sum_{k >n_0}  \big(  k \big)_*^{2\alpha_1-1} 2^{ -2(1-\alpha_1) (k-n_0)} \\ &  \lesssim  {\textstyle\frac{1}{1-\alpha_2}} \big(  n_0 \big)_*^{2\alpha_1-1}  + {\textstyle\frac{1}{(1-\alpha_1)(1-\alpha_2)}} \Big( \max\Big\{ \textstyle\frac{1}{1-\alpha_1}, n_0\Big\}\Big)^{2\alpha_1-1}_* .
\end{align*}
The bound of the last display, together with \eqref{theexcsetfine}, yields Theorem 2'.
 \end{proof}
  
 \subsection{Proof of Corollary \ref{rwtTH}}  Using symmetry, we can   work  with tuples $\vec \alpha \in \mathcal S_3$ and  treat the case $\alpha_2\geq \alpha_1$.  
%in this  regime, one has $1-\alpha_1 \geq \frac12$ and $M_{\vec\alpha}^{-1}\leq 4^5  (1-\alpha_2)^{-1}\lesssim \Ralpha^{-1}$ in estimate \eqref{propestrest} above. 
For tuples $\vec \alpha$ as such, specializing \eqref{propest} to  $|f_1|\leq \cic{1}_{F_1}$    yields the GRWT estimate   
 \begin{equation} \label{propestrest}
\big|\Lambda_{\vec\beta}(f_1,f_2,f_3)\big|  \lesssim    (1-\alpha_2)^{-1 } |F_1|^{\alpha_1}|F_2|^{\alpha_2} |F_3|^{-\frac12} \Big( \log \big(   \textstyle \frac{|F_3|}{|F_2|} \big) \Big)_*^{2(1-\alpha_2)} \qquad \forall \vec\alpha \in \mathcal S_3.
\end{equation}
 Fix an $\vec\alpha $ as above and a triple of sets $F_1,F_2,F_3$, and
let $0<\eps\leq 2(1-\alpha_2)$ to be chosen later. Let  $\vec a =(a_1,a_2,-\frac12) \in \mathcal S_3$ be the tuple with $a_2=1-\eps/2$:  given    $f_1,f_2$ restricted respectively to $F_1,F_2$, we may   apply \eqref{propestrest} with tuple $\vec a$ to bound
 \begin{equation} \label{propestrest2}
\big|\Lambda_{\vec\beta}(f_1,f_2,f_3)\big|   \lesssim   \eps^{-1} |F_1|^{a_1}|F_2|^{a_2} |F_3|^{-\frac12} \Big({\log} \big(   \textstyle \frac{|F_3|}{|F_2|} \big)\Big)_*^{\eps}\end{equation}
for all functions $f_3$ restricted to a major subset $F_3' \subset F_3$.     Switching the order of $F_1,F_2$ and replacing $F_3$ with  $F_3'$, we apply \eqref{propestrest}, again with tuple  $\vec a$, to  $\Lambda_{\sigma_{12}(\vec\beta)}$ instead, yielding
 \begin{align} \label{propestrest3}
\big|\Lambda_{\vec\beta}(f_1,f_2,f_3)\big|&=  \big|\Lambda_{\sigma_{12}(\vec\beta)}(f_2,f_1,f_3)\big|  \lesssim  \eps^{-1}  |F_2|^{a_1} |F_1|^{a_2} |F_3'|^{-\frac12} \Big({\log} \big(   \textstyle \frac{|F_3'|}{|F_1|} \big)\Big)_*^{\eps}\\ & \lesssim  \eps^{-1} |F_1|^{a_2}|F_2|^{a_1} |F_3|^{-\frac12} \Big({\log} \big(   \textstyle \frac{|F_3|}{|F_1|} \big)\Big)_*^{\eps}\nonumber 
\end{align}
for all functions $f_3$ restricted to a major subset $F_3'' \subset F_3'$, which (with different constant) is also a major subset of $F_3$.  Taking the $\vartheta$-geometric mean of \eqref{propestrest2} and \eqref{propestrest3},  for  $\frac12\leq \vartheta \leq 1 $ such that
$
	\alpha_1=\vartheta a_1 + (1-\vartheta) a_2, \alpha_2=\vartheta a_2 + (1-\vartheta) a_1,$  we obtain that for all $|f_3| \leq\cic{1}_{ F_3''}$  
 \begin{align*}  
|\Lambda_{\vec\beta}(f_1,f_2,f_3)| &\lesssim \eps^{-1} A^{\eps} |F_1|^{\alpha}|F_2|^{\alpha_\star} |F_3|^{-\frac12}, \qquad A:=  \Big({\log} \big(   \textstyle \frac{|F_3|}{\min{\{|F_1|,|F_2|\}}} \big) \Big)_*; \end{align*}
 estimate \eqref{propest2} then follows by taking $\eps= \min\{2(1-\alpha_2),(\log A)^{-1}\}$.

 \section{Interior estimates and Lorentz-Orlicz bounds for $\mathrm{BHT}_{\vec{b}}$} \label{sectintcor}
In this section, we list a number of corollaries following from our main theorems.    The proofs are given in the forthcoming Section \ref{sect2new}.
\subsection{Blowup rates of interior estimates} 
The endpoint bounds of our main results can be  equivalently reformulated as estimates, of the appropriate type,  for tuples $\vec{\alpha} \in \mathrm{int}\, \mathcal H$  with controlled dependence of the constants on  the  distances from $\vec{\alpha}$  to each side of $\partial \mathcal H$. We parametrize our tuples  by
\begin{align} \label{geometry} &
\vec \alpha (\varrho, \delta) =\textstyle \big(1-\varrho, \frac12+\varrho-\delta, -\frac12 +\ \delta\big), \quad q(\delta)= \big(\frac32-\delta\big)^{-1},
\\ &\textstyle \nonumber 0< \varrho \leq \frac14+2^{-5},\quad 0<\delta \leq \frac12+\varrho, \quad\min\{\varrho, \delta\} \leq 2^{-10}.
\end{align}  
The restrictions on $\varrho,\delta$ correspond  to approaching $\partial \mathcal H$ within the darker shaded region in Figure \ref{hexafig}. Estimates for other tuples near $\partial\mathcal H$ can be   recovered by symmetry considerations.
  
The first corollary is devoted to $L^{q(\delta),\infty}$ estimates. 
\begin{corollary} Let $\vec \alpha=\vec \alpha(\varrho, \delta), q=q(\delta)$ be as in \eqref{geometry}. \label{interiorzero} %\quad d_{\vec \alpha}:=   \delta_{\vec \alpha}^{-1}|\log(\e+\delta_{\vec \alpha})|^{3},\qquad \vec{\alpha}  \in \mathrm{int} \,} 
We   have the estimates
\begin{align}
& \label{introapproach}   \|\mathrm{BHT}_{\vec b}(f_1,f_2)\|_{q,\infty} \leq {C_{\vec\beta}} \max\big\{\varrho^{-1}, |\log \delta|\big\}|F_1|^{\alpha_1} |F_2|^{\alpha_2},   \qquad \forall\, |f_1| \leq \cic{1}_{F_1}, \, |f_2| \leq \cic{1}_{F_2};\\
&
\|\mathrm{BHT}_{\vec b}(f_1,f_2)\|_{q,\infty} \leq C_{\vec\beta}  \varrho^{-1}\big(\max\big\{\varrho^{-1},\delta^{-1}\}\big)_*^{1-2\varrho} \|f_1\|_{\frac{1}{\alpha_1} } |F_2|^{\alpha_2},   \quad \forall \, |f_2| \leq \cic{1}_{F_2}; \label{introapproachone}\\
&  \|\mathrm{BHT}_{\vec b}(f_1,f_2)\|_{q,\infty} \leq C_{\vec\beta}  \varrho^{-1}\max\Big\{1, \textstyle\Big(\frac{\varrho}{\delta}\Big)_*^{2\varrho}\Big\} |F_1|^{\alpha_1}\|f_2\|_{\frac{1}{\alpha_2} }   \quad \forall \, |f_1| \leq \cic{1}_{F_1}; \label{introapproachtwo} \\
& \|\mathrm{BHT}_{\vec b}(f_1,f_2)\|_{q,\infty} \leq C_{\vec\beta} \|f_1\|_{\frac{1}{\alpha_1} }\|f_2\|_{\frac{1}{\alpha_2} }  \varrho^{-1} \begin{cases}(\varrho^{-1})_*^{1-2\varrho}  & \varrho \leq \delta, \\ \delta^{-\frac{1}{q}} \textstyle\Big(\frac{\varrho}{\delta}\Big)^{2\varrho} & \varrho> \delta.\end{cases} \label{introapproachthree}
\end{align}
%\begin{equation} \label{introapproach}
%\big|\Lambda_{\vec\beta}(f_1,f_2,f_3)\big| \leq {C_{\vec\beta}}  \Calpha \prod_{j=1}^3|F_j|^{\alpha_j} \qquad \forall \vec{\alpha}  \in \mathrm{int} \,\mathcal H .\end{equation}
 \end{corollary}
% \begin{corollary} \label{interiorone}
%For tuples $ \vec{\alpha}  \in \mathrm{int} \,\mathcal H $ with $\alpha_3=\min \alpha_j$, $ \alpha_1>1/2$, 
%we have the estimate \begin{equation} \label{introapproachone}
%\big|\Lambda_{\vec\beta}(f_1,f_2,f_3)\big| \leq C_{\vec\beta} \textstyle   \frac{1}{ (1-\alpha_1)(2\alpha_1-1)   } \Big(\textstyle\frac{1}{\min\{1-\alpha_1, \delta_{\vec \alpha}\}}\Big)_*^{2\alpha_1-1}  \displaystyle   \|f_1\|_{\frac{1}{\alpha_1}} \Big( \prod_{j=2}^3|F_j|^{\alpha_j}\Big)
%\end{equation}
%for all $|f_2| \leq \cic{1}_{F_2}$, $|f_3| \leq \cic{1}_{F_3'} $,  $F_3'$ being a major subset of $F_3$ depending on $f_1,F_2,F_3$.
% \end{corollary}
%Note that the  estimates for the trilinear form of Corollary \ref{interiorzero} and \ref{interiorone} correspond respectively to bounds of the type
%$$
%\mathrm{BHT}_{\vec b}: L^{\frac{1}{\alpha_1},1}(\R) \times   L^{\frac{1}{\alpha_2},1}(\R) \to  L^{r,\infty}(\R), \qquad  \mathrm{BHT}_{\vec b}: L^{\frac{1}{\alpha_1}}(\R) \times   L^{\frac{1}{\alpha_2},1}(\R) \to  L^{r,\infty}(\R)
%$$
%where $r=(1-\alpha_3)^{-1}$. We now  detail how the restricted weak-type estimate of Corollary \ref{interiorzero} improves to weak-type, restricted strong-type and strong-type bounds for $\mathrm{BHT}_{\vec b}$. First, 
The second  deals with strong-type estimates. \begin{corollary} \label{weakstrong}  Let $\vec \alpha=\vec \alpha(\varrho,\delta), q=q$ be as in \eqref{geometry}. Then \begin{align}  
&
\|\mathrm{BHT}_{\vec b}(f_1,f_2)\|_{q } \leq {C_{\vec\beta}}   \frac{ \max\{\varrho^{-1}, |\log \delta|\}} {(\min\{\varrho, \delta\})^{\max\{1,\frac{1}{q }\}}}|F_1|^{\alpha_1} |F_2|^{\alpha_2},   \qquad \forall |f_1| \leq \cic{1}_{F_1},\,|f_2| \leq \cic{1}_{F_2};  \label{rstronga}
\\
& \|\mathrm{BHT}_{\vec b}(f_1,f_2)\|_{q} \leq  {C_{\vec\beta}}     \frac{ \max\{\varrho^{-1}, |\log \delta|\}} {(\min\{\varrho, \delta\})^{2\max\{1,\frac{1}{q}\}}}  \|f_1\|_{\frac{1}{\alpha_ 1}} \|f_2\|_{\frac{1}{\alpha_ 2}}.\label{stronga}
\end{align}
\end{corollary}

\begin{figure}[!ht] 
    \centering
  %   \setlength{\extrarowheight}{8pt}
     %\footnotesize
\begin{tabular}[b]{|l|cc|ccc|}   \hline \hline
 Est. @ $\vec \alpha(\varrho,\delta)$, $q=q(\delta)$, see \eqref{geometry}     &  \multicolumn{2}{c|}{   $\varrho =\textrm{dist}(\vec \alpha, \overline{CA}) \to 0$ }  &   \multicolumn{3}{c|}{   $\delta= \textrm{dist}(\vec \alpha, \overline{AB})\to 0$  } 
   \\ \hline & $\mathrm{BHT}$     & $\mathsf Q$\cite{DD2} & $\mathrm{BHT}$  & $\mathrm{BHT}$\cite{BG}  & $\mathsf Q$\cite{DD2}\\  \hline
 $L^{\frac{1}{\alpha_1},\mathrm{rst}} \times   L^{\frac{1}{\alpha_2},\mathrm{rst}} \to  L^{q,\infty}$, \eqref{introapproach} &    $\varrho^{-1}$   & $\varrho^{-1}$ &    $|\log \delta|$ &  $\delta^{-1}$ & 1\\ 

\hline
  {$L^{\frac{1}{\alpha_1}} \times   L^{\frac{1}{\alpha_2},\mathrm{rst}} \to  L^{q,\infty}$, \eqref{introapproachone}} & {$\varrho^{-1} (\varrho^{-1})_*^{1-2\varrho}$}  &
 {$\varrho^{-1}$} &   $(\delta^{-1})_*^{1-2\varrho}$ &  {N/A} &   {1} \\  
  
 \hline

  $L^{\frac{1}{\alpha_1},\mathrm{rst}} \times   L^{\frac{1}{\alpha_2}} \to  L^{r,\infty}$, \eqref{introapproachtwo}  & \multicolumn{2}{c|}{  $\varrho^{-1}$  } &   
  %\cellcolor{lightergray} 
   $\textstyle(\delta^{-1})_*^{2\varrho}$ & N/A & 1 \\
 \hline
 $L^{\frac{1}{\alpha_1}}\times L^{\frac{1}{\alpha_2}} \to L^{q,\infty}$, \eqref{introapproachthree}   &  $\varrho^{-1}  (\varrho^{-1})_*^{1-2\varrho} $  &$\varrho^{-1}$  &$\delta^{-\frac{1}{q}}|\log \delta|$ & $\delta^{-(1+\frac{1}{q})}$ & $\delta^{-\frac{1}{q}}$ \\

\hline 
 $L^{\frac{1}{\alpha_1},\mathrm{rst}}\times L^{\frac{1}{\alpha_2},\mathrm{rst}} \to L^{q}$, \eqref{rstronga}  & \multicolumn{2}{c|}{ $\varrho ^{-(1+\frac1q)}$ } & $\delta^{-\frac1q}|\log \delta|$ & $\delta^{-(1+\frac1q)}$ & $\delta^{-\frac1q}$
 \\ \hline
 
$L^{\frac{1}{\alpha_1}}\times L^{\frac{1}{\alpha_2},\mathrm{rst}}  \to L^{q }$  & \multicolumn{2}{c|}{ $\varrho^{-(1+\frac2q)}$ } & $\delta^{-\frac2q}|\log \delta|$ & $\delta^{-(1+\frac2q)}$ & $\delta^{-\frac2q}$
 \\ \hline
$ L^{\frac{1}{\alpha_1},\mathrm{rst}}\times L^{\frac{1}{\alpha_2}}  \to L^{q }$   & \multicolumn{2}{c|}{ $\varrho^{-(1+\frac2q)}$ } & $\delta^{-\frac2q}|\log \delta|$ & $\delta^{-(1+\frac2q)}$ & $\delta^{-\frac2q}$ \\ \hline 
 
$L^{\frac{1}{\alpha_1}}\times L^{\frac{1}{\alpha_2}} \to L^{q }$, \eqref{stronga}    & \multicolumn{2}{c|}{ $\varrho^{-(1+\frac2q)}$ } & $\delta^{-\frac2q}|\log \delta|$ & $\delta^{-(1+\frac2q)}$ & $\delta^{-\frac2q}$  \\ \hline

    \end{tabular}
    \captionlistentry[table]{Summary of interior estimates for $\mathrm{BHT}$ near $\partial \mathcal H$}
    \captionsetup{labelformat=table}
   \caption{ Blowup rates near the  $\overline{CA}$ and $\overline{AB}$ sides of $\mathcal H $ away from the endpoint $A$.
   We recall that $(t)_*=(1+t)(\log(\e+ t))^3$.  The first four rows come from Corollary \ref{interiorzero}. Rows five and eight    are obtained from Corollary \ref{weakstrong}. The sixth and seventh rows are  obtained by specializing the corresponding estimate of  line eight.  \label{tableside}}
  \end{figure}

In Table \ref{tableside}, we  summarize the blow-up rates of    eight possible types of interior estimates as the tuple $\vec \alpha$ approaches the segments $\overline{CA}$, $\overline{AB}$ away from the endpoint $A$  on the boundary of the shaded  hexagon  $\mathcal H$   in Figure \ref{hexafig}. We use  the results of Corollaries \ref{interiorzero} and \ref{weakstrong}, including for comparison the corresponding estimates following, with the same methods, from the endpoint results of   \cite{BG},  and \cite{DD2} for the Walsh case, mentioned in the introduction. 
% For the first group of columns, we assume that $\vec \alpha$ is approaching the  $\alpha_1=1$ segment  uniformly away from the corner $A$ and outside the Banach triangle $CDE$. For the second group of columns, we  let $\vec \alpha$ approach the side   $\alpha_3=-\frac{1}{2}$ uniformly away from the corners $A$, $B$ and $C$.
Note that the behavior of the estimates of Corollaries \ref{interiorzero} and  \ref{weakstrong} near the corners $A$ (where both parameters $\varrho, \delta$ can go to zero at the same time) and $C$ (where $\delta$ is away from zero) can be read  directly from the corollaries.

 \subsection{Lorentz-Orlicz space estimates}
 In the   same spirit of the article  \cite{CGMS},    we detail several Lorentz-Orlicz spaces bounds near   H\"older tuples $(p_1,p_2, \frac23)$. %Below,  $f^\star:(0,\infty) \to [0,\infty)$ denotes the decreasing rearrangement of $f: \R \to \mathbb C$.
%\begin{align*}
%&\|f\|_{L^{\frac23,\infty}(\log L)^{-\eps}(\R)}:= \sup_{t>0} \frac{t^{\frac32}f^\star (t)}{	\big(\log(\e+t)\big)^{\eps}}\\
%& \|f\|_{L^{p,q} (\log L)^{\eps}(\R)}:= \Big\|t^{\frac1p} \textstyle \big(\log\big(\e+\frac1t\big)\big)^{	\frac{\eps}{q}} f^\star(t) \Big\|_{L^{q}(\R; \frac{\d t}{t})},
%\end{align*}
%whereand $\eps \geq 0$.
 The first one   is obtained from Corollary \ref{rwtTH}, improving the logarithmic bumps in   \cite[Section 4.1]{CGMS} to doubly logarithmic ones.
\begin{corollary} \label{correstr}
Define the Lorentz-Orlicz quasinorms
\begin{align*}
&\|f\|_{L^{\frac23,\infty}(\log\log L)^{-1}(\R)}:= \sup_{t>0} \frac{t^{\frac32}f^\star (t)}{	\log\log(\e^\e+t)},  \\ & 
 \|f\|_{L^{p,\frac23} (\log\log L)^{\frac23}(\R)}:= \Big\|t^{\frac1p} \textstyle  \log\log\big(\e+\frac1t\big)f^\star(t) \Big\|_{L^{\frac23}(\R; \frac{\d t}{t})},\qquad 0<p <\infty.
\end{align*}
 Let $1<p_1,p_2<2, \frac{1}{p_1}+\frac{1}{p_2} = \frac32$. We have the estimate 
  \begin{equation*} 
\|\mathrm{BHT}_{\vec b}(f_1,f_2)\|_{L^{\frac23,\infty}(\log\log L)^{-1}(\R)} \leq {C_{\vec\beta}}     \prod_{j=1}^2 (p_j)'\|f_j\|_{L^{p_j,\frac23} (\log\log L)^\frac23 (\R)}. 
 \end{equation*} 
   \end{corollary} 
 In the second corollary, which stems from Theorem \ref{prLpL2}, the  first functional argument has no Lorentz-Orlicz bumps. This is also an improvement over \cite[Section 4.1]{CGMS}, which,  unlike the results therein, does not rely on extrapolation theory.  \begin{corollary} \label{corfull} For $\eps\geq 0$,  define the Lorentz-Orlicz quasinorms
\begin{align*}
&\|f\|_{L^{\frac23,\infty}(\log L)^{-\eps}(\R)}:= \sup_{t>0} \frac{t^{\frac32}f^\star (t)}{	\big(\log(\e+t)\big)^{\eps}}\\
& \|f\|_{L^{p,\frac23} (\log L)^{\eps}(\R)}:= \Big\|t^{\frac1p} \textstyle \big(\log\big(\e+\frac1t\big)\big)^{	\frac{3\eps}{2}} f^\star(t) \Big\|_{L^{\frac23}(\R; \frac{\d t}{t})}, \qquad 0<p<\infty.
\end{align*}
 Let $1<p_1,p_2<2, \frac{1}{p_1}+\frac{1}{p_2} = \frac32$. Then, for each $\eps>\frac{2}{(p_2)'}$,  there exists  $C_{p_1,\eps}>0$ such that  \begin{equation*} 
\|\mathrm{BHT}_{\vec b}(f_1,f_2)\|_{L^{\frac23,\infty}(\log L)^{-\eps}(\R)} \leq {C_{\vec\beta}} C_{p_1,\eps} \textstyle  \|f_1\|_{p_1} \|f_2\|_{L^{p_2,\frac23} (\log L)^\eps (\R)}.
 \end{equation*}    
 \end{corollary}

 Finally, Theorem \ref{propL1L2ag}   has as   corollaries   the following   bounds near $L^1 \times L^2$,   improving on the results of \cite[Section 4.2]{CGMS}. Notice that the $L^2$ component, unlike in \cite{CGMS}, has no Lorentz-Orlicz bumps. 
\begin{corollary} We have the bounds
\label{thL1L2lor}
\begin{align*}
&\mathrm{BHT}_{\vec b}:     L^{1,\frac23} (\log L)^{\frac23}(\R) \times L^2(\R)\to {L^{\frac23,\infty}(\log L)^{-1}(\R)}, \\
&
\mathrm{BHT}_{\vec b}:    L^{1} \log L \log\log L(\R) \times  L^2(\R)  \to {L^{\frac23,\infty}(\log L)^{-1}(\R)}.
\end{align*}
\end{corollary}

\section{Proofs of the Corollaries of Section \ref{sectintcor}}
\label{sect2new}

\subsection{Proof of Corollary \ref{interiorzero}} Recall that  $\vec \alpha =(1-\varrho, \frac12+\varrho-\delta, -\frac12+\delta)$.
In view of the equivalence  
 \begin{equation}
\label{rwtDEF} \|f\|_{p,\infty} \sim 
  \sup_{\substack{F \subset \R\\ |F| < \infty}} \inf_{\substack{F' \subset F \\ C|F'|\geq |F|}} |F|^{\frac{1}{p}-1}\big\l f, \exp(i \arg f(\cdot) )\cic{1}_{F'} \big\r,\end{equation}
  all the estimates of the corollary will be proved by showing that, for all $f_1$,$f_2$ as specified and for all $F_3 \subset \R$, there exists a major subset $F_3'$ of $F_3$ such that  
\begin{equation} \label{thetwist}
\sup_{|f_3| \leq \cic{1}_{F_3'}} |F_3|^{-\alpha_3} \big|\Lambda_{\vec\beta}(f_1,f_2,f_3)\big| 
\end{equation}
is bounded by the corresponding right hand side.

\subsubsection*{Interpolation preliminaries} Before the actual proofs, we derive three abstract off-diagonal weak-type interpolation lemmata which will be extensively relied upon.
Below, let $T$  be a sublinear operator on $\R$  mapping  Schwartz functions to locally integrable functions. We indicate by   $T^*$  the formal adjoint of $T$. What we  have in mind is the linear operator
$$
f_2 \mapsto T_{f_1}(f_2) := \mathrm{BHT}_{\vec b} (f_1,f_2),
$$
where $f_1$ is a fixed Schwartz function,
Observe that, by essential self-adjointness of $\mathrm{BHT}_{\vec b}$, we have that $(T_{f_1})^* (f_3) = \mathrm{BHT}_{\vec b'}(f_1,f_3)$, where $\vec b'$ shares the same degeneracy constant  $\Delta_{\vec \beta}$ associated to $\vec b$.

The first lemma is a variant of the usual off-diagonal Marcinkiewicz interpolation theorem, see e.g. \cite{LLY}.  We sketch the proof to emphasize the dependence of the constants.
\begin{lemma}
\label{interpolemma} Let there be given $$ \textstyle\frac12 \leq  p_0<p_1 \leq \infty, \quad \frac12\leq  q_0,q_1 \leq \infty , \, q_0,q_1 \neq 1,   \qquad \textstyle \frac{1}{p_0}-\frac{1}{p_1}= \frac{1}{q_0}-\frac{1}{q_1}= \mu>0.$$ Assume that  $T$ satisfies the bounds
\begin{equation} \label{interpolemmaass}
\|T g\|_{q_j,\infty} \leq K_j |G|^{\frac{1}{p_j}} \qquad \forall |g| \leq \cic{1}_G.
\end{equation}
Let $b_j=\min\{q_j,1\}$.
Then, for all $0<\vartheta<1$, 
$$
\|T g\|_{q(\vartheta),\infty} \leq C(\vartheta) \|g\|_{{p(\vartheta)}}, \qquad\textstyle \frac{1}{p(\vartheta)}= \frac{(1-\vartheta)}{p_0} + \frac{\vartheta}{p_1}, \quad \frac{1}{q(\vartheta)}= \frac{(1-\vartheta)}{q_0} + \frac{\vartheta}{q_1},
$$
where $$
C(\vartheta)= C \big(K_0\gamma_0 (\mu\vartheta)^{-\frac{1}{b_0}})^{1-\vartheta}   \big(K_1\gamma_1 (\mu(1-\vartheta))^{-\frac{1}{b_1}})^{\vartheta},  \qquad \gamma_j:= \frac{q_j}{|1-q_j|}, \,j=0,1. $$
\end{lemma}
\begin{proof} By eventually replacing $p_1$ with $p((\vartheta+1)/2)$, we can assume that $p_1<\infty$.   
We preliminarily observe that the assumptions \eqref{interpolemmaass}, coupled with the $\ell^q$ triangle inequality on $L^{q,\infty}(\R)$ \cite[Section 3]{CCM}  upgrade to the bounds
$$
\|T g\|_{{q_j,\infty}(\R)} \lesssim  \gamma_j K_j \|g\|_{{p_j,b_j}}, \qquad j=0,1, 
$$
where $\|\cdot\|_{\pi,\mu}$ denotes the Lorentz quasinorm on the Lorentz space ${L^{\pi,\mu}(\R)}$. We begin the actual proof;  by  a rearrangement argument, we can assume that $g=g^\star$. Let $\delta>0$ be a parameter to be chosen later. For $t>0$, we define  $g^t(x)= g(x) \cic{1}_{g(x)>g( 	\delta t)}$, $g_t = g-g^t$. Using the above display, we see that
$$
t^{\frac1q} (T g^{2t})^\star (t) \lesssim  \gamma_0  K_0 t^{-(\frac{1}{p_0}-\frac{1}{p(\vartheta)})} \|g^{2t}\|_{{p_0,b_0}}, \quad t^{\frac1q} (T g_{2t})^\star (t) \lesssim \gamma_1 K_1 t^{-(\frac{1}{p_1}-\frac{1}{p(\vartheta)})} \|g_{2t}\|_{p_1,b_1}.
$$
The lemma then follows from the estimates
\begin{align*}&
\sup_{t>0} t^{-(\frac{1}{p_0}-\frac{1}{p(\vartheta)})} \|g^{t}\|_{p_0,b_0 } \lesssim \delta^{\frac{1}{p_0}-\frac{1}{p(\vartheta)}}	  (\mu\vartheta)^{-\frac{1}{b_0}}\|g\|_p, \\ &\sup_{t>0} t^{-(\frac{1}{p_1}-\frac{1}{p(\vartheta)})}  \|g_{t}\|_{ {p_1,b_1} }\lesssim  \delta^{\frac{1}{p_1}-\frac{1}{p(\vartheta)} } (\mu(1-\vartheta))^{-\frac{1}{b_1}}\|g\|_p,
\end{align*}
which are obtained by means of H\"older's inequality, and finally by optimizing $\delta$. \end{proof} 

We will also use a version which does not upgrade the type of the estimate. Notice that the constant in \eqref{weaktypeinterp} does not   blow up as $	\vartheta \to 0$ or $1$. The proof is simple and we omit it.
\begin{lemma} \label{weakintlemma}
 Let there be given $$ \textstyle 	\frac12 < p_0<p_1 \leq \infty, \quad \frac12 \leq  q_0,q_1 \leq \infty ,    \qquad \textstyle \frac{1}{p_0}-\frac{1}{p_1}= \frac{1}{q_0}-\frac{1}{q_1}= \mu>0.$$ Assume that  $T$ satisfies the bounds
$$
\|T g\|_{{q_j,\infty}} \leq K_j \|g\|_{p_j}.$$ 
Then, for all $0<\vartheta<1$, 
\begin{equation} \label{weaktypeinterp}
\|T g\|_{q(\vartheta),\infty} \lesssim  (K_0)^{1-\vartheta} (K_1)^\vartheta\|g\|_{{p(\vartheta)}}, \qquad\textstyle \frac{1}{p(\vartheta)}= \frac{(1-\vartheta)}{p_0} + \frac{\vartheta}{p_1}, \quad \frac{1}{q(\vartheta)}= \frac{(1-\vartheta)}{q_0} + \frac{\vartheta}{q_1}.
\end{equation}
\end{lemma}
The next lemma exploits the equivalence \eqref{rwtDEF}
to interpolate between the $L^{p,\infty}$ estimates of $T$ and $T^*$.
\begin{lemma} \label{uptuple}
Let $0<\alpha<1$, $-1< \beta<0$,  and assume that for all $G \subset \R$ of finite measure 
$$
\|T g\|_{\frac{1}{1-\beta},\infty}, \, \|T^* g\|_{\frac{1}{1-\beta},\infty}\leq K |G|^{\alpha} \qquad \forall |g| \leq \cic{1}_G.
$$  
Then, for all $0<t\leq \alpha$,  and for all $G \subset \R$ of finite measure, 
  we have the estimate
$$
\|T g\|_{\frac{1}{1-\beta -t},\infty}, \, \|T^* g\|_{\frac{1}{1-\beta -t},\infty} \lesssim  \frac{K}{|\beta|}  |G|^{\alpha-t} \qquad \forall |g| \leq \cic{1}_G.
$$ 
\end{lemma}
\begin{proof} By symmetry, it suffices to carry the proof for $T$.  Fixing $G \subset \R$ of finite measure, $|g| \leq \cic{1}_G$,  and using  \eqref{rwtDEF}, it suffices to show that for all $0<t\leq\alpha$, $F \subset \R$ with $|F|<\infty$, there exists a set $F'\subset F$ with $|F| \leq C|F'|$ such that, for all $|f| \leq \cic{1}_{F'}$
\begin{equation} \label{lemmasublin1}
|\l T g, f \r| \lesssim  K  |\beta|^{-1}  |G|^{\alpha}|F|^{\beta} \textstyle \big( \frac{|F|}{|G|}\big)^{t},  
\end{equation}
Since inequality \eqref{lemmasublin1} holds by assumption for $t=0$, with no need for $|\beta|^{-1}$,  there is nothing to prove if $|G| \leq  |F| $. Assume $|G| >|F|$ and let $n=\lceil \log|G|-\log |F| \rceil.$ We apply our assumption for $T^*$ instead, so that the roles of $F$ and $G$ are reversed, and, via \eqref{rwtDEF}, we obtain the existence of a set $H^{(1)} \subset G:=G^{(0)}$, with $|H^{(1)}| \geq C^{-1} |G|$, such that
$$
|\l  g \cic{1}_{H^{(1)}}, T^*f \r| \lesssim K |F|^{\alpha}|G|^{\beta} \qquad \forall |f | \leq \cic{1}_{F}.
$$
Iterating, we define a sequence  $G^{(k+1)}= G^{(k)}\backslash H^{(k+1)}$, stopping when $|G^{(\bar k)}| \leq |F|$.  
 Note that $|G^{(k)}| \leq \e^{-ck} |G|$, so that $\bar k \leq  n/c $, where $c=\log{C} - \log(C-1)$. Finally, the assumption provides a   set $F'\subset F$ with $|F'|\geq C^{-1}|F|$ such that 
 $$
 |\l  T(g\cic{1}_{G^{(\bar k)} }), f\cic{1}_{F'} \r| \lesssim K |G^{(\bar k)}|^\alpha|F|^{ \beta}  \leq K|F|^{\alpha +\beta} \qquad \forall |g| \leq 	\cic{1}_G.
 $$
 Observing that $G = H^{(1)} \cup \ldots H^{(\bar k-1)} \cup G^{(\bar k)} $ leads   to the estimate\begin{align*}
|\l  Tg, f\cic{1}_{F'} \r|& \leq  |\l  T(g\cic{1}_{G^{(\bar k)} }), f\cic{1}_{F'} \r|+\sum_{k=1}^{\bar k} |\l  g \cic{1}_{H^{(k)}}, T^*(f\cic{1}_{F'}) \r|    \lesssim    K  |F|^{\alpha+ \beta}+ K |F|^{\alpha } |G|^{\beta} \sum_{k=1}^{n/c} \e^{-\beta c k} \\ &  \lesssim  K(1+ c^{-1}|\beta|^{-1})|F|^{\alpha+\beta}  \lesssim K |\beta|^{-1} |G|^{\alpha-t}|F|^{\beta+t}.\end{align*}
which completes the proof of \eqref{lemmasublin1}, and in turn, of the lemma.
\end{proof}

 \begin{proof}[Proof of \eqref{introapproach}]  The proof is split into two cases.
%By symmetry considerations, we can reduce to the case $|F_1|\leq |F_2|$, $\alpha_1 \geq \alpha_2 \geq \alpha_3 >-1/2$, so that, writing  $\varrho=\Ralpha , \delta= \delta_{\vec \alpha}  $,  $\vec\alpha=(1-\varrho, \frac12+\varrho-\delta, -\frac12+\delta)$.
%Note that, in this region of exponents, the case $ |F_3|\leq |F_1|$ follows from the (uniform) strong-type bounds \eqref{BHTnoform}   on the line segment $(1/4,a,3/4-a)$,  $0\leq a \leq 3/4$. We are left  with the range $|F_1|\leq |F_3| $, whose proof is split into three cases.  By H\"older  scaling invariance of the trilinear forms, we further  assume $|F_3|=1$ below to simplify notation. 

 \emph{Case   $0< \delta  \leq \varrho/2$.}  The bound   \eqref{thetwist} $\lesssim |F_1|^{\alpha_1}|F_2|^{\alpha_2} |F_3|^{\alpha_3}$ when $ |F_3|\leq |F_1|$ follows from the (uniform) strong-type bounds \eqref{BHTnoform}   on the line segment $(1/4,a,3/4-a)$,  $0\leq a \leq 3/4$. When $|F_1|\leq |F_3| $ instead,
we apply \eqref{propest2} with tuple $\vec{a}:=(1-\varrho+\delta ,\alpha_2,-\frac12) \in \mathcal S_3$, which has $1-a_1 \geq \varrho /2,$ $1-a_2 \gtrsim 1$,  yielding a major subset $F_3'$ of $F_3$ such that, for all suitably restricted $f_j$,
$$\big|\Lambda_{\vec\beta}(f_1,f_2,f_3)\big| \lesssim  \left[\textstyle\big(\frac{|F_1|}{|F_3|}\big)^{\delta } \max\Big\{ \varrho  ^{-1},  \log \log \Big( \e^\e+ \textstyle\frac{|F_3|}{|F_1|}\Big) \Big\}    \right]  |F_1|^{\alpha_1}|F_2|^{\alpha_2} |F_3|^{\alpha_3}.$$
By optimizing $t \mapsto t^{\delta } \log \log(\e^\e + t^{-1}) $ on $t \in [0,1]$, one sees that the square bracketed term is bounded by $C\max\{\varrho^{-1},   | \log  \delta   | \}  $, as claimed in \eqref{introapproach}, concluding the proof of this case.

\emph{Case  $ \delta  >\varrho/2$.}  In this range  $\max\{\varrho^{-1},   | \log  \delta   | \} =\varrho^{-1} $. We fix $F_1$, $|f_1| \leq \cic{1}_{F_1}$ and define the linear operator $g\mapsto T_{f_1} (g):= \mathrm{BHT}_{\vec{b}}(f_1,g)$.
The previous case and essential selfadjointness yield
$$
\|T_{f_1} (g)\|_{q(\varrho/2),\infty}, \|(T_{f_1})^* (g)\|_{q(\varrho/2),\infty} \lesssim \varrho^{-1} |F_1|^{\alpha_1} |G|^{\frac12+\frac\varrho 2} \qquad \forall |g| \leq \cic{1}_{G},
$$
so that an application of  Lemma \ref{uptuple} with $t=\delta-\varrho/2$ entails
$$
\sup_{G, |g| \leq \cic{1}_G} |G|^{-\alpha_2} \|T_{f_1} (g)\|_{q(\delta),\infty} \lesssim \sup_{G, |g| \leq \cic{1}_G} |G|^{-(\frac12+\frac\varrho2)} \|T_{f_1} (g)\|_{q(\varrho/2),\infty} \lesssim   \varrho^{-1} |F_1|^{\alpha_1}, 
$$ which is what is required in \eqref{introapproach}.
\end{proof}
\begin{proof}[Proof of \eqref{introapproachone}]We fix $f_1 \in L^{\frac{1}{\alpha_1}}(\R)$ of  unit norm. The proof is again split into two cases.

\emph{ Case $\varrho \geq \delta$.} We show that for any given $F_2,F_3$ there is a major subset $F'_3$ of $F_3$ such that
\begin{equation} \label{introapproachonepf1}
\big|\Lambda_{\vec\beta}(f_1,f_2,f_3)\big| \lesssim  \varrho^{-1}\big(\max\big\{\varrho^{-1},\delta^{-1}\}\big)_*^{1-2\varrho}|F_2|^{\alpha_2} |F_3|^{\alpha_3} \qquad \forall |f_2| \leq \cic{1}_{F_2}, \,  |f_3| \leq \cic{1}_{F_3'},
\end{equation}
which implies \eqref{introapproachone} via the usual equivalence.
Assume first $|F_2| \geq |F_3|$. We apply Theorem \ref{thelinfversion}, to the obvious choice of  $f_1$, and to $f_2$ restricted to $F_2$, obtaining a major subset $F_3'$ of $F_3$ such that     
$$
\big|\Lambda_{\vec\beta}(f_1,f_2 ,f_3)\big| \lesssim   \textstyle \varrho^{-1}\big( \varrho^{-1} \big)_*^{1-2\varrho}  \|f_2\|_2|F_3|^{\frac12-\alpha_1} \leq \varrho^{-1}\big( \varrho^{-1} \big)_*^{1-2\varrho}  |F_2|^{\alpha_2} |F_3|^{\alpha_3}\Big( \frac{|F_3|}{|F_2|}\Big)^{\alpha_2-\frac12}, $$
which,  since $\alpha_2 \geq \frac12$, is stronger than \eqref{introapproachonepf1}. 
If instead $|F_3| \geq |F_2|$,  we apply estimate \eqref{propest} with tuple $\vec a=(\alpha_1, \alpha_2 + \delta ,-1/2) \in \mathcal S_3$, obtaining, for any pair of suitably restricted functions $f_2,f_3$,
$$
\big|\Lambda_{\vec\beta}(f_1,f_2,f_3)\big| \lesssim \textstyle\varrho^{-1} |F_2|^{\alpha_2}|F_3|^{\alpha_3}\left[
 \left(\max\Big\{\varrho^{-1},  \log  \big(  \textstyle\frac{|F_3|}{|F_2|}\big) \Big\}\right)_*^{1-2\varrho}\left(\frac{|F_2|}{|F_3|}\right)^{\delta} \right]; 
  $$
 noting that the term in square brackets is  $\lesssim (\max\{\varrho^{-1},\delta^{-1}\})_*^{1-2\varrho}$ leads to    \eqref{introapproachonepf1}.
 
 \emph{Case $\delta >\varrho$.} Analogously to what we did in the previous proof, we define the linear operator $g\mapsto T_{f_1} (g):= \mathrm{BHT}_{\vec{b}}(f_1,g)$.
Again by Theorem \ref{thelinfversion} (note that $\delta=\varrho$ corresponds to $\alpha_2=\frac12$)
$$
\|T_{f_1} (g)\|_{q(\varrho),\infty}, \|(T_{f_1})^* (g)\|_{q(\varrho),\infty} \lesssim \varrho^{-1}\big( \varrho^{-1} \big)_*^{1-2\varrho}  \|f_1\|_{\frac{1}{\alpha_1}}|G|^{\frac12} \qquad \forall |g| \leq \cic{1}_{G},
$$
and  Lemma \ref{uptuple} with $t=\delta-\varrho$ entails
$$
\sup_{G, |g| \leq \cic{1}_G} |G|^{-\alpha_2} \|T_{f_1} (g)\|_{q(\delta)} \lesssim \sup_{G, |g| \leq \cic{1}_G} |G|^{-\frac12} \|T_{f_1} (g)\|_{q(\varrho)} \lesssim   \varrho^{-1}\big( \varrho^{-1} \big)_*^{1-2\varrho}  \|f_1\|_{\frac{1}{\alpha_1}}, 
$$ which is what we had to prove in \eqref{introapproachone}.
\end{proof}
\begin{proof}[Proof of \eqref{introapproachtwo}] We separate two cases.

\emph{Case $\varrho\geq \delta$.}
We fix $f_2 \in L^{{1}/\alpha_2}(\R)$ of  unit norm and prove that for any given $F_1,F_3$ there is a major subset $F'_3$ of $F_3$ such that
\begin{equation} \label{introapproachtwopf1}
\big|\Lambda_{\vec\beta}(f_1,f_2,f_3)\big| \lesssim \textstyle\varrho^{-1} \Big(\frac{\varrho}{\delta}\Big)^{2\varrho}_* |F_1|^{\alpha_1} |F_3|^{\alpha_3} \qquad \forall |f_1| \leq \cic{1}_{F_1}, \,  |f_3| \leq \cic{1}_{F_3'}.
\end{equation}
As in the proof of \eqref{introapproach}, the case $ |F_3|\leq |F_1|$ follows from the known  strong-type bounds \eqref{BHTnoform}. When $|F_1|\leq |F_3| $, we apply estimate \eqref{propest} with tuple $\vec a=( \alpha_2, \alpha_1 +\delta   ,-1/2) \in \mathcal S_3$, and switching the roles of $f_1,f_2$. We obtain, for suitably restricted   $f_1,f_3$,
$$
\big|\Lambda_{\vec\beta}(f_1,f_2,f_3)\big| \lesssim \textstyle \varrho^{-1} |F_1|^{\alpha_1}|F_3|^{\alpha_3}\left[
 \left(   \log  \big(  \textstyle\frac{|F_3|}{|F_1|}\big)  \right)_*^{2(\varrho-\delta)}\left(\frac{|F_1|}{|F_3|}\right)^{\delta} \right]; 
  $$
 estimating the term in square brackets by $C (\frac{\varrho}{\delta})^{2\varrho}_*$ leads to    \eqref{introapproachtwopf1}, and finishes the proof.
 
 \emph{Case $\varrho\leq \delta$.} For this case, we fix $F_1$ and $|f_1|\leq \cic{1}_{F_1}$, and interpolate through Lemma \ref{weakintlemma} the estimates $$
\|T_{f_1} (g)\|_{q(\varrho),\infty} \lesssim  \varrho^{-1} |F|^{\alpha_1} \|g\|_2, \qquad  \|T_{f_1} (g)\|_{q(\varrho+1/2),\infty} \lesssim \varrho^{-1} |F|^{\alpha_1} \|g\|_\infty
$$
for the linear operator 
$g\mapsto T_{f_1} (g):= \mathrm{BHT}_{\vec{b}}(f_1,g)$, the first of which is obtained in the previous case, while the second can be read from \eqref{introapproach} when $\alpha_2=0$.
\end{proof}
\begin{proof}[Proof of \eqref{introapproachthree}] The proof is split into two cases, both relying on interpolation.

 \emph{Case $\varrho\leq \delta$.} For this case, we fix $f_1 \in L^{1/\alpha_1}(\R)$ and   interpolate through Lemma \ref{weakintlemma} the estimates $$
\|T_{f_1} (g)\|_{q(\varrho),\infty} \lesssim   \varrho^{-1} (\varrho^{-1})_*^{1-2\varrho} \|f_1\|_{\frac{1}{\alpha_1}} \|g\|_2, \;\, \|T_{f_1} (g)\|_{q(\varrho+1/2),\infty} \lesssim \varrho^{-1} (\varrho^{-1})_*^{1-2\varrho}  \|f_1\|_{\frac{1}{\alpha_1}} \|g\|_\infty
$$
for $ T_{f_1} (g):= \mathrm{BHT}_{\vec{b}}(f_1,g)$.
The first estimate above is exactly a reformulation of Theorem \ref{thelinfversion}, while the second one can be read from the case $\alpha_2=0$ of \eqref{introapproachone}.

 \emph{Case $\varrho>\delta$.} We fix $f_2 \in L^{1/\alpha_2}(\R)$. We learn from \eqref{introapproachtwo} that the linear operator $ T_{f_2} (g):= \mathrm{BHT}_{\vec{b}}(g,f_2)$ satisfies the estimates 
 $$\textstyle
\|T_{f_2} (g)\|_{q(\frac\delta 2),\infty} \lesssim \varrho^{-1} \Big(\frac{\varrho}{\delta}\Big)^{2\varrho}_*  \|f_2\|_{\frac{1}{\alpha_2}} |F_1|^{\alpha_1 +\frac\delta2}, \quad \|T_{f_2} (g)\|_{q(\frac{3\delta}{2}),\infty} \lesssim \varrho^{-1} \Big(\frac{\varrho}{\delta}\Big)^{2\varrho}_*  \|f_2\|_{\frac{1}{\alpha_2}} |F_1|^{\alpha_1 -\frac\delta2},$$
respectively corresponding to \eqref{introapproachtwo} for tuples $(1-\rho\pm \delta/2, \alpha_2, -\frac12 +\delta\mp \delta/2)$. We now  use Lemma \ref{interpolemma} on  $T$, with $\vartheta=\frac12$, $p_0=\frac{1}{\alpha_1+\delta/2}$, $p_1=\frac{1}{\alpha_1-\delta/2}$, $q_0=q(3\delta/2), $ $q_1=q( \delta/2)$, and observe that $\mu=\delta$, $\gamma_0,\gamma_1$ therein are uniformly bounded, and $ 1/(2b_0)+ 1/(2b_1)=1/q(\delta) $, which entails the estimate
 $$\textstyle
\|T_{f_2} (f_1)\|_{q(\delta),\infty} \lesssim \delta^{-\frac{1}{q(\delta)}}  \varrho^{-1} \Big(\frac{\varrho}{\delta}\Big)^{2\varrho}_*  \|f_1\|_{\frac{1}{\alpha_1}}   \|f_2\|_{\frac{1}{\alpha_2}} .$$
 This completes the proof.  
\end{proof}
 
 \subsection{Proof of Corollary \ref{weakstrong}} 
 In the proof, we will need the following more precise form of the interpolation result \cite[Lemma 3.11]{MTT} (see also \cite[Theorem 3.8]{ThWp}).
\begin{lemma}
\label{interp} 
Let $T$ be a bisublinear operator on $\R$ mapping pairs of Schwartz functions into measurable functions. Let $\vec \alpha=(\alpha_1,\alpha_2,\alpha_3) \in \mathrm{int} \,\mathcal H$ be  a H\"older tuple with $\alpha_3 =\min\{\alpha_j\} \leq 0$. Suppose that for a given $\eps>0$, $O_\eps(\vec \alpha):=\{\vec a : a_1+a_2+a_3=1,\,\max_{j}|\alpha_j-a_j| \leq \eps \} \subset \mathrm{int}\, \mathcal H$, and  there exists $K>0$ such that the estimate
\begin{equation}
\label{estunif}
\|T(f_1,f_2)\|_{\frac{1}{1-a_3},\infty} \leq K  \prod_{j=1}^2 |F_j|^{a_j} \qquad \forall\, |f_1| \leq \cic{1}_{F_1}, \, |f_2| \leq \cic{1}_{F_2}
\end{equation}
holds for all  tuples $\vec a \in O_\eps(\vec \alpha)$. Then,  \begin{align*}
& \|T(f_1,f_2)\|_{\frac{1}{1-\alpha_3}} \lesssim \eps^{-(1-\alpha_3)} K |F_1|^{\alpha_1} |F_2|^{\alpha_2},   \qquad \forall |f_1| \leq \cic{1}_{F_1}, \; |f_2| \leq \cic{1}_{F_2},  \\
& \|T(f_1,f_2)\|_{\frac{1}{1-\alpha_3}}  \lesssim  \eps^{-2(1-\alpha_3)} K \|f_1\|_{\frac{1}{\alpha_1}} \|f_2\|_{\frac{1}{\alpha_2}}.
\end{align*}
\end{lemma}
\begin{proof}
We sketch the proof of the second   estimate: the proof of  the first  estimate is implicit in the argument for the second one. Let $q=(1-\alpha_3)^{-1}$.
By rearrangement, we can assume that $f_1,f_2$ are nonnegative, supported on $(0,\infty)$ and nondecreasing. For $j=1,2$, $k_j \in \mathbb Z$, define $f_j^{k_j}= f_j \cic{1}_{[2^{k_j},2^{k_j+1})}$. Arguing like in \cite{MTT}, we exploit uniformity in $O_\eps(\vec \alpha)$ of \eqref{estunif} to obtain the estimate 
$$
 \|T(f_1^{k_1},f_2^{k_2})\|_{ q}^{q} \lesssim \eps^{-1}  K^{r}  2^{-\eps|k_1-k_2|} \big(\textstyle\prod_{j=1}^2 f_j(2^{k_j})2^{\alpha_j k_j}\big)^{q},
$$
the implicit constant being absolute. Since 
$$
\|T(f_1 ,f_2 )\|_{q}^{q} \leq \sum_{k_1,k_2 \in \mathbb Z}\|T(f_1^{k_1},f_2^{k_2})\|_{ q}^{q}$$
the  second estimate of the lemma follows by bounding the resulting sum  as in \cite{MTT}.
\end{proof}

 We first  prove Corollary \ref{weakstrong} for tuples     $\vec \alpha$ outside the reflexive Banach triangle, that is, with $q(\delta) \leq 1$. For such a tuple, referring to \eqref{geometry}, set $\eps=\min\{\varrho,\delta\}/2$. We read from  \eqref{introapproach} that condition  \eqref{estunif} of the interpolation Lemma \ref{interp} holds for all tuples  in $O_\eps(\vec \alpha)$ with constant $K \lesssim \max\{\varrho^{-1}, |\log \delta\}$, so that the estimates of  Corollary \ref{weakstrong} in this range  follow by  a straightforward application of the lemma.

We now deal with the case of $\vec \alpha$ inside the reflexive Banach triangle:  by symmetry and duality, we can restrict to proving  the case $\alpha_2\geq \alpha_3$. Note that, according to \eqref{geometry}, $\varrho \leq 2^{-5}$.
 We can then  write 
 $\vec \alpha$   as a convex combination of the tuple $\vec \omega=(1/2,1/6,1/3)$ and of a tuple $\vec \gamma $ with $\gamma_3 =\min \gamma_j =0$,  $1-\gamma_1 \geq \varrho/16$.  Therefore, estimates \eqref{rstronga}, \eqref{stronga} for $\vec \alpha$ follow from complex interpolation  of   the corresponding estimates for $\vec \gamma$, established in the previous step,  with   those, well-known, for $\vec \omega$. This concludes the proof of Corollary \ref{weakstrong}.

\subsection{Proofs of Corollaries  \ref{correstr} to  \ref{thL1L2lor}}
To prove Corollary \ref{correstr}, we observe that estimate \eqref{propest2} of Corollary \ref{rwtTH} can be rewritten as
$$
\frac{t^{\frac32}}{\log\log(\e^\e + t)} \big(\mathrm{BHT}_{\vec b}(\cic{1}_{F_1}, \cic{1}_{F_2}) \big)^\star (t) \lesssim \prod_{j=1}^2(p_j)'|F_j|^{\frac{1}{p_j}} \log \log \big( \e^\e + |F_j|^{-1}\big) \qquad \forall t >0.
$$
Then, the corollary follows from the above display by arguing along the   lines of \cite[Section 4.1]{CGMS}; we omit the details.

To obtain Corollary \ref{corfull} from Theorem \ref{prLpL2}, for each $ \eps = \frac{2}{(p_2)'}+\zeta>\frac{2}{(p_2)'}$ we  take $\alpha_1=\frac {1}{p_1},\alpha_2=\frac{1}{p_2}$,  in \eqref{propest} and estimate, using that   $1-\alpha_2=\frac{1}{(p_2)'}$,
\begin{align*}\big|\l\mathrm{BHT}_{\vec b}(f_1,f_2),f_3\r\big|=\big|\Lambda_{\vec \beta}(f_1,f_2,f_3)\big| & \lesssim \    C_{p_1}\|f_1\|_{p_1}|F_2|^{\frac{1}{p_2}} |F_3|^{-\frac12} \Big(\log  \big(     \textstyle \frac{|F_3|}{|F_2|} \big)\Big)_*^{\frac{2}{(p_2)'} } \\ &\lesssim \textstyle  \zeta^{-3}  C_{p_1}  \|f_1\|_{p_1}|F_2|^{\frac{1}{p_2}} |F_3|^{-\frac12} \Big( \log  \Big(  \e+   \textstyle \frac{|F_3|}{|F_2|} \Big)\Big)^{\eps };
\end{align*}
  we can take $C_{p_1}:= (p_1')^3(p_2)'$. Setting $C_{p_1,\eps}:=\zeta^{-3}C_{p_1}$,  this can be rearranged into
\begin{equation} \label{propestref}
\| \mathrm{BHT}_{\vec b}(f_1,f_2) \|_{L^{\frac23,\infty}(\log L)^{-\eps}(\R)} \lesssim \textstyle C_{p_1,\eps} \textstyle   \|f_1\|_{p_1} |F_2|^{\frac{1}{p_2}} \textstyle  \left( 
\log\big( \e + \frac{1}{|F_2|}\big)\right)^{\eps} \quad \forall \, |f_2| \leq	\cic{1}_{F_2}.
\end{equation}
Corollary \ref{corfull}    follows from \eqref{propestref} by recalling   that   (see \cite{CCM}) 
$$
\|f\|_{{L^{\frac23,\infty}(\log L)^{-\eps}(\R)}}\sim \inf\Big\{ \|\gamma_k\|_{\ell^{\frac23}}: f= \sum_{k} \gamma_k f_k, \;\|f_k\|_{{L^{\frac23,\infty}(\log L)^{-\eps}(\R)}} \leq 1\Big\},
$$
and subsequently performing the elementary  procedure described in \cite[Section 2]{DD2}. 

For the  details of the derivation of  Corollary \ref{thL1L2lor} from Theorem \ref{propL1L2ag},  we refer to \cite[Section 2]{DD2}.  We only mention that an intermediate\ step towards the second estimate is the strenghtening of Theorem \ref{propL1L2ag} $$ \|\mathrm{BHT}_{\vec b}(f_1,f_2)\|_{{L^{\frac23,\infty}(\log L)^{-1}(\R)}}
 \leq C_{\vec \beta}\|f_1\|_1 \|f_2\|_2\log \big(  \textstyle \frac{\|f_1\|_\infty}{\|f_1\|_1}\big).
 $$
The above inequality   follows from Theorem \ref{propL1L2ag} via,  for instance, the theory of \cite{CCM} (see also \cite[Theorem 3.3]{CGMS}).

\subsection*{Acknowledgements} The first author wants to express his gratitude to  the Hausdorff Center for Mathematics at the University of Bonn, Germany, and in particular to Diogo Oliveira e Silva,  for the kind hospitality during his May 2013 research visit, when this article was initiated.

\bibliography{endpointBHT-fourier}{}
\bibliographystyle{amsplain}
 \end{document}